\documentclass[11pt, reqno]{amsart} %reqno forces equation numbering on the right
 \usepackage[margin=1in]{geometry}

\usepackage{wasysym}
\usepackage{bm}
\usepackage{amssymb, amsmath, amscd, amsthm, color, epsfig,url, graphicx}
\usepackage[font={footnotesize}]{caption}

\usepackage[all]{xy}          % for xy-pic pictures
\xyoption{dvips}              % for xy-pic pictures

\usepackage{enumitem}

\makeatletter
\renewcommand\l@subsection{\@tocline{2}{0pt}{2pc}{5pc}{}}
\makeatother

\newcommand{\R}{{\mathbb R}}

\newcommand{\Z}{{\mathbb Z}}
\newcommand{\Q}{{\mathbb Q}}

\newcommand{\Map}{\operatorname{Map}}
\newcommand{\Hom}{\operatorname{Hom}}
\newcommand{\Emb}{\operatorname{Emb}}

\newcommand{\Link}{\operatorname{Link}}

\newcommand{\Ho}{\operatorname{H}}
\newcommand{\Ch}{\operatorname{C}}
\newcommand{\Ba}{\operatorname{B}}
\newcommand{\ChdR}{\operatorname{C}_{dR}}

\newcommand{\Conf}{\operatorname{Conf}}

\newcommand{\Chen}{\operatorname{Ch}}
\newcommand{\ZZ}{\operatorname{Z}}

\newcommand{\calA}{{\mathcal{A}}}

\newcommand{\D}{{\mathcal{D}}}

\newcommand{\Lk}{{\mathcal{L}}}

\newcommand{\LD}{{\mathcal{LD}}}
\newcommand{\LDf}{{\mathcal{LD}_{\mathrm{f}}}}

\newcommand{\HD}{{\mathcal{HD}}}

\newcommand{\BCD}{{\mathcal{BCD}}}

\newcommand{\LV}{{\mathcal{LV}}}

\newcommand{\Xdot}{{X^{\bullet}}}

\newcommand{\ev}{\operatorname{ev}}
\newcommand{\pr}{\operatorname{pr}}

\newcommand{\x}{\times}

\newcommand{\Tot}{\operatorname{Tot}}

\newtheorem*{rep@theorem}{\rep@title}
\newcommand{\newreptheorem}[2]{%
	\newenvironment{rep#1}[1]{%
		\def\rep@title{#2 \ref{##1}}%
		\begin{rep@theorem}}%
		{\end{rep@theorem}}}

\theoremstyle{plain}
\newtheorem{thm}{Theorem}[section]
\newreptheorem{thm}{Theorem}

\newreptheorem{thmp}{Theorem} 
\newtheorem{prop}[thm]{Proposition}

\newtheorem{cor}[thm]{Corollary}

\theoremstyle{definition}
\newtheorem{defin}[thm]{Definition}
\newtheorem{example}[thm]{Example}
\newtheorem{def/ex}[thm]{Definition/Example}

\theoremstyle{remark}
\newtheorem{rem}[thm]{Remark}
\newtheorem{rems}[thm]{Remarks}

\newtheorem{claim}[thm]{Claim}

\newcommand{\refS}[1]{Section~\ref{S:#1}}
\newcommand{\refT}[1]{Theorem~\ref{T:#1}}
\newcommand{\refC}[1]{Corollary~\ref{C:#1}}
\newcommand{\refP}[1]{Proposition~\ref{P:#1}}
\newcommand{\refD}[1]{Definition~\ref{D:#1}}

\def\phi{\varphi}

\graphicspath{{pict/},{./../pict/},{./../../pict/}}

\begin{document}

%%%%%%%%%%%%%%%%%%%%%%%%%%%%%%%%%%%%%%%%%%%%%%

\title[Braids and integrals]{Diagram complexes, formality, and configuration space integrals for spaces of braids}

%%%%%%%%%%%%%%%%%%%%%%%%%%%%%%%%%%%%%%%%%%%%%%

\author{Rafal Komendarczyk}
\address{Department of Mathematics, Tulane University, 6823 St. Charles Ave, New Orleans, LA 70118}
\email{rako@tulane.edu}
\urladdr{dauns01.math.tulane.edu/\textasciitilde rako}

\author{Robin Koytcheff}
\address{Department of Mathematics, University of Louisiana at Lafayette, P.O.~Box 43568, Lafayette, LA 70504}
\email{koytcheff@louisiana.edu}
\urladdr{ucs.louisiana.edu/\textasciitilde c00401634}

\author{Ismar Voli\'c}
\address{Department of Mathematics, Wellesley College, 106 Central Street, Wellesley, MA 02481}
\email{ivolic@wellesley.edu}
\urladdr{ivolic.wellesley.edu}

\subjclass[2010]{Primary: 55P35, 55R80, 57Q45, 81Q30; Secondary: 57R40, 57M27}
\keywords{Braids, loop spaces, spaces of links, cobar construction, configuration space integrals, iterated integrals, formality, graph complexes, trivalent graphs, chord diagrams}

%%%%%%%%%%%%%%%%%%%%%%%%%%%%%%%%%%%%%%%%%%%%%%%%%%%%%%%%%%%%%%%%%%%%%%%%%%%%%%%%%%%%%%%%%%%%%%%%%%%%%%%%

\begin{abstract}
We use rational formality of configuration spaces and the bar construction to study the cohomology of the space of braids in dimension four or greater.  We provide a diagram complex for braids and a quasi-isomorphism to the de Rham cochains on the space of braids.  The quasi-isomorphism is given by a configuration space integral followed by Chen's iterated integrals.
This extends results of Kohno and of Cohen and Gitler on the cohomology of the space of braids to a CDGA suitable for integration.
We show that this integration is compatible with Bott-Taubes configuration space integrals for long links via a map between two diagram complexes.
As a corollary, we get a surjection in cohomology from the space of long links to the space of braids.  
We also discuss to what extent our results apply to the case of classical braids.
\end{abstract}

%%%%%%%%%%%%%%%%%%%%%%%%%%%%%%%%%%%%%%%%%%%%%%%%%%%%%%%%%%%%%%%%%%%%%%%%%%%%%%%%%%%%%%%%%%%%%%%%%%%%%%%%

\maketitle

\tableofcontents

\baselineskip=13pt
\parskip=4pt
\parindent=0cm

%%%%%%%%%%%%%%%%%%%%%%%%%%%%%%%%%%%%%%%%%%%%%%%%%%%%%%%%%%%%%%%%%%%%%%%%%%%%%%%%%%%%%%%%%%%%%%%%%%%%%%%%

%\newpage

\section{Introduction}\label{S:Intro}

%%%%%%%%%%%%%%%%%%%%%%%%%%%%%%%%%%%%%%%%%%%%%%%%%%%%%%%%%%%%%%%%%%%%%%%%%%%%%%%%%%%%%%%%%%%%%%%%%%%%%%%%

We study the space of $m$-strand braids in $\R^{n+1}$,
% for $n\geq 3$.
or more precisely, the space $\Omega \Conf(m, \R^n)$ of based loops in the space of ordered configurations in $\R^n$.  
We study its cohomology via complexes of diagrams and integration maps from these complexes.  
One ingredient is the bar complex on cochains of a space, which is well-known to compute the cohomology of the based loop space, and (pure) braids are precisely loops in the configuration space.
On the other hand, the cochain algebra of configuration spaces is quasi-isomorphic to a complex of diagrams $\D(m)$, which was used by Kontsevich \cite{K:OMDQ} to establish the rational formality of configuration spaces.  (That is, there is a CDGA map from $\D(m)$ to the cochain algebra inducing an isomorphism on rational cohomology.)  
We combine these ideas to show that a bar complex $\Ba(\D(m))$ of \emph{braid diagrams} produces all the cohomology of the space of braids, via integration.
We then relate the integrals from this bar complex to the integrals from a complex $\LD(m)$ of \emph{long link diagrams}, which produce cohomology classes in the space $\Lk_m^{n+1}$ of long $m$-component links in $\R^{n+1}$, as previously studied by the last two authors.  The commutative diagram in the theorem below summarizes our main results. In the rest of this introduction we elaborate on its individual components.  We will also mention various consequences that can be deduced from the maps in the square.  
We mainly focus on the case where $n \geq 3$, that is, braids in dimension four or greater.  
However, these methods can be applied to classical braids as well, and we gather those results separately in \refS{R^3BraidsCohomology}.

\begin{repthm}{T:MainDiagramThm}
	\label{T:IntegralsAgree}
	Let $n \geq 3$.  Then there is a commutative square of maps of cohomology with real coefficients:
	\begin{equation}
	\label{E:MainCommutativeSquare}
	\xymatrix{
		\Ho^{0,*}(\LD_\mathrm{f}(m)) \ar^-{\varint_{\mathrm{Bott-Taubes}}}[rr] \ar@{->>}_-\phi[d] & & \Ho^*(\Lk_m^{n+1}) \ar^-{\iota^*}[d] \\
		\Ho^*(\Ba(\D(m))) \ar^-{\varint_{\mathrm{Chen}} \circ \varint_{\mathrm{formality}}}_-{\cong}[rr] & & \Ho^*(\Omega \Conf(m, \R^n))
	}
	\end{equation}
	The maps are induced by CDGA maps, provided de Rham cochains are used in the spaces on the right.
\end{repthm}

In \refS{IntroChenFormality}, we will discuss the bottom horizontal map; In \refS{BraidsAsLinks}, we will describe the right vertical map; 
In \refS{IntroBT}, we will discuss the top horizontal map and the left vertical map.

\subsection{Configuration space formality integrals, the bar construction, and Chen integrals}
\label{S:IntroChenFormality}
Let $\Conf(m,\R^n)$ be the based configuration space of $m \geq 1$ points in $\R^n$, $n \geq 3$, namely
$$
\Conf(m,\R^n)=\{(x_1,...,x_m)\in (\R^n)^m \colon x_i\neq x_j \text{ for } i\neq j\}.
$$  
Kontsevich \cite{K:OMDQ} defined quasi-isomorphisms of commutative differential graded algebras (CDGAs)
\begin{equation}\label{E:ConfFormalityQuasiIsosIntro}
\ChdR^\ast(\Conf(m,\R^{n}))\stackrel{I}{\longleftarrow} \D(m) \stackrel{\overline{I}}{\longrightarrow} \Ho^*(\Conf(m,\R^{n})).
\end{equation}
Here $\ChdR^\ast(\Conf(m,\R^{n}))$ is the de Rham cochain algebra and $\D(m)$ is a certain complex of diagrams (see \refS{Diagrams}).  The map $\overline{I}$ is easy to define, so the crux of the construction is the map $I$ (denoted by $\varint_{\mathrm{formality}}$ in the diagram \eqref{E:MainCommutativeSquare}), which is given by integration of differential forms over configuration spaces (see \refS{Integral}).  The proof that this is a quasi-isomorphism is subtle and ingenious.
 \footnote{In fact, Kontsevich proved more, namely that the little balls operad is formal; an elaboration of this proof can be found in \cite{LV}.}  

Now one can take cochains and cohomology CDGAs $\ChdR^{\ast}(\Conf(m,\R^n))$ and $\Ho^{*}(\Conf(m,\R^n))$ and perform the standard algebraic bar construction to obtain bar complexes, which are CDGAs (and moreover Hopf algebras), that we will denote by $\Ba(\ChdR^{\ast}(\Conf(m,\R^n)))$ and $\Ba(\Ho^{*}(\Conf(m,\R^n)))$.  (The bar construction is reviewed in \refS{BarConstruction}.)  Similarly, the  bar construction on $\D(m)$ gives another CDGA, $\Ba(\D(m))$.  
We will observe that the maps $I$ and $\overline I$ pass through the bar construction and will consequently, for $n\geq 3$ and $m\geq 1$, get CDGA quasi-isomorphisms, over $\R$ (\refT{BraidsFormality}):
\begin{equation}\label{E:MainZigzag}
\Ba(\ChdR^{\ast}(\Conf(m,\R^n))) \stackrel{I}{\longleftarrow}\Ba(\D(m)) \stackrel{\overline I}{\longrightarrow}\Ba(\Ho^{*}(\Conf(m,\R^n)))
\end{equation}

The map $I$ is precisely $\int_{\mathrm{formality}}$ from \eqref{E:MainCommutativeSquare}.
Now, for $n\geq 3$, the left CDGA  in \eqref{E:MainZigzag} is quasi-isomorphic to cochains on the based loop space $\Omega\Conf(m,\R^n)$; see \refT{CommuteTotAndCochains}.  (This fact can be viewed as a consequence of a theorem on cosimplicial models for mapping spaces, such as the geometric cobar construction for $\Omega X$; see \refS{Cosimplicial}.)  
By taking $\Omega\Conf(m,\R^n)$ to consist of piecewise smooth loops and using de Rham cochains, one can realize this quasi-isomorphism 
via Chen's iterated integrals.

To explain, the de Rham cohomology of a loop space $\Omega M$ was first defined by Chen in \cite{Chen:Itr-Integrals, Chen:Path-Integrals} in a general setting of a differentiable space $M$. More precisely, Chen constructed a double complex
\[
\Chen^{*,*}(M)=\bigoplus_{p,q\geq 0} \Chen^{-p,q}(M),\quad \Chen^{-p,q}(M)=\big\langle \varint \omega_1 \omega_2 \cdots \omega_p\ |\ \omega_i\in \ChdR^{\ast}(M)\big\rangle
\]
of so called {\em iterated integrals} which can be formally viewed as a subcomplex of differential forms $\ChdR^\ast(\Omega M)$ on $\Omega M$. Further, in \cite{Chen:Reduced_bar, Hain:Integrals} Chen defined a natural chain map, which  we will refer to as {\em Chen's integration map}, from the bar complex of $\ChdR^\ast(M)$ to $\Chen^{*,*}(M)$:
\begin{equation}\label{E:ChenIntegrationMap}
\begin{split}
\varint_{\mathrm{Chen}}\colon\Ba(\ChdR^{\ast}(M)) & \longrightarrow \Chen^{*,*}(M)\subset \ChdR^\ast(\Omega M),\\
[\omega_1 | \omega_2 | \cdots | \omega_p] & \longrightarrow \varint \omega_1 \omega_2 \cdots \omega_p.
\end{split}
\end{equation}
This map is a quasi-isomorphism for simply connected $M$ as proved in \cite{Chen:Itr-Integrals, Chen:Path-Integrals}. In view of the classical results \cite{Anderson:HSS, EM:Cotor} one concludes that the iterated integrals determine $\Ho^\ast(\Omega M)$. We will review Chen's complex  in Section \ref{S:Chen} where we deduce Chen's theorem using the machinery of cosimplicial spaces presented in the second part of \refS{CobarCohomology}.

In our setting, the configuration space $\Conf(m,\R^n)$ is simply connected for $n\geq 3$, and the quasi-isomorphism \eqref{E:ChenIntegrationMap} thus becomes the quasi-isomorphism
\begin{equation}\label{E:ChenIntegrationMapBraids}
\Ba(\ChdR^{\ast}(\Conf(m,\R^n))) \longrightarrow \ChdR^*(\Omega\Conf(m,\R^n)).
\end{equation}
This is what was referred to as $\varint_{\mathrm{Chen}}$ in the diagram \eqref{E:MainCommutativeSquare}.

Composing the quasi-isomorphisms $I=\varint_{\mathrm{formality}}$ from \eqref{E:MainZigzag} and $\varint_{\mathrm{Chen}}$ from \eqref{E:ChenIntegrationMapBraids} gives the bottom map in the diagram \eqref{E:MainCommutativeSquare}, which is one of our primary results: 
\begin{repthm}{T:DiagramsBraidCohomology}
For $n\geq 3$ and $m\geq 1$, there is a quasi-isomorphism of CDGAs
\begin{equation}\label{E:MainQuasi-iso}
\Ba(\D(m)) \stackrel{}{\longrightarrow} \ChdR^{\ast}(\Omega\Conf(m,\R^n))
\end{equation}
which can be realized as the formality integration map $I\colon\Ba(\D(m)) \longrightarrow \Ba(\ChdR^{\ast}(\Conf(m,\R^n)))$ followed by Chen's integration map.  Moreover, the induced map in cohomology is an isomorphism of Hopf algebras. 
\end{repthm}

Of course, $\Omega\Conf(m,\R^n)$ is precisely the space of (pure) braids in $\R^{n+1}$.  The left side is thus a diagram complex that captures all of the cohomology of the space of braids in dimension four or greater.  
Accordingly, we call $\Ba(\D(m))$ the \emph{braid diagram complex}.

A consequence of the quasi-isomorphisms in  \eqref{E:MainZigzag} that goes beyond what is captured in the diagram \eqref{E:MainCommutativeSquare} is that we can use the CDGA on the right side to compute the cohomology of the space of braids (\refC{CDGAModelForBraids}).  
This cohomology was already known from work of Cohen and Gitler \cite{CG}.  In fact, they showed that the homology is given as 
\[
\Ho_*(\Omega \Conf(m,\R^n); \Z) \cong \mathrm{U} \mathcal{L}_m(n-2),
\]
where $\mathrm{U}$ denotes the universal enveloping algebra, and $\mathcal{L}_m(n-2)$ is the Lie algebra on generators $B_{ij}$ of degree $n-2$, $1 \leq j < i \leq n$, satisfying the \emph{Yang--Baxter Lie relations} (a.k.a.~the \emph{infinitesimal braid relations}) below, where $B_{ji}:=(-1)^nB_{ij}$.  We will denote these relations by YB:
\begin{equation}
\label{E:Yang-Baxter}
[B_{ij}, B_{jk} +B_{ik}] = 0,\qquad
[B_{ij}, B_{k\ell}] =0,\qquad i, j, k, \ell \quad \text{distinct}.
\end{equation}
The above homology can be further identified with an algebra $\mathcal{BCD}(m)$ of chord diagrams on $m$ (braid) segments (see \refD{BCD(m)}).  Each $B_{ij}$ corresponds to a diagram with just one chord with endpoints on segments $i$ and $j$, and the product in $\mathrm{U}\mathcal{L}_m(n-2)$ corresponds to stacking of diagrams.  
This space of chord diagrams modulo the above relations is dual to a subspace of our braid diagram complex $\Ba(\D(m))$.
Indeed, we represent generators $\alpha_{ij}$ of $\Ho^*(\Conf(m,\R^n))$ by chord diagrams $\Gamma_{ij}$ (linear dual to the homology classes $B_{ij}$).  Applying a result of Priddy \cite{Priddy:Koszul} on the bar complex, we recover the dual in cohomology of the Cohen--Gitler result.  
This calculation may be well-known to experts (see e.g.~\cite{B:Koszul, Bezrukavnikov:Koszul-Config}), 
but we provide the details here in terms of the cohomology of the space of braids.  
\begin{repthm}{T:BraidCohIsCDMod4T}
For $n\geq 3$ and $m\geq 1$, $\Ho^*(\Omega \Conf(m,\R^n))$ is given by linear combinations of chord diagrams satisfying the four-term (4T) and shuffle relations (the dual of the Yang--Baxter Lie relations).
\end{repthm}

There are two consequences we can deduce from Theorems \ref{T:DiagramsBraidCohomology} and \ref{T:BraidCohIsCDMod4T} which recover and refine theorems of Kohno \cite{Kohno:LoopsFiniteType, Kohno:BarComplex}, who studied Chen's constructions for braids  \cite{Kohno:Vassiliev, Kohno:VassilievBraids, Kohno:LoopsFiniteType} in the context of finite type invariants \cite{Vass:Cohom}.

Namely, let $P_m:=\pi_1(\Conf(m,\R^2))$, the pure braid group on $m$ strands.  Let $J$ be the augmentation ideal in the group ring $\R[P_m]$.  If we set $V_k(P_m):= \R[P_m]/(J^{k+1})$, then $V_k(P_m)$ is the space of $\R$-valued finite type $k$ (Vassiliev) invariants of pure braids on $m$ strands.
Functionals on chord diagrams with $k$ chords which vanish on sums of the form \eqref{E:Yang-Baxter} are well-known to correspond to finite type $k$  invariants of pure braids:
\[
(\mathcal{BCD}_k(m) / (\mathrm{YB}))^* \cong V_k(P_m) / V_{k-1}(P_m).
\]
The left side also corresponds exactly to classes of degree $k(n-2)$ in $\Ho^*(\Omega(\Conf(m,\R^n)))$ in \refT{BraidCohIsCDMod4T}.  This theorem thus immediately implies that
\begin{equation}
\label{E:BraidCohIsVassilievSpace}
\Ho^{k(n-2)}(\Omega \Conf(m,\R^n)) \cong V_k(P_m) / V_{k-1}(P_m),
\end{equation}
as observed by Kohno \cite[Theorem 4.1]{Kohno:LoopsFiniteType}.
He also notes that for even $n$, the filtration quotients fit together to give isomorphisms of Hopf algebras, by considering all possible $k$ at once.  In the classical case of  $\Omega\Conf(m,\R^2)$, Kohno additionally uses Chen's integrals and residual nilpotence of $P_m$ to show in \cite{Kohno:VassilievBraids} that finite type invariants separate pure braids, or equivalently $\Ho_0(\Omega\Conf(m,\R^2))$.

This connection to finite type invariants is more than peripheral; we use it to establish the relationship to Bott--Taubes configuration space integrals for long links in $\R^{n+1}$, $n\geq 3$ (as briefly discussed in \refS{IntroBT}).

Kohno's other result of relevance here is \cite[Theorem 5.1]{Kohno:LoopsFiniteType}.  Our \refT{DiagramsBraidCohomology} refines that result, which states that every cohomology class in the space of braids is represented by a Chen integral of the form 
\begin{equation}
\label{E:Kohno}
\sum a_{i_1 j_1 \dots i_p j_p}\varint \omega_{i_1j_1} \cdots \omega_{i_p j_p} + \bigl(\text{iterated integrals of length}<p\bigr)
\end{equation}
where the coefficients $a_{i_1 j_1 \ldots i_p j_p}$ satisfy the Yang--Baxter Lie relations.  
Specifically,  for $\gamma \in \Ba(\D(m))$, write $\gamma = \gamma' + \gamma''$, where $\gamma'$ consists of all the chord diagram terms in $\gamma$.
Let $\int \gamma$ denote the image of $\gamma$ under the map in \refT{DiagramsBraidCohomology}. 
Then the expression \eqref{E:Kohno} corresponds exactly to $\int \gamma' + \int \gamma''$.
In the formality zig-zag \eqref{E:MainZigzag}, the map $\overline{I}$ to cohomology is zero on $\gamma''$.  Thus $\overline{I}$ captures $\int \gamma'$, where the Yang--Baxter relations are necessary conditions on $\gamma'$ for $\gamma$ to be closed.  Our refinement is that via the formality integration map $I$ to cochains, we get the full expression for $\int\gamma$, including the iterated integrals of length $<p$.

For example, the following linear combination $\gamma'$ of chord diagrams on $m=3$ strands, with $p=2$ chords, satisfies the Yang--Baxter relations:
\[
\raisebox{-1.4pc}{\includegraphics[scale=0.2]{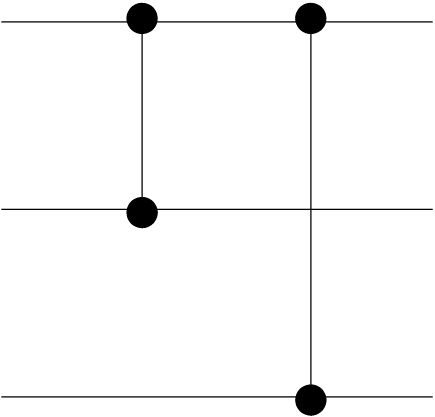}} -
\raisebox{-1.4pc}{\includegraphics[scale=0.2]{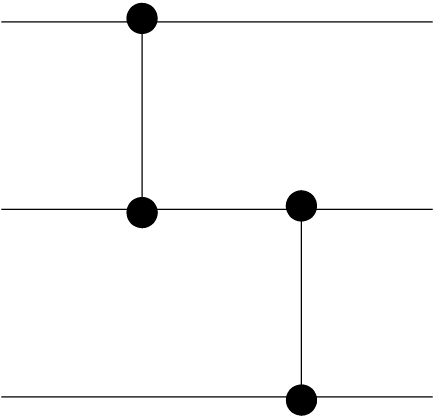}} +
\raisebox{-1.4pc}{\includegraphics[scale=0.2]{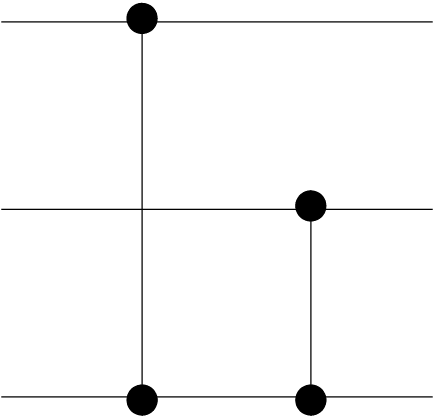}}
\,\,\, \in \,\, (\mathcal{BCD}(3)/(\mathrm{YB}))^*
\]
One can show that this class spans $(\mathcal{BCD}(3)/(\mathrm{YB}))^*$ and hence for $n\geq 3$, it spans $\Ho^{2(n-2)}(\Omega \Conf(m,\R^n))$.  However, to get a formula for an integral, one needs the full expression
\begin{equation}\label{E:mu_123}
\mu_{123} := 
\raisebox{-1.4pc}{\includegraphics[scale=0.2]{LM.eps}} -
\raisebox{-1.4pc}{\includegraphics[scale=0.2]{LR.eps}} +
\raisebox{-1.4pc}{\includegraphics[scale=0.2]{MR.eps}} -
\raisebox{-1.4pc}{\includegraphics[scale=0.2]{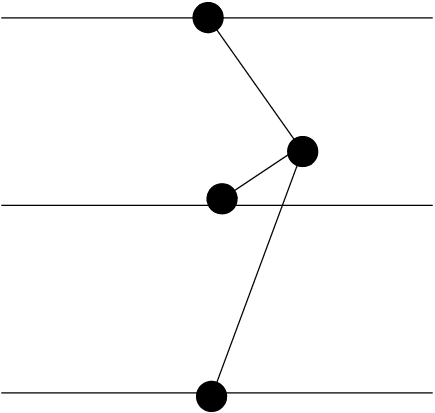}}
  \,\,\, \in \,\, \Ba(\D(m))
\end{equation}
for the cocycle corresponding to this class.  
(This is because for $n\geq 3$, the three-term relation \eqref{E:CohomCond5} in the cohomology of configuration space is only an exact form, not identically zero.)
The map $\overline{I}$ to $(\mathcal{BCD}(3)/(\mathrm{YB}))^*$ is zero on the tripod $\gamma''$ (the rightmost diagram in \eqref{E:mu_123}), but the integral $\int \gamma''$ may be nonzero.  See Example \ref{3StrandsOrder3}, also mentioned in \refS{IntroBT}, for another example.

\subsection{Viewing braids as links}\label{S:BraidsAsLinks}
Let 
$$
\Lk_m^{n+1}=\Emb_c\left(\coprod_m \R, \, \R^{n+1}\right)
$$ 
be the space of \emph{compactly supported long links of $m$ strands in $\R^{n+1}$}, namely the space of embeddings of $\coprod_m \R$ in $\R^{n+1}$ which outside of $\coprod_m I$ agree with $2m$ fixed affine-linear maps with disjoint images.  We may assume that these linear maps pass through the points $(x_i,0)$ and $(x_i,1)$, where $(x_1,...,x_m)$ is the basepoint in $\Conf(m,\R^n)$.
The space $\Omega \Conf(m,\R^n)$ of braids in $\R^{n+1}$ is a subspace of $\Lk_m^{n+1}$.
More precisely, view a braid $\beta$ as $m$ maps $x_i\colon I \to \R^{n}$, $i=1,...,m$, such that $x_i(t) \neq x_j(t)$ for all $i \neq j$ and for all $t \in I$ and such that $x_i(0)=x_i(1)$ is the basepoint.  
A long link is determined by $m$ embeddings $I \to \R^{n+1}$ with disjoint images, with the endpoints mapping to the boundary of the box.  (This embedding is extended outside the box in the prescribed way, smoothing the maps at the boundary of the box if necessary.)
Then the inclusion $\iota$ of the space of braids into the space of links is given by
\begin{align*}
\iota\colon  \Omega \Conf(m,\R^n) & \longrightarrow  \Lk_m^{n+1} \\
 [\beta\colon t  \mapsto (x_1(t),...,x_m(t))] & \longmapsto 
[ (t_i \mapsto (t_i, x_i(t_i)) )_{i=1}^m].
\end{align*}
The right vertical map $\iota^*$ in the diagram \eqref{E:MainCommutativeSquare} is then just the map induced on cohomology by $\iota$ by the standard pullback of (smooth) singular cochains.

\subsection{Bott--Taubes integrals for long links and the map of diagram complexes}
\label{S:IntroBT}
Configuration space integrals, originally due to Bott and Taubes \cite{BT}, have shown to be effective in understanding the cohomology of knots and links. The application of this theory to long links was studied in \cite{KMV:FTHoLinks}, where one defines a \emph{link diagram complex} $\LD(m)$ and a chain map 
$$
\LD(m)\longrightarrow \ChdR^\ast(\Lk_m^{n+1}),
$$
often called the \emph{Bott--Taubes configuration space integral}. The map is given by integrating differential forms over configuration spaces, where the combinatorics of the integrals is encoded by the complex $\LD(m)$.  We will review the complex $\LD(m)$ and the above map in \refS{B-TRelation}.  

The above map essentially induces the top horizontal map $\int_{\mathrm{Bott-Taubes}}$ from the diagram \eqref{E:MainCommutativeSquare} on cohomology, except one must first restrict to a subspace $\LD_\mathrm{f}(m)$ of forest diagrams.  The reason why this restriction is needed comes from the last remaining map in the diagram, $\phi$.

To explain, we will construct a map from the link diagram complex to the braid diagram complex (\refD{LinksToBraidsMap}),
\begin{equation}\label{E:MapPhiIntro}
\phi\colon\LD(m)\longrightarrow \Ba(\D(m)),
\end{equation}
given by modifying a link diagram in almost the simplest possible way to make it graphically look like a braid diagram.  
However, we will only be able to show that $\phi$ is a chain map when restricted to a subspace of cocycles of trivalent forests in $\LD_\mathrm{f}(m)$.  The complex $\LD(m)$ is bigraded  by \emph{order} and \emph{defect} (see \refS{OrderDefectForBraidDiagrams}), and the restriction of \eqref{E:MapPhiIntro} maps from the cohomology with defect zero onto the cohomology of $\Ba(\D(m))$ (\refT{DiagramsSurjOnCohom}).
The proof of this surjectivity uses results from the classical case $n=2$, i.e.~Vassiliev invariants of long links and pure braids in $\R^3$.
We will thus get the left column of the square \eqref{E:MainCommutativeSquare}, where we abuse notation by using $\phi$ again for the map induced on cohomology:
\[
\xymatrix{
\phi\colon\Ho^{0,n}(\LD_\mathrm{f}(m)) \ar@{->>}[r] & \Ho^*(\Ba(\D(m)))
}
\]
We will then show that $\phi$ is compatible with integration, giving us the commutativity of that square, thereby completing the proof of \refT{MainDiagramThm}.

As one consequence (Corollary \ref{C:LinksSurjectBraids}), we get that for $n\geq 3$, the inclusion of the space of braids into the space of long links in $\R^{n+1}$ induces a surjection in cohomology
\begin{equation}
\label{E:LinksToBraidsSurj}
\Ho^*(\mathcal{L}_m^{n+1}) \twoheadrightarrow \Ho^*(\Omega \Conf(m, \R^n)).
\end{equation}
Although the surjectivity of $\phi$ comes from knowledge of Vassiliev invariants rather than a diagrammatic argument involving $\phi$ itself, this Corollary requires the full diagram \eqref{E:MainCommutativeSquare}.  Indeed, to know that the surjection \eqref{E:LinksToBraidsSurj} is induced by the inclusion of spaces, one needs the commutativity of \eqref{E:MainCommutativeSquare} and thus the map $\phi$ itself.

Another consequence is that $\Ba(\D(m))$ provides integral formulae with no more terms than those coming from $\LD(m)$, and in many cases strictly fewer. 
On the one hand, the element $\mu_{123} \in \Ba(\D(m))$ from  \eqref{E:mu_123} coincides exactly with an element in the diagram complex $\LD(m)$ for long links, so no simplification occurs in this case.  
On the other hand (Example \ref{3StrandsOrder3}), we may consider 

\[
\raisebox{-1.4pc}{\includegraphics[scale=0.2]{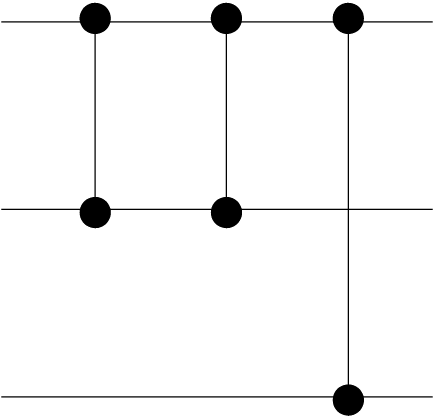}} - 
\raisebox{-1.4pc}{\includegraphics[scale=0.2]{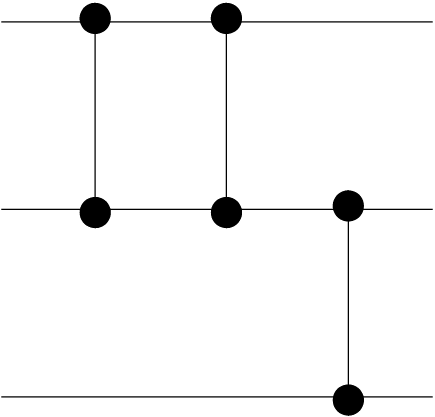}} + 
\raisebox{-1.4pc}{\includegraphics[scale=0.2]{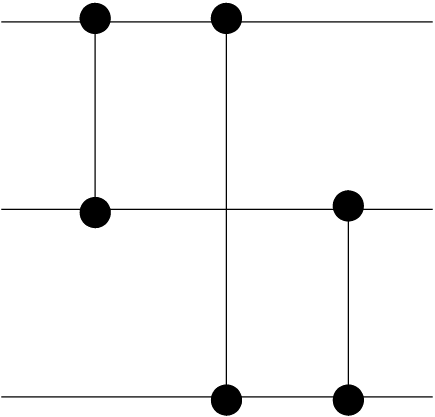}} - 
\raisebox{-1.4pc}{\includegraphics[scale=0.2]{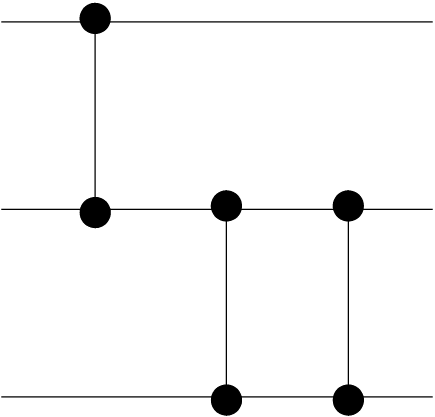}} + 
\raisebox{-1.4pc}{\includegraphics[scale=0.2]{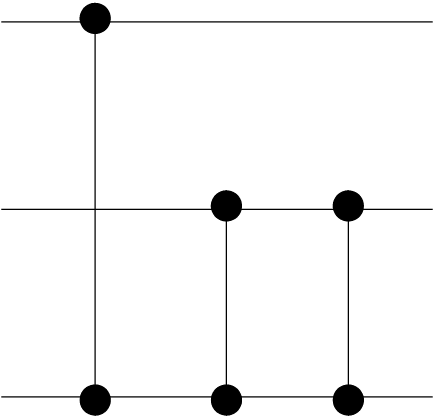}}  \pm 
\raisebox{-1.4pc}{\includegraphics[scale=0.2]{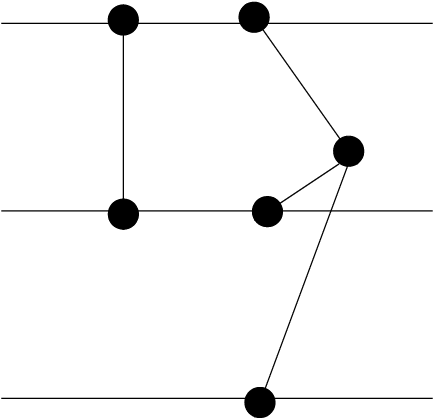}} \pm 
\raisebox{-1.4pc}{\includegraphics[scale=0.2]{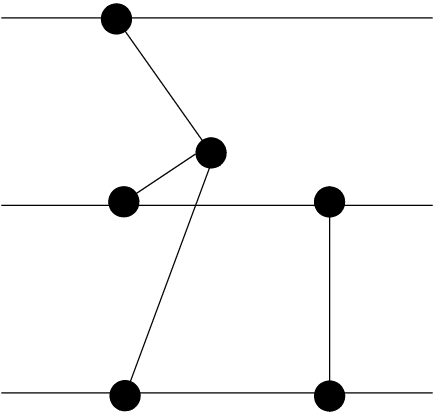}} \pm 
\raisebox{-1.4pc}{\includegraphics[scale=0.2]{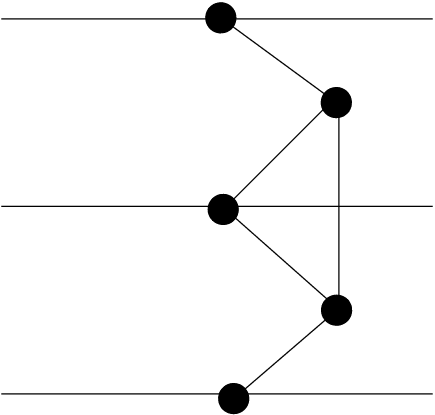}},
\]

a cocycle with 8 terms in $\Ba(\D(m))$ representing a class in $\Ho^{3(n-2)}(\Omega(\Conf(m,\R^n))$.  
This class also comes from a cocycle in $\LD(m)$ -- as guaranteed by \refT{DiagramsSurjOnCohom} -- but the latter cocycle has (at the very least) more than 11 terms.
In the classical case $n=2$, this cocycle corresponds to a finite type 3  invariant of pure braids.  

For $n=2$, $\mu_{123}$ corresponds to the triple linking number for long links, and the lack of simplification under the map $\phi$ is representative of the situation for all ``Milnor link homotopy invariant'' classes, as follows:  
Since Milnor invariants are of finite type, the identification \eqref{E:BraidCohIsVassilievSpace} provides a cohomology class for each such invariant.  In \refP{HtpyDiagramsAreGood}, we will show that for such a class, the corresponding sum of diagrams in $\Ba(\D(m))$ (and thus the associated integral) has exactly the same terms as the corresponding one in $\LD(m)$.

\subsection{Heuristics and remarks}\label{S:Heuristics}
\refT{MainDiagramThm} unifies various approaches to the study of the cohomology of spaces of braids: the bar construction, configuration space formality, Chen integrals, and Bott-Taubes configuration space integrals.  While each of these constructions is familiar, we bring them together in a novel way through exposition that is as self-contained as possible.

The heuristic behind the compatibility of these constructions is as follows: The complex $\LD(m)$ encodes the combinatorics of the integration that produces cohomology classes of links.  It indexes collections of configuration points where some points move on the link and some points move in the ambient space.  Certain directions between the points are recorded, and, intuitively, these directions capture the complexity of the link.  
In this construction, we may choose the directions to lie in an equatorial $S^{n-1} \subset S^n$ perpendicular to the $x_1$ axis (where $(x_1,\ldots,x_{n+1})$ are coordinates on $\R^{n+1}$).
For an arbitrary link, such directions may be attained 
regardless of where the configuration points are on the strands.  However, because braids always ``flow with constant speed in the $x_1$ direction,'' such directions are only attained between configuration points that lie in ``time slice'' $\R^n$ perpendicular to the flow of the braids.  Thus for braids, it suffices to consider diagrams where components lie in parallel planes.
%are never attained between two configuration points on the same strand.  
%Thus for braids, it suffices to consider diagrams 
%chords joining distinct strands and more generally, $(\geq 3)$-valent diagrams where each component lies in a ``time slice'' $\R^n$ perpendicular to the flow of the braids.
The combinatorics and algebra of such diagrams are organized precisely by the bar construction.  The map \eqref{E:MapPhiIntro} makes this precise as it sends a diagram in $\LD(m)$ (that indexes a configuration of points anywhere) to a diagram that indexes configurations lying in parallel planes perpendicular to the flow of the braid.

Now, integrating forms over configuration spaces indexed by the bar complex of diagrams to produce forms on the space of braids can be regarded as a two-step process:  First take a diagram with free vertices and integrate over all possible positions of the configuration points (in $\R^n$) corresponding to those free vertices. This is the formality integration. What remains is an integral over all possible positions of the time slices, each given by $(0\leq t_1 \leq \ldots \leq t_p\leq 1) \in \Delta^p$.
The Chen iterated integral is this second step, which produces a form on the space of braids.
The bottom map in \eqref{E:MainCommutativeSquare} is the composition of these two integrations.

 Since the map  \eqref{E:MainQuasi-iso} is a quasi-isomorphism, it is natural to ask whether the same is true for the map \eqref{E:MapPhiIntro}, namely whether $\LD(m)$ captures all the cohomology of links (and knots).  This is a standing conjecture, posed in the early days of the Bott-Taubes configuration space integration \cite{CCRL}.  Our affirmative answer for the case of braids provides evidence that this conjecture might indeed be true.

In the case of classical braids, our \refT{MainDiagramThm} still applies, except that the bottom horizontal map is not an isomorphism (see \refS{R^3BraidsCohomology}).  
In this situation, our work bears some relation to the Kontsevich integral, which is a universal finite type invariant of tangles (including knots, links, and braids) in $\R^3$, meaning that it provides a construction of any finite type invariant out of certain combinations of diagrams.\footnote{Note that the diagrams in $\D(m)$ and $\LD(m)$ represent elements dual to the space of diagrams considered e.g.~by Bar-Natan \cite{BN:Vass}.  This is analogous to our comparison with the Cohen--Gitler theorem above, where we consider cohomology rather than homology.}  
In fact, for $n=2$, the zig-zag \eqref{E:ConfFormalityQuasiIsosIntro} can be replaced by a map directly from cohomology to cochains, requiring no integration.  
Then the map \eqref{E:MainQuasi-iso} is given simply by the Chen integral, as studied in \cite{Kohno:Vassiliev}.  
Classes in $\Ho^*(\Conf(m,\R^2))$ can be identified with {chord} diagrams in $\D(m)$, and the Chen integral for classical braids is the prototype of the Kontsevich integral for general tangles.  
In this setting, chords correspond to propagators coming from $S^1$.  
The Bott--Taubes configuration space integral is, up to an anomaly term, also a universal finite type invariant of knots, links, and braids in $\R^3$, but with a propagator coming from $S^2$.  
It is an open conjecture that in the classical case, the Bott--Taubes integral (a.k.a.~the Chern--Simons invariant) and Kontsevich integral (a.k.a~the Knizhnik--Zamolodchikov invariant) agree \cite[p. 149]{K:Vass}.  Our \refT{MainDiagramThm} 
affirms the compatibility in the case of braids, not only in the classical case but also in the more general setting of braids in Euclidean spaces of arbitrary dimension. 

In the view of the generalization of Chen's integrals to mapping spaces $\Map(M,N)$ by \cite{GTZ:ChenMapping, BottSegal}, the study of the cohomology of $\Omega \Conf(m, \R^n)$ via the diagram complex $\D(m)$ should admit a natural generalization to the space $\Map(M,\Conf(m, \R^n))$.  This should have interesting applications to the homotopy invariants of link maps, as originally investigated by Koschorke in \cite{Ko:Linkhomotopy, Ko:GenMilnor}, as well as geometric questions in knot theory and related areas of physics \cite{DGKMSV:Triplelinking, KV:VolpresFields, Kom:Helicity, KM:Ropelength}.  
Graph complexes similar to $\D(m)$ appear in recent work generalizing this CDGA model from configurations in Euclidean space to manifolds \cite{CW:ModelConfig, Idrissi:ModelConfig}, as well as recent work on spaces of long embeddings $\Emb(\R^m, \R^n)$ and their deloopings \cite{AT:Graph, FTW:EnOperads}.  
We will explore these avenues and connections in future work.

 \subsection{Organization of the paper}
 In \refS{CobarCohomology}, we review the bar construction on an algebra and the geometric cobar construction on a space.  We then recall the notion of rational formality of a space and examine how it interacts with the cobar construction.  
In \refS{CobarBraidCohomology}, we apply this general setup to the case of configuration spaces of points in Euclidean space.  We review the diagram complex and integrals that guarantee their formality, and we use formality to rederive the calculation the cohomology of spaces of braids in terms of chord diagrams.
In \refS{Chen}, we review Chen's iterated integrals for the loop space of a simply connected manifold and recover Chen's Theorem on de Rham cohomology of the loop space as a consequence of general results of \refS{CobarCohomology}. 
In \refS{B-TRelation}, we first review Bott--Taubes configuration space integrals for long links, including the associated diagram complex.  We then define the map of diagram complexes and prove its desired properties, establishing our main theorem.  
The Appendix has a proof of a statement related to orientations of diagrams.

%POSSIBLE ADDITION: Should we bring the tripod and Arnold identity (or even type-3 invariants for 3-strand braids) somewhere earlier into the introduction?  Maybe at the beginning of this heuristics subsection?

\subsection{Acknowledgments}\label{S:Acknowledgments}
The third author would like to thank the Simons Foundation for its support and University of Sarajevo for its hospitality.  He would also like the thank Pascal Lambrechts and Yves F\' elix for helpful discussions on formality of mapping spaces.

%%%%%%%%%%%%%%%%%%%%%%%%%%%%%%%%%%%%%%%%%%%%%%%%%%%%%%%%%%%%%%%%%%%%%%%%%%%%%%%%%%%%%%%%%%%%%%%%%%%%%%%%

\section{The cobar construction, formality, and the cohomology of the loop space}
\label{S:CobarCohomology}

%%%%%%%%%%%%%%%%%%%%%%%%%%%%%%%%%%%%%%%%%%%%%%%%%%%%%%%%%%%%%%%%%%%%%%%%%%%%%%%%%%%%%%%%%%%%%%%%%%%%%%%%

In this section, we discuss the interaction of rational formality with the geometric cobar construction, which is a standard model for the loop space.  The main consequences on the cohomology of the loop space are Theorem \ref{T:BarFormal} and Corollary \ref{C:CDGAModelForOmegaX}.  These are general results which may be of independent interest and will be applied to the special case of braids in \refS{CobarBraidCohomology}.  
We will assume that the reader is familiar with the basics of (co)simplicial objects and totalizations, as well as with (graded-)commutative differential graded algebras (CDGAs).

%%%%%%%%%%%%%%%%%%%%%%%%%%%%%%%%%%%%%%%%%%%%%%%%%%%%%%%%%%%%%%%%%%%%%%%%%%%%%%%%%%%%%%%%%%%%%%%%%%%%%%%%

\subsection{Bar construction of an algebra}\label{S:BarConstruction}

%%%%%%%%%%%%%%%%%%%%%%%%%%%%%%%%%%%%%%%%%%%%%%%%%%%%%%%%%%%%%%%%%%%%%%%%%%%%%%%%%%%%%%%%%%%%%%%%%%%%%%%%

Suppose $\calA$ is a differential graded algebra (DGA) with differential $d$ that raises degree.  Then the \emph{bar construction} $\Ba_\bullet(\calA)$ is a simplicial chain complex with
\begin{equation}\label{E:BarComplex}
\Ba_p(\calA) = \calA^{\otimes p},
\end{equation}
spanned by elements $a_1 \otimes \cdots  \otimes a_p$. In keeping with more standard notation, we will denote these elements by $[a_1 | \cdots  | a_p]$.
Face maps are given by
\begin{align}
d_i\colon \calA^{\otimes p} & \longrightarrow \calA^{\otimes (p-1)} , \ \ \ 0\leq i\leq p \label{E:BarFaces}\\
[a_1 | \cdots  | a_p] & \longmapsto
\begin{cases} \notag
[a_2 | \cdots  | a_p],            & i=0;  \\
[a_1 | \cdots | a_ia_{i+1}|\cdots | a_p], & 1\leq i\leq p-1;  \\
[a_1 | \cdots  |a_{p-1}],            & i=p,  \\
\end{cases}
\end{align}
and degeneracies are given by
\begin{align}
s_i\colon \calA^{\otimes p} & \longrightarrow \calA^{\otimes p+1} , \ \ \ 0\leq i\leq p\label{E:BarDegeneracies} \\
[a_1 | \cdots  | a_p]  &  \longmapsto
[a_1 | \cdots | a_{i}| \,1\, | a_{i+1}|\cdots | a_p].\notag
\end{align}

One can form a double complex 
\begin{equation}\label{E:BarDoubleComplex}
\Ba^{*,*}(\calA)=\bigoplus_{p,q\geq 0} \Ba^{-p,q}(\calA)
\end{equation}
where
\begin{equation}\label{E:BarNormalizedDoubleComplex}
\Ba^{-p,q}(\calA)=(\calA^{\otimes p})^q/\Big(\sum_{i=0}^{p-1} \operatorname{im}(s_i)\Big)
\end{equation}
where $(\calA^{\otimes p})^q$ means the monomials $\calA^{\otimes p}$ whose degrees add up to $q$.

Since we took out the images of the degeneracies, this produces the \emph{normalized} complex; the degeneracies form an acyclic part of the double complex.  The differentials are
\begin{align}
\delta_1\colon \Ba^{-p,q}(\calA) & \longrightarrow \Ba^{-p,q+1}(\calA)\label{E:VertDiffBar} \\
[a_1 | \cdots  | a_p] & \longmapsto \sum_{i=1}^p (-1)^{\varepsilon(i-1)}[a_1| \cdots | da_i | \cdots | a_p], \notag\\
\delta_2\colon \Ba^{-p,q}(\calA) & \longrightarrow \Ba^{-p+1,q}(\calA) \label{E:HorDiffBar}\\
[a_1 | \cdots  | a_p] & \longmapsto \sum_{i=0}^{p} (-1)^{\varepsilon(i)} d_i[a_1 | \cdots  | a_p]\notag
\end{align}
where $\varepsilon(i) = \sum_{j=1}^i (|a_j|+1)$.  (In particular, $\varepsilon(0)=0$, the empty sum.)
We will refer to these differentials as ``vertical'' and ``horizontal'', respectively.

Now take the total complex, or totalization, of the double complex $\Ba^{*,*}(\calA)$, which can be denoted by $\Tot \Ba^{*,*}(\calA)$, but we will use the more standard shorter notation
\begin{equation}\label{E:BarTotalComplex}
\Ba(\calA):=\Tot \Ba^{*,*}(\calA).
\end{equation}
In degree $k$, this complex is 
\begin{equation}\label{E:TotCochainsCosimplicial}
\Ba(\calA)^k =\bigoplus_{k=q-p}\Ba^{-p,q} (\calA)
\end{equation}
with the total differential 
\begin{equation}\label{E:BarTotalDiff}
\delta\colon \Ba^{-p,q}(\calA) \longrightarrow \Ba^{-p,q+1}(\calA)\oplus \Ba^{-p+1,q}(\calA)
\end{equation}
given by $\delta=\delta_1+\delta_2$.  This is the \emph{(normalized) bar complex}, or the \emph{(normalized) Hochschild complex} of $\calA$.  Its homology is the \emph{Hochschild homology} of $\calA$.

The complex $\Ba(\calA)$ is always a differential graded coalgebra.  The coproduct $\Delta$ (which preserves total degree) is given by deconcatenation
\begin{align}
\Delta \colon \Ba^{*,*}(\calA) & \longrightarrow \Ba^{*,*}(\calA) \otimes \Ba^{*,*}(\calA) \label{E:CoproductOnBar} \\
[a_1 | \cdots  | a_p] & \longmapsto \sum_{i=0}^p 
[a_1| \cdots | a_i] \otimes [a_{i+1} | \cdots | a_p], \notag
\end{align}
where the terms for $i=0$ and $i=p$ are respectively $1 \otimes [a_1 | \cdots  | a_p]$ and $[a_1 | \cdots  | a_p] \otimes 1$.

If $\calA$ is {graded-commutative} (i.e.~a CDGA), then the bar complex has a graded-commutative shuffle product which satisfies the graded Leibniz rule with respect to the differential $\delta$.  It is given by
\begin{equation}
\label{E:BarShuffle}
[a_1 | \cdots  | a_p] \ast 
[a_{p+1} | \cdots  | a_{p+p'}] =
\sum_{\mbox{shuffles $\sigma$ of $p+p'$}} (-1)^{\tau(i)} 
[a_{\sigma(1)} | \cdots | a_{\sigma(p+p')}]
\end{equation}
where $\tau(i) = \sum_{j=1}^i (|a_{\sigma(j)}|+1)$.  
In this case, $\Ba^{*,*}(\calA)$ is then a differential graded Hopf algebra.  
Moreover, the bar construction $B(-)$ is functorial and takes quasi-isomorphisms of CDGAs to quasi-isomorphisms of Hopf algebras.

%%%%%%%%%%%%%%%%%%%%%%%%%%%%%%%%%%%%%%%%%%%%%%%%%%%%%%%%%%%%%%%%%%%%%%%%%%%%%%%%%%%%%%%%%%%%%%%%%%%%%%%%

\subsection{Cobar model for the loop space}\label{S:Cosimplicial}

%%%%%%%%%%%%%%%%%%%%%%%%%%%%%%%%%%%%%%%%%%%%%%%%%%%%%%%%%%%%%%%%%%%%%%%%%%%%%%%%%%%%%%%%%%%%%%%%%%%%%%%%
The geometric cobar construction is a space-level model for loop spaces, which appeared at least as early as work of Rector \cite{Rector:E-MSS}.  At the chain level, the cobar construction goes back to Adams and Hilton \cite{ Adams:CobarPNAS, Adams:Cobar, Adams-Hilton}.  Strictly speaking, the bar construction on an algebra suffices for our computations and main theorem, but lifting to the space level provides conceptual clarity as to why the bar construction is related to loop space cohomology.

Let $X$ be a space with basepoint $*$, and let 
%$$
%\Omega X =\Map_*(I,X)
%$$
$$
\Omega X =\Map_*(S^1,X)
$$
denote the usual based loop space of $X$, consisting of the set of maps sending a basepoint in $S^1$ to the basepoint in $X$.  This space is endowed with the compact-open topology. There is a standard cosimplicial model $\Xdot$ for $\Omega X$, called the \emph{cobar construction on $X$}, which consists of the collection of spaces $X^p$, $p\geq 0$, with coface maps
\begin{align}
d^i\colon X^{p} & \longrightarrow X^{p+1}, \ \ \ 0\leq i\leq p+1 \label{E:LoopCofaces} \\
(x_1, x_2, ..., x_p) & \longmapsto
\begin{cases} \notag
(*, x_1, x_2, ..., x_p),            & i=0;  \\
(x_1, x_2, ...,x_i, x_i, ..., x_p), & 1\leq i\leq p;  \\
(x_1, x_2, ..., x_p, *),            & i=p+1,  \\
\end{cases}
\end{align}
and codegeneracy maps
\begin{align}
s^i\colon X^{p} & \longrightarrow X^{p-1} , \ \ \ 0\leq i\leq p-1\label{E:LoopCodegeneracies} \\
(x_1, x_2, ..., x_p) &  \longmapsto
(x_1, x_2, ...,x_{i}, x_{i+2}, ... x_p).\notag
\end{align}

Recall also the standard cosimplicial simplex $\Delta^\bullet$ which consists of the simplices 
$$
\Delta^p=\{(t_1,...,t_p)\in \R^p \colon 0\leq t_1\leq t_2\leq \cdots \leq t_p\leq 1\}, \ \ \ p\geq 0,
$$ 
with cofaces given by
\begin{align}
d^i\colon \Delta^p & \longrightarrow \Delta^{p+1}, \ \ \ 0\leq i\leq p+1\label{E:DeltaCofaces} \\
(t_1, ..., t_p) &\begin{cases} \notag
(0, t_1, t_2, ..., t_p),            & i=0;  \\
(t_1, t_2, ...,t_i, t_i, ..., t_p), & 1\leq i\leq p;  \\
(t_1, t_2, ..., t_p, 1),            & i=p+1,  \\
\end{cases}
\end{align}
and codegeneracies given by 
\begin{align}
s^i\colon \Delta^{p} & \longrightarrow \Delta^{p-1} , \ \ \ 0\leq i\leq p-1\label{E:DeltaCodegeneracies} \\
(t_1, ..., t_p) &  \longmapsto
(t_1, t_2, ..., t_i, t_{i+2}, ... t_p).\notag
\end{align}

Now consider the (geometric) totalization of $X^\bullet$,
\begin{equation}\label{E:TotalizationDef}
\Tot X^\bullet=\Map(\Delta^\bullet, X^\bullet),
\end{equation}
the space of maps from $\Delta^\bullet$ to $X^\bullet$.  This is the subspace of the product of spaces of maps (with compact-open topology) $\Map(\Delta^p, X^p)$ for all $p\geq 0$ containing those maps that commute with the cofaces and the codegeneracies.

We have an evaluation map 
\begin{align}
\ev_p\colon  \Delta^p\times \Omega X & \longrightarrow X^p \label{E:Evaluation}\\
(t_1, ..., t_p, \gamma) & \longmapsto (\gamma(t_1), ..., \gamma(t_p)). \notag
\end{align} 
This has an adjoint
$$
\Omega X\longrightarrow \Map(\Delta^p, X^p)
$$
It is not hard to see that the collection of such maps for all $p\geq 0$ commutes with cofaces $d^i$ and codegeneracies $s^i$.  In other words, the above map passes to a map
\begin{equation}\label{E:Loops=Tot}
\Omega X\longrightarrow \Tot X^\bullet.
\end{equation}
A standard result (see, for example, \cite[Example 9.2.3]{MV:Cubes}) is that the map \eqref{E:Loops=Tot} is a homeomorphism.

\begin{rem}\label{R:MapsModel}
The cobar construction and the equivalence \eqref{E:Loops=Tot} are an instance of the cosimplicial model for the based mapping space $\Map_*(X,Y)$ obtained by applying $\Map_*(-,Y)$ to a simplicial model $X_\bullet$ for $X$.  Then we have  $\Map_*(X,Y) = \Tot \Map_*(X_\bullet,Y)$ \cite[Proposition 5.1]{BottSegal}. (Also see \cite[Example 2.3.1]{GTZ:ChenMapping} for more on how the cobar construction gives an example of this general procedure).  
\end{rem}

%%%%%%%%%%%%%%%%%%%%%%%%%%%%%%%%%%%%%%%%%%%%%%%%%%%%%%%%%%%%%%%%%%%%%%%%%%%%%%%%%%%%%%%%%%%%%%%%%%%%%%%%

%%%%%%%%%%%%%%%%%%%%%%%%%%%%%%%%%%%%%%%%%%%%%%%%%%%%%%%%%%%%%%%%%%%%%%%%%%%%%%%%%%%%%%%%%%%%%%%%%%%%%%%%

%\subsection{Cohomology spectral sequence of the cobar construction}\label{S:CobarSS}

%%%%%%%%%%%%%%%%%%%%%%%%%%%%%%%%%%%%%%%%%%%%%%%%%%%%%%%%%%%%%%%%%%%%%%%%%%%%%%%%%%%%%%%%%%%%%%%%%%%%%%%%

Now apply cochains $\Ch^*$ with coefficients in a field to $X^\bullet$.  This produces a simplicial DGA $\Ch^*(X^\bullet)$ with face maps $(d^i)^*$ and degeneracy maps $(s^i)^*$.  
We then get a double complex
\begin{equation}\label{E:CosimplicialDoubleComplex}
\Ch^{*,*}(\Xdot)=\bigoplus_{p,q\geq 0} \Ch^{-p,q}(\Xdot)
\end{equation}
where
\begin{align}
\Ch^{-p,q}(\Xdot) & =\Ch^q (X^p)/\Big(\sum_{i=0}^{p-1} \operatorname{im}((s^i)^*)\Big) \label{E:NormalizedDoubleComplex} \\
& \cong \Big(\bigoplus_{i_1+i_2+\cdots+i_p=q,\ i_k\geq 0} \Ch^{i_1} (X)\otimes\cdots\otimes\Ch^{i_p} (X)\Big)/\Big(\sum_{i=0}^{p-1} \operatorname{im}((s^i)^*)\Big) \notag
\end{align}

The differentials are 
\begin{align}
\delta_1\colon \Ch^{-p,q}(\Xdot) & \longrightarrow \Ch^{-p,q+1}(\Xdot)\label{E:Diff1Cobar} \\
\alpha & \longmapsto (-1)^p d\alpha \notag\\
\delta_2\colon \Ch^{-p,q}(\Xdot) & \longrightarrow \Ch^{-p+1,q}(\Xdot) \label{E:Diff2Cobar}\\
\alpha & \longmapsto \sum_{i=0}^{p+1}(-1)^i(d^i)^*\alpha \notag
\end{align}

\begin{rem}\label{R:CohomologySS}
The double complex $\Ch^{*,*}(\Xdot)$ is the $E_0$ page of the standard Bousfield-Kan cohomology spectral sequence associated to the cosimplicial space $\Xdot$.
\end{rem}

We finally have the totalization 
\begin{equation}\label{E:TotalComplex}
\Tot \Ch^{*,*}(X^\bullet)
\end{equation}
which is in degree $k$ given by 
\begin{equation}\label{E:TotCochainsCosimplicial}
\Tot \Ch^{*,*}(X^\bullet)^k =\bigoplus_{k=q-p}\Ch^{-p,q} (X^\bullet).
\end{equation}
The total differential 
\begin{equation}\label{E:TotalDiff}
\delta\colon \Ch^{-p,q}(\Xdot) \longrightarrow \Ch^{-p,q+1}(\Xdot)\oplus \Ch^{-p+1,q}(\Xdot)
\end{equation}
is as usual given by $\delta=\delta_1+\delta_2$.  The algebra structure is given by a shuffle product as in \eqref{E:BarShuffle}  (also see \cite[p.~293]{BottSegal} for more detail about the shuffle in the particular case of cochains).

The following two results are well-known and go back to  Anderson \cite{Anderson:HSS}, Bendersky-Gitler \cite{BG:FunctSpaces},   and Bott-Segal \cite{BottSegal}.  The first says that $\Tot \Ch^{*,*}(X^\bullet)$ can readily be identified with the bar complex from \refS{BarConstruction} and it follows from the  K\"{u}nneth Theorem and \cite[Proposition 5.9]{BottSegal}.

\begin{prop}\label{P:ModelsInduceQuasiIsos}
If $\calA$ is CDGA that is quasi-isomorphic to $\Ch^*(X)$, then the CDGAs $\Ba_\bullet(\calA)$ and $\Ch^{*}(X^\bullet)$ are quasi-isomorphic.  It follows that the double complexes $\Ba^{*,*}(\calA)$ and $\Ch^{*,*}(X^\bullet)$, and hence their totalizations $\Ba(\calA)$ and $\Tot\Ch^{*,*}(X^\bullet)$, are also quasi-isomorphic as CDGAs.  
\end{prop}
%\begin{proof}
%The first part of the above statement follows from the K\"{u}nneth Theorem.  One just needs to verify that the maps $(d^i)^*$ and $(s^i)^*$ are indeed the same as the faces and degeneracies from \refS{BarConstruction}, which is routine.
% That a quasi-isomorphism of simplicial CDGAs induces a quasi-isomorphism of double and total complexes can be found, for example, in \cite[Proposition 5.9]{BottSegal}.
%\end{proof}

In particular, the K\"{u}nneth Theorem gives a quasi-isomorphism
\begin{equation}\label{E:BarQuasiTot}
\Ba(\Ch^*(X))\stackrel{\simeq}{\longrightarrow}\Tot\Ch^{*,*}(X^\bullet).
\end{equation}

There is a natural cochain map \cite[p. 433]{BG:FunctSpaces} 
\begin{equation}\label{E:TotAndCochains}
\rho:\Tot \Ch^{*,*}(X^\bullet) \longrightarrow \Ch^*\Tot (X^\bullet),
\end{equation}
defined as the composition
\begin{equation}\label{E:rho-def}
 C^{p+q}(X^p)\stackrel{\ev^\ast_p}{\longrightarrow} C^{p+q}(\Delta^p\times \Tot (X^\bullet))\stackrel{\theta^\ast}{\longrightarrow} \bigl(C^\ast(\Delta^p)\otimes C^\ast(\Tot (X^\bullet))\bigr)^{p+q}\stackrel{\rho_p}{\longrightarrow} C^q(\Tot (X^\bullet)),
\end{equation}
where $\ev^\ast_p$ is the pullback via the evaluation map \eqref{E:Evaluation}, $\theta^\ast$ is the isomorphism of the K\"{u}nneth 
formula (i.e.~the Eilenberg-Zilber map), and $\rho_p$ is defined as the slant product (see \cite[p. 287]{Spanier:Top})
\[
 \rho_p(x\otimes y):=x(\iota_p) y,\quad \iota_p\in C_p(\Delta^p),
\]
where $\iota_p$ is the canonical $p$--simplex $\iota_p:\Delta^p\longrightarrow \Delta^p$.

The map $\rho$ is not necessarily a quasi-isomorphism, i.e., cochains and totalization do not always commute.  However, there is the following well-known result, one version of which can be found in the work of Bendersky and Gitler \cite[Theorem 5.1]{BG:FunctSpaces}.
In \refS{Chen} we use the differential form analog of $\rho$ to obtain the Chen's Iterated Integral Theorem.

\begin{thm}\label{T:CommuteTotAndCochains}
If $X$ is simply connected, then the map \eqref{E:TotAndCochains} is a quasi-isomorphism, so it induces isomorphisms
$$
\Ho^i(\Ba (\Ch^*(X))
\cong 
\Ho^i(\Tot \Ch^{*,*}(X^\bullet))
 \stackrel{\cong}{\longrightarrow} \Ho^i(\Tot (X^\bullet))\cong \Ho^i(\Omega X)
$$
for all $i\geq 0$.
\end{thm}

\begin{rem}\label{R:MapsModel2}
In line with Remark \ref{R:MapsModel}, the above theorem is a special case of a more general result that, as long as the dimension of the nerve of $X_\bullet$ is less than or equal to the connectivity of $Y$, the canonical map $\Tot \Ch^{*,*}(\Map(X_\bullet,Y)) \to \Ch^*\Tot \Map(X,Y)$ is a quasi-isomorphism.
This is a classical result; for the case when the dimension of the nerve of $X_\bullet$ is less than the connectivity of $Y$, see  \cite[Theorem 5.1]{BG:FunctSpaces} (which in turn cites papers of Anderson \cite{Anderson:HSS} and Bott and Segal \cite[Proposition 5.3]{BottSegal}), and for the case when the dimensions are the same, see the work of Patras and Thomas \cite[Theorem 2]{PT:Mapping}.   The case of the loop space goes back to Eilenberg and Moore \cite{EM:Cotor}; see also McCleary's book \cite[Chapter 7]{McCleary:UserGuide}.
\end{rem}
\begin{rem}
Notice that we have two kinds of totalization in \refT{CommuteTotAndCochains}.  One is the algebraic totalization of the double complex $\Ch^{*,*}(\Xdot)$, and the other is the geometric totalization of the cosimplicial space $\Xdot$ which produces a space; this is an unfortunate abuse of notation.
\end{rem}

The map \eqref{E:TotAndCochains} can be composed with the pullback quasi-isomorphism induced by the equivalence \eqref{E:Loops=Tot}, namely
\begin{equation}\label{E:HomeoOnCochains}
\Ch^*\Tot (X^\bullet)\stackrel{\simeq}{\longrightarrow} \Ch^*(\Omega X),
\end{equation}
and we can further extend this to the left by \eqref{E:BarQuasiTot}.   
For $X$ simply connected, we thus have a composition of quasi-isomorphisms
\begin{equation}\label{E:QuasiisoComp}
\Ba(\Ch^*(X))\stackrel{\simeq}{\longrightarrow}\Tot \Ch^{*,*}(X^\bullet) \stackrel{\simeq}{\longrightarrow} \Ch^*\Tot (X^\bullet)\stackrel{\simeq}{\longrightarrow} \Ch^*(\Omega X)
\end{equation}
which is then itself a quasi-isomorphism.

%%%%%%%%%%%%%%%%%%%%%%%%%%%%%%%%%%%%%%%%%%%%%%%%%%%%%%%%%%%%%%%%%%%%%%%%%%%%%%%%%%%%%%%%%%%%%%%%%%%%%%%%

\subsection{Formality and the cobar construction}\label{S:CobarFormality}

%%%%%%%%%%%%%%%%%%%%%%%%%%%%%%%%%%%%%%%%%%%%%%%%%%%%%%%%%%%%%%%%%%%%%%%%%%%%%%%%%%%%%%%%%%%%%%%%%%%%%%%%

Here we recall the notion of formality from rational homotopy theory and deduce some observations about how formality interacts with the cobar construction.  This will be used in \refS{CobarBraidCohomology} to give us a complete picture of the cohomology of spaces of braids.

A DGA $\calA$ is said to be \emph{formal} if it is connected via a zig-zag of quasi-isomorphisms to its homology.  This notion can be extended to spaces as follows.

Let $\mathbb K$ be a field of characteristic zero.  For a space $X$, let $A_{PL}^*(X)$ be the CDGA of Sullivan's piecewise polynomial forms (see, for example, \cite[Section 10(c)]{FHT:RHT}).  For $X$ a connected manifold, $A_{PL}^*(X)$ is quasi-isomorphic (as CDGAs) to the usual de Rham algebra $\ChdR^\ast(X)$ of differential forms on $X$; this is the situation we will mostly be interested in.  So $A_{PL}^*(X)$ will now play the role of $\Ch^*(X)$ from \refS{Cosimplicial}.  \refP{ModelsInduceQuasiIsos} ensures that any result involving quasi-isomorphisms with $\Ch^*(X)$ can be replaced by one with $A_{PL}^*(X)$.

Also note that the cohomology $\Ho^*(A_{PL}^*(X,\mathbb K))\cong\Ho^*(X,\mathbb K)$ can be regarded as a CDGA if endowed with the zero differential.  

\begin{defin}\label{D:Formal}
A space $X$ is \emph{$\mathbb K$-formal} if there is a zig-zag of quasi-isomorphisms of CDGAs
\begin{equation}\label{E:Zigzag}
A_{PL}^*(X,\mathbb K)\stackrel{\simeq}{\longleftarrow}\cdots \stackrel{\simeq}{\longrightarrow}\Ho^*(X,\mathbb K).
\end{equation}
\end{defin}

\begin{thm}\label{T:BarFormal}
If $X$ is a $\mathbb K$-formal space, there is a zig-zag of quasi-isomorphisms of 
% bialgebras?
% commutative differential graded bialgebras?
(graded-commutative) differential graded Hopf algebras
\begin{equation}\label{E:BarFormal}
\Ba(A_{PL}^*(X,\mathbb K))\stackrel{\simeq}{\longleftarrow} \cdots \stackrel{\simeq}{\longrightarrow} \Ba(\mathcal A)\stackrel{\simeq}{\longleftarrow}\cdots\stackrel{\simeq}{\longrightarrow} \Ba(\Ho^*(X,\mathbb K)).
\end{equation}
\end{thm}

\begin{proof}
This follows immediately from the fact that $B(-)$ is functorial and takes quasi-isomorphisms of CDGAs to quasi-isomorphisms of Hopf algebras, as noted at the end of Section \ref{S:BarConstruction}.
\end{proof}

A version of the above theorem for the case of braids will appear later as \refT{BraidsFormality}.

%
%\begin{rem} 
%Via the K\"{u}nneth Theorem, \refT{BarFormal} may be rephrased by saying that the geometric cobar construction $\Xdot$ is formal (as a cosimplicial diagram).
%\end{rem} 
%

If $X$ is simply connected, then \eqref{E:QuasiisoComp} says that $\Ba(A_{PL}^*(X,\mathbb K))$ is additionally quasi-isomorphic to $A_{PL}^*(\Omega X, \mathbb K)$ so that the zig-zag \eqref{E:BarFormal}
 can be extended to the left.  %We will use this in \refT{DiagramsBraidCohomology}.
Unravelling the construction of $\Ba(A_{PL}^*(X,\mathbb K))$ then immediately gives

\begin{cor}\label{C:CDGAModelForOmegaX}
If $X$ is $\mathbb K$-formal and simply connected, then the cohomology of $\Omega X$ over any field of characteristic zero is the cohomology of the CDGA
\begin{equation}\label{E:CDGAModel}
\Big(      \bigoplus\limits_{p=0}^{\infty}   s^{-p} (\Ho^*(X))^{\otimes p}/\Big(\sum_{i=0}^{p-1} \operatorname{im}((s^i)^*)\Big), \,   
              \sum_{i=0}^{p+1}(-1)^i (d^i)^*
 \Big),
 \end{equation}
 where $s^{-p}$ denotes the degree shift.
\end{cor} 
%
%
%\begin{proof}
%First taking coefficients in $\mathbb K$, we have that the $E_1$ page in the Bousfield-Kan cohomology spectral sequence for $\Xdot$ is precisely $\Ba^{*,*}(A_{PL}^{*}(X,\mathbb K))$  (the standard reference is \cite{BK}; see also \cite{MV:Cubes}).  Because of formality,  that is quasi-isomorphic to the double complex $\Ba^{*,*}(\Ho^{*}(X,\mathbb K))$, as in \eqref{E:QuasiisoDoubleComplexes}.
%But since the differential in the latter complex is zero, the spectral sequence collapses at the $E_2$ page, and the collapse holds over any field of characteristic zero  \cite[Proposition 3.2]{LTV:Vass}. Thus the total complex of the double complex $\Ba^{*,*}(\Ho^{*}(X))$, namely $\Ba(\Ho^*(X))$, is quasi-isomorphic to the total complex of the $E_\infty$ page.
%
%Since $X$ is simply connected, the spectral sequence in fact converges to the cohomology of the totalization of $\Xdot$, i.e.~to $\Ho^*(\Tot\Xdot,\mathbb K)\cong\Ho^*(\Omega X,\mathbb K)$, by \refT{CommuteTotAndCochains}.  
%Finally, unravelling the definition of the complex $\Ba(\Ho^*(X,\mathbb K))$ as the total complex of the double complex $\Ba^{*,*}(\Ho^{*}(\Xdot,\mathbb K))$ (see the discussion starting with \eqref{E:CosimplicialDoubleComplex} but replace cochains by cohomology and $\delta_1$ by the zero differential), we deduce the desired result.
%\end{proof}

To see how \refC{CDGAModelForOmegaX} applies to the case of braids, see \refC{CDGAModelForBraids}.

\begin{rem}
Similar arguments can be applied when $X$ is \emph{coformal}, namely when there is there is a sequence of quasi-isomorphisms of differential
graded Lie algebras
\begin{equation}
\xymatrix{ \mathcal{L}\big( A_{PL}^*(X)\big) &  \cdots
\ar[l]_(0.3){\simeq} \ar[r]^(0.3){\simeq} & \big( \pi_{*}(\Omega
X)\otimes\Q,\ 0  \big). }
\end{equation}
Here $\mathcal{L}$ is a functor that takes a DGA $\mathcal A$ and produces
a DGL which is roughly the free Lie algebra on the dual of $\mathcal A$
with appropriate differential.
\end{rem}

%%%%%%%%%%%%%%%%%%%%%%%%%%%%%%%%%%%%%%%%%%%%%%%%%%%%%%%%%%%%%%%%%%%%%%%%%%%%%%%%%%%%%%%%%%%%%%%%%%%%%%%%

\section{The cobar construction on configuration spaces and the cohomology of braids}\label{S:CobarBraidCohomology}

%%%%%%%%%%%%%%%%%%%%%%%%%%%%%%%%%%%%%%%%%%%%%%%%%%%%%%%%%%%%%%%%%%%%%%%%%%%%%%%%%%%%%%%%%%%%%%%%%%%%%%%%

We now use the machinery of \refS{CobarCohomology} in the special case of $X=\Conf(m,\R^n)$, the configuration space of $m$ points in $\R^n$.  Then $\Omega\Conf(m,\R^n)$ is precisely the space of (pure) braids. 
The reason we can get more mileage in this case is that $\Conf(m,\R^n)$ is a formal space with two explicit formality quasi-isomorphisms
\begin{equation}\label{E:ConfFormalityQuasiIsos}
\ChdR^\ast(\Conf(m,\R^{n}))\stackrel{I}{\longleftarrow} \D(m) \stackrel{\overline{I}}{\longrightarrow} \Ho^*(\Conf(m,\R^{n})).
\end{equation}
Here $\D(m)$ is a certain complex of diagrams which we will review in \refS{Diagrams}.
The map $I$, the harder of the two quasi-isomorphisms, is given by integration maps whose combinatorics are kept track of by $\D(m)$.  This integration will be reviewed in Section \ref{S:Integral} but we will omit the details since they are spelled out carefully in \cite[Chapters 6 and 9]{LV}.  It is this diagram complex and the integration map that provide the bridge to the theory of Bott-Taubes configuration space integrals and Chen integrals, as will be further explored in Sections \ref{S:Chen} and \ref{S:B-TRelation}.

In \refS{CobarFormalityBraids}, we will combine \eqref{E:ConfFormalityQuasiIsos} with the cobar construction $\Conf(m,\R^n)^\bullet$ using the setup of \refS{CobarFormality} and this will subsequently provide information about the cohomology of braids $\Omega\Conf(m,\R^n)$.

Since integration will be $\R$-valued, from now on we will assume that our coefficient ring is $\R$ but will suppress it from the notation.
In addition, cochains $\Ch^*$ and polynomial forms $A_{PL}^*$ will be replaced by the CDGA of de Rham cochains $\ChdR^\ast$. 

%%%%%%%%%%%%%%%%%%%%%%%%%%%%%%%%%%%%%%%%%%%%%%%%%%%%%%%%%%%%%%%%%%%%%%%%%%%%%%%%%%%%%%%%%%%%%%%%%%%%%%%%

\subsection{The CDGA of admissible diagrams}\label{S:Diagrams}

%%%%%%%%%%%%%%%%%%%%%%%%%%%%%%%%%%%%%%%%%%%%%%%%%%%%%%%%%%%%%%%%%%%%%%%%%%%%%%%%%%%%%%%%%%%%%%%%%%%%%%%%

Here we review the CDGA of \emph{admissible diagrams} $\D(m)$.  These originally appear in \cite[Section 3.3.3]{K:OMDQ} and are discussed in detail in \cite[Chapter 6]{LV}.

\begin{defin}
\label{D:Diagrams}
Fix $n\geq 2$ and $m\geq 1$.  A \emph{diagram $\Gamma$ (in $\R^n$ on $m$ vertices)} is given by sets $V(\Gamma)$ and $E(\Gamma)$:
\begin{itemize}
\item
$V(\Gamma)$ is the set of vertices, partitioned into $m$ \emph{segment vertices} $V_{\mathrm{seg}}(\Gamma)$ with distinct labels from $\{1,...,m\}$, and any number of \emph{free vertices} $V_{\mathrm{free}}(\Gamma)$.  Thus $$V(\Gamma) = V_{\mathrm{seg}}(\Gamma) \sqcup V_{\mathrm{free}}(\Gamma).$$
\item
$E(\Gamma)$ is the set of edges joining vertices of $\Gamma$.  An edge between two segment vertices is called a \emph{chord}.
\end{itemize}
Furthermore, $\Gamma$ satisfies the following conditions:
\begin{itemize}
\item
Each free vertex of $\Gamma$ has valence at least 3 and is joined by a path of edges to some segment vertex.
\item
There are no self-loops in $\Gamma$, i.e.~no edge has a single vertex as both its endpoints.
\end{itemize}
\end{defin}

In \cite{LV}, segment vertices are called \emph{external vertices}, and free vertices are called \emph{internal vertices}.  Six examples of diagrams in $\D(3)$ are shown below.  
A segment vertex is represented by either a dot with a small horizontal line segment through it (if it has at least one incident edge) or just a small horizontal line segment (otherwise).
\[
\includegraphics[scale=0.2, angle=-90]{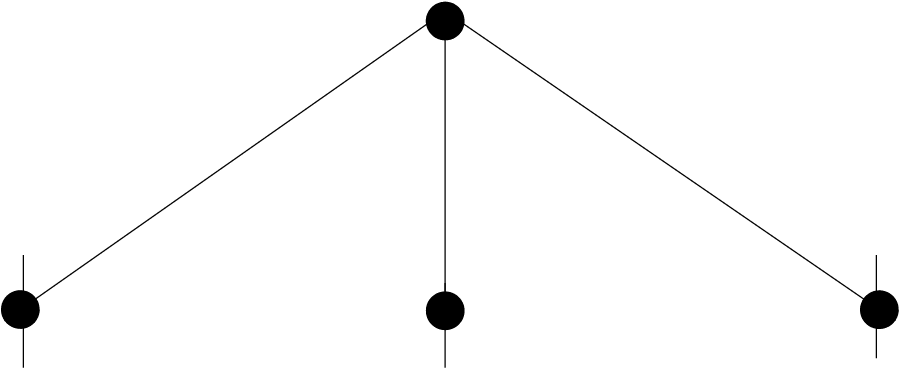}
\qquad \quad
\includegraphics[scale=0.18, angle=-90]{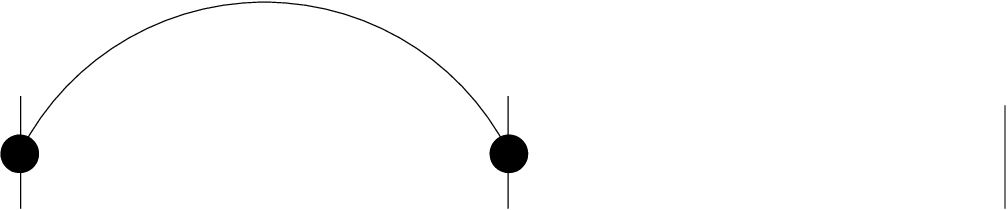}
\qquad \quad
\includegraphics[scale=0.18, angle=-90]{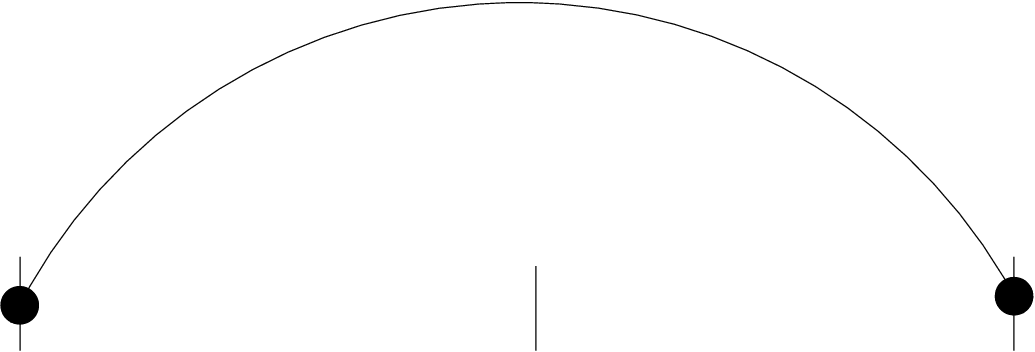}
\qquad \quad
\includegraphics[scale=0.18, angle=-90]{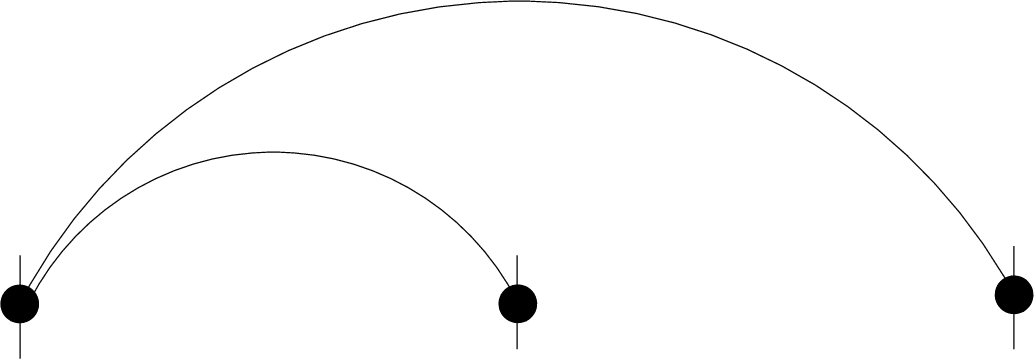}
\qquad \quad
\includegraphics[scale=0.18, angle=-90]{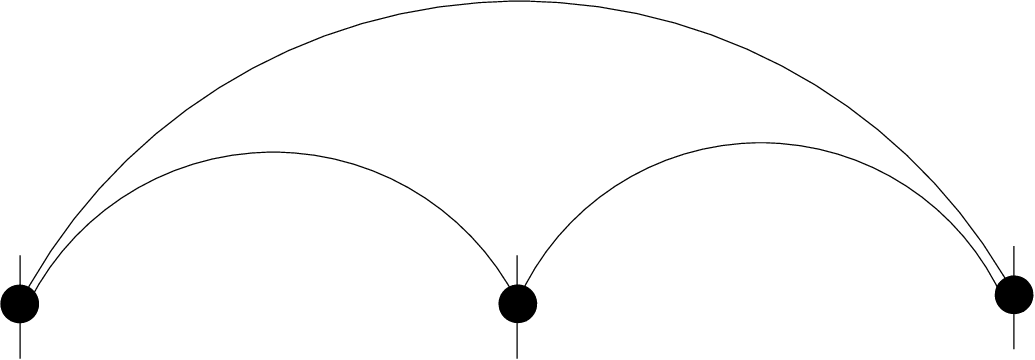}
\qquad \quad
\includegraphics[scale=0.2, angle=-90]{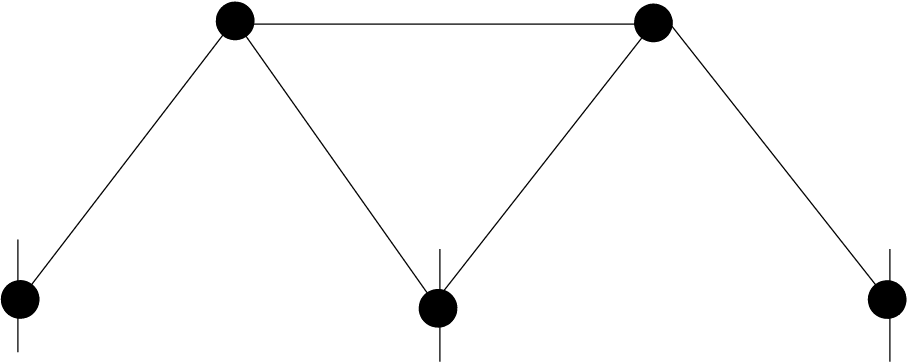}
\]
When we later picture elements of the bar construction on $\D(m)$, it will be convenient that we have aligned the segment vertices vertically rather than horizontally, as in \cite{LV}.  

\begin{defin}
\label{D:DmOrientations}
An \emph{orientation} of a diagram $\Gamma$ consists of the following data:
\begin{itemize}
\item
for odd $n$, a labeling of $V_{\mathrm{free}}(\Gamma)$ by integers $\{m+1,...\}$ and an orientation of each edge;
\item 
for even $n$, an ordering of the set of edges.
\end{itemize}
We call a diagram $\Gamma$ together with an orientation an \emph{oriented diagram}.
\end{defin}

We then let $\D(m)$ be the $\R$-vector space of oriented diagrams modulo the equivalence generated by the following relations:
\begin{itemize}
\item
$\Gamma \sim 0$ if $\Gamma$ has more than one edge between any two vertices;
\item
for odd $n$, $\Gamma \sim - \Gamma'$ if $\Gamma$ and  $\Gamma'$ differ by a transposition of two free vertex labels;
\item
for odd $n$, $\Gamma \sim -\Gamma'$ if $\Gamma$ and $\Gamma'$ differ by an orientation-reversal of an edge;
\item
for even $n$, $\Gamma \sim - \Gamma'$ if $\Gamma$ and  $\Gamma'$ differ by a transposition of two edge labels.
\end{itemize}

We view  $\D(m)$ as a graded vector space where the \emph{degree} of a diagram $\Gamma$ is
\[
|\Gamma| = (n-1)|E(\Gamma)| - n |V_{\mathrm{free}}(\Gamma)|.
\]
We will consider the bar complex on $\D(m)$ in Section \ref{S:CobarFormalityBraids}, in which case we may think of this as the internal degree.

There is a differential $d$ on $\D(m)$ given by 
\[
d\Gamma  = \sum \varepsilon(e) \Gamma/e
\]
where the sum is taken over non-chord edges $e$, $\varepsilon(e)$ is a sign depending on $e$, and $\Gamma/e$ is the result of contracting $e$ to a point.  This differential $d$ makes $\D(m)$ into a cochain complex.  For details, see \cite[Section 6.4]{LV}.

There is also a product of diagrams 
\begin{align*}
\D(m) \otimes \D(m) &\longrightarrow \D(m)\\
(\Gamma_1, \Gamma_2) & \longmapsto \Gamma_1 \cdot \Gamma_2
\end{align*}  
given by the juxtaposition, or superposition, of the two diagrams along segment vertices in the following sense: $\Gamma_1 \cdot \Gamma_2$ has $|V_{\mathrm{free}}(\Gamma_1 \cdot \Gamma_2)| = |V_{\mathrm{free}}(\Gamma_1)| + |V_{\mathrm{free}}(\Gamma_2)|$ and 
$|E(\Gamma_1 \cdot \Gamma_2)| = |E(\Gamma_1)| + |E(\Gamma_2)|$, but only $m$ segment vertices.  
The orientation is given by appropriately raising the labels of vertices or edges from $\Gamma_2$.
The degree of a product is the sum of degrees, and the product is graded-commutative, so that $\D(m)$ is a CDGA.  More details can be found in \cite[Sections 6.3 and 6.5]{LV}.

\begin{rem}
In \cite{LV}, a slightly broader class of diagrams was considered; from the definition of the differential, one can see that our $\D(m)$  is isomorphic to the complex of admissible diagrams in that reference.
\end{rem}

%%%%%%%%%%%%%%%%%%%%%%%%%%%%%%%%%%%%%%%%%%%%%%%%%%%%%%%%%%%%%%%%%%%%%%%%%%%%%%%%%%%%%%%%%%%%%%%%%%%%%%%%

\subsection{Formality maps for configuration spaces}\label{S:Integral}

%%%%%%%%%%%%%%%%%%%%%%%%%%%%%%%%%%%%%%%%%%%%%%%%%%%%%%%%%%%%%%%%%%%%%%%%%%%%%%%%%%%%%%%%%%%%%%%%%%%%%%%%

We first describe the harder of the two quasi-isomorphisms in \eqref{E:ConfFormalityQuasiIsos} that yield the formality of configuration spaces, namely the map
$$
I\colon \D(m) \longrightarrow \ChdR^\ast(\Conf(m, \R^n)).
$$
\begin{rem}\label{R:Compactification}
The configuration space $\Conf(m, \R^n)$ should at this point be replaced by its Axelrod--Singer (or Fulton--MacPherson) compactification.  This is a space that has the same homotopy type as $\Conf(m, \R^n)$ and is better to work with because of the issue of the convergence of the integral $I$.  Since we will not need the details of this construction, it will do no harm to abuse notation and continue using $\Conf(m, \R^n)$.  For further details, see for example \cite{FM}, \cite{AS} \cite{S:Compact}, and \cite[Section 4.1]{KMV:FTHoLinks}. 

In addition, the cochain complex $\ChdR^\ast$ used here is that of \emph{PA forms}, which are a variant of differential forms for manifolds with corners such as the compactified configuration space (see \cite[Section 8]{KoSo}, \cite{HLTV}).  We will also not need the details of this and will continue to think of the usual de Rham cochains. 
\end{rem}
Given a diagram $\Gamma\in\D(m)$ with $v$ internal vertices and $e$ edges, there are maps
\begin{equation}\label{E:I}
\phi_{ij} \colon \Conf(m+v, \R^n)\longrightarrow (S^{n-1})^{E(\Gamma)}
\end{equation}
given by 
$$
\phi_{ij}\colon (x_1, x_2, ..., x_{m+v})\longmapsto \frac{x_j-x_i}{\vert  x_j-x_i\vert}.
$$
Each such map can be used to pull back the product of the rotation-invariant unit volume form $\nu$ on the sphere.  Denote this pullback by $\alpha_{ij} := \phi_{ij}^*(\nu)$, and let
$$
\alpha_\Gamma = \bigwedge_{\text{edges $(i,j)$ of $\Gamma$}}\alpha_{ij}.
$$
Now let
$$
\pi\colon \Conf(m+v, \R^n)\longrightarrow \Conf(m, \R^n)
$$
be the projection onto the first $m$ configuration points.
Then $I$ is the linear map
$$
I\colon \D (m) \longrightarrow \ChdR^*(\Conf(m, \R^n))
$$
whose value on $\Gamma$ is
\begin{equation}\label{E:FormalityIntegral}
I(\Gamma)=\pi_{*} (\alpha_\Gamma)\in \ChdR^*(\Conf(m, \R^n)).
\end{equation}
Here $\pi_{*}$ denotes the pushforward, or integration along the fiber, of the projection $\pi$.  Thus for $(x_1,...,x_m)\in \Conf(m, \R^n)$, 
$$
I(\Gamma)(x_1,...,x_m)=\int_{\pi^{-1}(x_1,...,x_m)} \alpha_\Gamma.
$$
The degree of the form obtained this way is precisely the degree of $\Gamma$, as defined just after \refD{DmOrientations}.
One then shows that the form $I(\Gamma)$ is closed using Stokes' Theorem. This amounts to checking that the restrictions of integrals to the codimension one boundary components of the compactified configuration space vanish.  For details, see \cite[Chapter 9]{LV}.

The construction of the other quasi-isomorphism 
$$
\overline{I}\colon \D(m)\longrightarrow \Ho^*(\Conf(m,\R^{n})),
$$
relies on the ring structure of the cohomology of $\Conf(m,\R^n)$, $n\geq 2$, as computed in \cite{Cohen:Homology}.  
Namely, $\Ho^*(\Conf(m,\R^n))$ is generated by products $\alpha_{ij}$, each of which is the pullback of the unit volume form $\nu$ from $S^{n-1}$ by the direction map $\phi_{ij}$ defined in \eqref{E:I}.  The relations are
\begin{gather}
\alpha_{ij}^{2}=0,\label{E:CohomCond2} \\
\alpha_{ij}=(-1)^{n}\alpha_{ji}, \label{E:CohomCond3} \\
\alpha_{ij}\alpha_{kl}=(-1)^{n-1}\alpha_{kl}\alpha_{ij}, \label{E:CohomCond4} \\
\alpha_{ij}\alpha_{jk}+\alpha_{jk}\alpha_{ki}+\alpha_{ki}\alpha_{ij}=0.
\label{E:CohomCond5}
\end{gather}
The last relation is called the \emph{three-term (3T) relation}.  
The Poincar\'{e} polynomial for $\Conf(m,\R^n)$ is
\begin{equation}\label{E:PoincarePoly}
P(t)=(1+t^{n-1})(1+2t^{n-1})\cdots(1+(m-1)t^{n-1}).
\end{equation}
Via relation \eqref{E:CohomCond3}, we can take as generators $\alpha_{ij}$ with $i<j$, and use the convention that if $i<j$, then $\alpha_{ji} :=(-1)^n\alpha_{ij}$.  This leaves us with only the relations \eqref{E:CohomCond2}, \eqref{E:CohomCond4}, and \eqref{E:CohomCond5}, all of which are quadratic, which will be useful in \refP{NoCohomAboveVanishLine}.

The map $\overline{I}$ is then simply the unique algebra map which sends to $\alpha_{ij}$ the diagram with a single chord between segment vertices $i$ and $j$, and which sends any diagram with free vertices to 0.  

For all the details of why $I$ and $\overline I$ are quasi-isomorphisms, see \cite[Sections 8 and 9]{LV}.

Putting this together, we have the following result which exhibits the formality of configuration spaces.

\begin{thm}[Kontsevich \cite{K:OMDQ}]
\label{T:ConfigurationFormality}
For $n\geq 2$ and $m\geq 1$, there are CDGA quasi-isomorphisms
\begin{equation}
\label{E:IntegrationZigZag}
\ChdR^\ast(\Conf(m,\R^{n}))\stackrel{I}{\longleftarrow} \D(m) \stackrel{\overline{I}}{\longrightarrow} \Ho^*(\Conf(m,\R^{n})).
\end{equation}
\end{thm}

%%%%%%%%%%%%%%%%%%%%%%%%%%%%%%%%%%%%%%%%%%%%%%%%%%%%%%%%%%%%%%%%%%%%%%%%%%%%%%%%%%%%%%%%%%%%%%%%%%%%%%%%

\subsection{Cohomology of spaces of braids in $\R^{\geq 4}$}\label{S:CobarFormalityBraids}

%%%%%%%%%%%%%%%%%%%%%%%%%%%%%%%%%%%%%%%%%%%%%%%%%%%%%%%%%%%%%%%%%%%%%%%%%%%%%%%%%%%%%%%%%%%%%%%%%%%%%%%%

We can now extend the formality result to the cobar construction on $\Conf(m,\R^n)$ and thus obtain the following version of \refT{BarFormal}. 

\begin{thm}\label{T:BraidsFormality} 
For $n\geq 2$ and $m\geq 1$, there is a zig-zag of quasi-isomorphisms
\begin{equation}
\label{E:BarZigzag}
\Ba(\ChdR^{\ast}(\Conf(m,\R^n)))\stackrel{B(I)}{\longleftarrow}
\Ba(\D(m)) \stackrel{B(\overline I)}{\longrightarrow}
\Ba(\Ho^{\ast}(\Conf(m,\R^n))).
\end{equation}

%\ifn{I'm not sure that we need that this is actually a quasi-iso of commutative graded differential algebras (CDGAs).  It's a nice feature and it's what rational homotopy theory works with.  The problem is that the product in $\D(m)$ wants to just superimpose the segment vertices (we encountered this issue before in \eqref{E:SimplicialDiagamsFaces}, so it's not the shuffle product. But again, in this paper we don't really need the shuffle product, so if we declare the product to be the superposition, then we can state the above as a quasi-iso of CDGAs. This would also make $\D(m)_\bullet$ into a simplicial CDGA, and this should be noted above.  We also might want to note earlier that $\Omega^*(\Conf(m,\R^n)^\bullet)$ is a CDGA (the deRham algebra always is).}
%The differential on the left is the usual de Rham one while the simplicial cohomology algebra on the right is taken with the zero differential.
\end{thm}

%\ifn{This formality descends to $\Q$, but I don't know that we need that.  But it might be useful that we can make everything rational...}
%
%

%%%%%%%%%%%%%%%%%%%%%%%%%%%%%%%%%%%%%%%%%%%%%%%%%%%%%%%%%%%%%%%%%%%%%%%%%%%%%%%%%%%%%%%%%%%%%%%%%%%%%%%%

%%%%%%%%%%%%%%%%%%%%%%%%%%%%%%%%%%%%%%%%%%%%%%%%%%%%%%%%%%%%%%%%%%%%%%%%%%%%%%%%%%%%%%%%%%%%%%%%%%%%%%%%

Each of the maps in \eqref{E:BarZigzag} will have certain consequences. The first is the following:
\begin{thm}\label{T:DiagramsBraidCohomology}
For $n\geq 3$ and $m\geq 1$, there exists a quasi-isomorphism (over $\R$)
$$
\varint_{\mathrm{Chen}} \circ \varint_{\mathrm{formality}}\colon\Ba(\D(m)) \stackrel{\simeq}{\longrightarrow} \Ch^{\ast}(\Omega\Conf(m,\R^n)).
$$
Thus for all $i\geq 0$, there are real 
isomorphisms
\begin{equation}\label{E:TotBraidsCohomIso}
\Ho^i(\Ba(\D(m)))\cong\Ho^i(\Omega\Conf(m,\R^n)).
\end{equation}
\end{thm}

\begin{proof}
Recall from \eqref{E:QuasiisoComp} that, as long as $\Conf(m,\R^n)$ is simply connected, which is the case for $n\geq 3$, we have a quasi-isomorphism
\begin{equation}\label{E:ChenMainProof}
\varint_{\mathrm{Chen}}\colon \Ba(\ChdR^{\ast}(\Conf(m,\R^n))\stackrel{\simeq}{\longrightarrow} \Ch^\ast(\Omega\Conf(m,\R^n)).
\end{equation}
Combining this with the formality integration quasi-isomorphism in \refT{BraidsFormality}, namely
$$
 \varint_{\mathrm{formality}}=B(I) \colon \Ba(\D(m))\stackrel{\simeq}{\longrightarrow} \Ba(\ChdR^{\ast}(\Conf(m,\R^n))),
$$
we have the desired result.
\end{proof}

In Section \ref{S:Chen}, we will justify why \eqref{E:ChenMainProof} can indeed be understood as the Chen iterated integral map, as indicated in the Introduction.  Intuitively, formality takes care of integration over diagrams with free vertices.  What remains are chord diagrams, and integration over those can be seen as Chen's iterated integrals.

We can also apply \refC{CDGAModelForOmegaX} to this setup and immediately deduce the following:  

\begin{cor}\label{C:CDGAModelForBraids} 
For $n\geq 3$, the cohomology of $\Omega\Conf(m,\R^n)$ over any field of characteristic zero is the cohomology of the CDGA
\begin{equation}\label{E:CDGAModel}
\Ba(\Ho^{\ast}(\Conf(m,\R^n)))=
\Big(      \bigoplus\limits_{p=0}^{\infty}   s^{-p} (\Ho^*(\Conf(m, \R^n))^{\otimes p})
\Big/
\sum_{i=0}^{p-1} \operatorname{im}((s^i)^*),
 \,   
              \sum_{i=0}^{p+1}(-1)^i (d^i)^*
 \Big),
 \end{equation}
 where $s^{-p}$ denotes the degree shift.
\end{cor}

Let us elaborate on the structure of the cohomology of \eqref{E:CDGAModel}.  It will turn out to be conveniently described in terms of a certain subspace of $\D(m)$ of \emph{braid chord diagrams} (\refD{BCD}).

\begin{defin}\label{D:BraidDiagram}
A \emph{braid diagram} is an element 
$$
[\Gamma_1 | \cdots | \Gamma_p]\in \D(m)^{\otimes p}
$$
such that each $\Gamma_i$ is a diagram.
\end{defin}

The following two pictures represent braid diagrams in $\Ba(\D(3))$, with $p=3$ and $p=2$ respectively.  The horizontal segments are strictly not part of the data, but they can be drawn in a canonical manner and suggest a relationship to braids on 3 strands.
\[
\raisebox{-1.4pc}{\includegraphics[scale=0.25]{LLM.eps}} \qquad \qquad
\raisebox{-1.4pc}{\includegraphics[scale=0.25]{LT.eps}} 
\]

Now since $\overline{I}: \D(m) \to \Ho^*(\Conf(m, \R^n))$ is surjective, we can represent the CDGA \eqref{E:CDGAModel} by the images of braid diagrams under $\overline{I}$.  From the definition of $\overline{I}$, each class in $\Ho^*(\Conf(m, \R^n))$ is thus represented by a product of chords, using the product on $\D(m)$.   
Thus an element of $\Ba(\Ho^*(\Conf(m, \R^n)))$ is a linear combination of terms $[\Gamma_1 | \cdots | \Gamma_p]$ where each $\Gamma_i$ has no free vertices.  
(Note however that such linear combinations are generally not closed in $\D(m)$ because the 3-term relation does not hold there.)

With the combinatorics of $\D(m)$ in hand, it is easier to think about the normalization module and the differential in \eqref{E:CDGAModel}.  First, this allows us to establish a vanishing line in the double complex as follows.

\begin{prop}\label{P:VanishingLine}
The double complex $\Ba^{*,*}(\Ho^{*}(\Conf(m, \R^n)))$  has a lower vanishing line of slope $1-n$, i.e.~everything below this line is zero. (Here we picture the double complex in the second quadrant as usual.)
\end{prop}
\begin{proof}
Consider the $i$th codegeneracy
\begin{align*}
s^i\colon \Conf(m, \R^{n})^p &  \longrightarrow \Conf(m, \R^{n})^{p-1},\ \ \ \ 0\leq i\leq p-1 \\
(a_1, a_2, ..., a_p) & \longmapsto (a_1, a_2, ..., a_{i}, a_{i+2}, ..., a_p).
\end{align*}
Since $a_{i+1}$ does not appear in the image of $s^i$, it follows that a braid diagram representing an element of  $\Ho^*(\Conf(m, \R^{n})^p)$ is in the image of 
$$(s^i)^*\colon \Ho^*(\Conf(m, \R^{n})^{p-1})\longrightarrow  \Ho^*(\Conf(m, \R^{n})^p)$$
if and only if it has the empty diagram in the $(i+1)$st entry.
The normalization module $\sum_{i=0}^{p-1} \operatorname{im}((s^i)^*)$ is thus generated by all braid diagrams which have the empty diagram in at least one entry.  It then follows that any term in the double complex $\Ba^{*,*}(\Ho^{*}(\Conf(m, \R^n)))$ can be represented by braid diagrams which have at least one chord in every entry.  This is only possible if $q\geq p(n-1)$.
In other words, $\Ho^{-p,q}(\Conf(m, \R^n)^\bullet)$ has a lower vanishing line of slope $1-n$.
\end{proof}

The double complex also has an upper vanishing line.  Namely, we see from the Poincar\'e polynomial \eqref{E:PoincarePoly} that the top cohomology group of  $\Conf(m, \R^{n})$ occurs in degree $(m-1)(n-1)$, and thus the double complex vanishes for $q>p(m-1)(n-1)$.
However, we can do even better.  \refP{NoCohomAboveVanishLine}, as well as \refT{BraidCohIsCDMod4T} below, appear in various guises in earlier literature; see Remark \ref{R:RelToEarlierWork}.

\begin{prop}\label{P:NoCohomAboveVanishLine}
All of the cohomology of $\Omega\Conf(m,\R^n)$ is concentrated on the vanishing line $q=(1-n)p$ of the double complex $\Ba^{*,*}(\Ho^{*}(\Conf(m, \R^n)))$.
\end{prop}

\begin{proof}
It is convenient to rephrase this result in the following terms.  Let $A:=\Ho^{*}(\Conf(m, \R^n))$.
For $\alpha = \alpha_{i_1j_1} \cdots \alpha_{i_kj_k} \in A$, define the \emph{weight} of $\alpha$ as $w(\alpha)=k$.
For an element $\alpha = [\alpha_1 | \cdots | \alpha_p]$ in $\Ba(A)$, define the \emph{weight} $w(\alpha) := \sum_{i=1}^p w(\alpha_i)$.  
As we explain below, a theorem of Priddy \cite[Theorem 5.3]{Priddy:Koszul} implies that for our $A$, the cohomology in $\Ba(A)$, is spanned by terms $[\alpha_{i_1j_1} | \cdots | \alpha_{i_pj_p}]$ whose weight is equal to the homological degree $p$.\footnote{For an algebra $A$ with quadratic relations, this condition on $\Ba(A)$ is equivalent  to $A$ being a Koszul algebra.  The terminology ``weight'' is used by Loday and Vallette \cite{LV:AlgOperads}, though in \cite{Priddy:Koszul} Priddy uses the term ``length'' instead.}  
Since each $\alpha_{ij}$ has degree $n-1$, this is equivalent to the result stated above.

The rough idea behind Priddy's theorem is that an element $\beta=[ \ldots |ab| \ldots]$ in $\Ba(A)$ looks like a term in $d[\ldots | a | b | \ldots]$.  Of course, the bar differential has multiple terms, so showing the triviality of $\beta$ in homology involves some subtlety.  However, with a certain kind of basis for $A$, namely a Poincar\'e--Birkhoff--Witt (PBW) basis, Priddy shows that one can appropriately filter the complex to establish the result.
In the setting of configuration spaces, the relations \eqref{E:CohomCond2}, \eqref{E:CohomCond4}, and \eqref{E:CohomCond5}
among the $\alpha_{ij}$ can be used to show that the set of products $\alpha_{i_1 j_1} \cdots \alpha_{i_k j_k}$ with $i_\ell < j_\ell$ for all $\ell=1,...,k$ and with $i_{\ell -1} < i_\ell$ for all $\ell=2,...,n$ is a PBW basis for $\Ho^*(\Conf(m, \R^n))$.
This PBW basis appears in work of Berglund \cite[Example 32]{B:Koszul} but goes back at least as far back as work of Bezrukavnikov \cite[Section 2]{Bezrukavnikov:Koszul-Config}.
\end{proof}

The Koszul property of $\Ho^*(\Conf(m,\R^n))$ from \refP{NoCohomAboveVanishLine} motivates the following definition.

\begin{defin}\label{D:BCD}
Let $\Gamma_{ij}$ denote the diagram in $\D(m)$ with a single chord between segment vertices $i\neq j$, oriented from $i$ to $j$ if $n$ is odd.  Note that $\Gamma_{ji}=(-1)^n\Gamma_{ij}$.
A \emph{braid chord diagram} is a braid diagram of the form $[\Gamma_{i_1j_1} | \cdots | \Gamma_{i_pj_p}] \in \Ba(\D(m))$, i.e.~the diagram in each entry is a single chord.
\end{defin}

The proof of \refP{NoCohomAboveVanishLine} shows that every cohomology class can be viewed (via the map  $B(\overline{I})$ from $\Ba(\D(m))$) as a linear combination of braid chord diagrams.  
To fully describe this cohomology, it remains to describe the subspace of \emph{closed} linear combinations of braid chord diagrams. 
That is, we just have to compute the kernel of the differential 
$\sum_{i=0}^{p+1}(-1)^i (d^i)^*$ in $\Ba(\Ho^*(\Conf(m,\R^n))$ \eqref{E:CDGAModel}.  

\begin{defin}\label{D:BCD(m)}
Let $\BCD(m)$ be the dual to the space of braid chord diagrams defined above.  More concretely, consider the basis of diagrams for $\D(m)$, which is canonical up to sign.  Let $B_{ij}$ ($i \neq j$) be the element in the dual space $\D(m)^*$ which is dual to the diagram $\Gamma_{ij}$ with a single chord.  
Then $\BCD(m)$ is spanned by elements $[B_{i_1j_1} | \cdots | B_{i_pj_p}]$ dual to the elements $[\Gamma_{i_1j_1} | \cdots | \Gamma_{i_pj_p}]$.
\end{defin}

In the literature, the elements of $\BCD(m)$ are sometimes called \emph{horizontal chord diagrams} and appear in finite type knot theory.  They correspond to the homology of spaces of braids, which is why we define $\BCD(m)$ so that our diagrams $\Gamma_{ij}$ live in its dual.

\begin{thm}\label{T:BraidCohIsCDMod4T}
For $n\geq 3$, the cohomology of $\Omega \Conf(m,\R^n)$ is concentrated in degrees that are multiples of $n-2$.
For any $p\geq 0$, $\Ho^{(n-2)p}(\Omega \Conf(m,\R^n); \R)$ is isomorphic to the space of linear combinations $\gamma$ of braid chord diagrams $[\Gamma_{i_1 j_1} | \cdots | \Gamma_{i_p j_p}]$ satisfying the following conditions.
\begin{itemize}
\item
\emph{4T relations}: $\gamma$ yields zero when paired with 
\[
[\cdots | B_{ij} | B_{jk} | \cdots]
- (-1)^{n} [\cdots | B_{jk} | B_{ij} | \cdots]
 + [\cdots | B_{ij} | B_{ik} | \cdots]
- (-1)^n[\cdots | B_{ik} | B_{ij} | \cdots]  
\]
where $i,j,k$ are distinct.
\item
\emph{Shuffle relations}:  $\gamma$ yields zero when paired with 
\[
[ \cdots | B_{ij} | B_{k\ell} | \cdots]
- (-1)^{n} [\cdots | B_{k\ell} | B_{ij} | \cdots]
\]
where $i,j,k,\ell$ are distinct.
\end{itemize}
\end{thm}

\begin{rems}
\label{R:RelToEarlierWork}
 \ 
\begin{enumerate}
\item
As mentioned in the Introduction, the conditions in \refT{BraidCohIsCDMod4T} are linear dual to the Yang--Baxter Lie relations
\begin{align}
\label{E:Yang-Baxter1}
[B_{ij}, B_{jk} + B_{ik}]&=0 \text{ for } i,j,k,\ell \text{ distinct} \\
\label{E:Yang-Baxter2}
[B_{ij}, B_{k\ell}]&=0 \text{ for } i,j,k,\ell \text{ distinct}
\end{align}
where for classes $B,B'$ of degree $n-2 \equiv n \mod 2$ as above, one defines $[B,B']:=BB' - (-1)^n B'B$.
These appear in Cohen and Gitler's calculation \cite{CG} of the \emph{homology} of the space of braids.  
Since the cohomology of the bar complex computing $\Ho^*(\Omega \Conf(m,\R^n))$ is also the Koszul dual coalgebra of $\Ho^*(\Conf(m, \R^n))$ \cite[Chapter 3]{LV:AlgOperads}, the {homology} $\Ho_*(\Omega \Conf(m,\R^n))$ is its Koszul dual {algebra}, and the Yang--Baxter Lie relations are just the orthogonal complement to the relations \eqref{E:CohomCond2}, \eqref{E:CohomCond3}, and \eqref{E:CohomCond5}, where $B_{ij}$ is dual to $\alpha_{ij}$.  This orthogonality is essentially what we check below.
This proof is indicated in \cite[Example 32]{B:Koszul} and \cite[Section 4]{Bezrukavnikov:Koszul-Config}.
\item
%\label{R:KohnoIntegration}
As also mentioned in the Introduction, versions of Theorems \ref{T:DiagramsBraidCohomology} and \ref{T:BraidCohIsCDMod4T} are familiar from work of Kohno.  Namely, our Theorem \ref{T:DiagramsBraidCohomology} is a refinement of his \cite[Theorem 5.1]{Kohno:LoopsFiniteType} in that it explicitly describes his ``iterated integrals of length $<p$'' as integrals over diagrams with free vertices.  In addition, our Theorem \ref{T:BraidCohIsCDMod4T} is \cite[Theorem 4.1]{Kohno:LoopsFiniteType}, although Kohno's main input appears to be the Cohen--Gitler result \cite{CG} while we use just the bar complex \cite{Priddy:Koszul}.
\end{enumerate}
\end{rems}

\begin{proof}[Proof of \refT{BraidCohIsCDMod4T}] 
Each $\Gamma_{ij}$ has degree $(n-1)$, so the total degree of a braid diagram 
$[\Gamma_{i_1j_1} | \cdots | \Gamma_{i_pj_p}]$ is $p(n-1)-p = p(n-2)$.  
Thus the cohomology of $\Omega\Conf(m,\R^n)$ is indeed concentrated in degrees that are multiples of $n-2$.

For a combination $\gamma$ of diagrams to be closed, each $(d^k)^*$ must vanish on it, since $(d^k)^*(\gamma)$ has a product of chords in only the $k$th factor.  The effect of $(d^k)^*$ is to contract two chords between vertices $i<j$ and $i'<j'$.  If $i=i'$ and $j=j'$, then the result is zero by \eqref{E:CohomCond2}.  If $i,j,i',j'$ are all distinct, then by \eqref{E:CohomCond4}, the result can be canceled if and only if there is a matching term with the order of the two chords swapped.  This yields the shuffle relations, with the claimed sign.

If $\{i,j,i',j'\}$ are three distinct indices, then the contraction can be canceled either by one other term via \eqref{E:CohomCond4} or two other terms via the 3T relation \eqref{E:CohomCond5}.  The former accounts for the second term in the 4T relation.  The latter accounts for the third and fourth terms in the 4T expression, keeping in mind relations \eqref{E:CohomCond3} and \eqref{E:CohomCond4}.  
The agreement of the signs of the first and third terms in the 4T relation can be seen by rewriting the 3T relation as $\alpha_{ij}\alpha_{jk} - \alpha_{ik}\alpha_{jk} - \alpha_{ij}\alpha_{ik}=0$.  
(The first and second terms in 3T lead to $[B_{ij} + B_{ik}, B_{jk}]=0$, while the first and third terms lead to $[B_{ij}, B_{jk} + B_{ik}]=0$, but the former relation can be rewritten in the form of the latter, by swapping all the indices.)
\end{proof}

%%%%%%%%%%%%%%%%%%%%%%%%%%%%%%%%%%%%%%%%%%%%%%%%%%%%%%%%%%%%%%%%%%%%%%%%%%%%%%%%%%%%%%%%%%%%%%%%%%%%%%%%

\section{Chen iterated integrals}\label{S:Chen}

%%%%%%%%%%%%%%%%%%%%%%%%%%%%%%%%%%%%%%%%%%%%%%%%%%%%%%%%%%%%%%%%%%%%%%%%%%%%%%%%%%%%%%%%%%%%%%%%%%%%%%%%

Recall that one of our main results, \refT{DiagramsBraidCohomology}, had  \eqref{E:QuasiisoComp} as one of the main inputs:  For a simply-connected space $X$, there is a quasi-isomorphism
$$
\Ba(\Ch^*(X))\stackrel{\simeq}{\longrightarrow} \Ch^*(\Omega X).
$$
For $X=\Conf(m,\R^n)$, $n\geq 3$, this turned into \eqref{E:ChenMainProof}, and was used in the proof of \refT{DiagramsBraidCohomology}:
$$
\Ba(\ChdR^{\ast}(\Conf(m,\R^n))\stackrel{\simeq}{\longrightarrow} \Ch^\ast(\Omega\Conf(m,\R^n)).
$$
The claim made there and in the Introduction was that this map is really $\varint_{\mathrm{Chen}}$, the Chen iterated integral.  

The goal of this section is to make that relationship explicit.  In particular, we will explain how Chen integrals are related to the map $\rho$ from \eqref{E:rho-def}, as this is the main ingredient in the quasi-isomorphism \eqref{E:QuasiisoComp}.

In Section \ref{S:ChenLoops}, we first review the set up of Chen interated integrals  and Chen's version of the de Rham loop space theorem \cite{Chen:Itr-Integrals, Chen:Path-Integrals}. More details can be found, for example, in the monograph \cite{Hain:Integrals} by Hain or in Kohno's work, e.g.~\cite{Kohno:VassilievBraids}.
In \refS{Chen'sTheorem}, we then relate $\rho$ to the Chen iterated integrals and deduce Chen's Iterated Integral Theorem giving the cohomology of the loop space.

%%%%%%%%%%%%%%%%%%%%%%%%%%%%%%%%%%%%%%%%%%%%%%%%%%%%%%%%%%%%%%%%%%%%%%%%%%%%%%%%%%%%%%%%%%%%%%%%%%%%%%%%

\subsection{Chen iterated integrals and the cohomology of loop spaces}\label{S:ChenLoops}

%%%%%%%%%%%%%%%%%%%%%%%%%%%%%%%%%%%%%%%%%%%%%%%%%%%%%%%%%%%%%%%%%%%%%%%%%%%%%%%%%%%%%%%%%%%%%%%%%%%%%%%%

As before we denote by $\Delta^p$ the standard $p$-simplex, $X$ will now be a based smooth finite-dimensional manifold, and $\Omega X$ the based loop space as usual.  Recall from \eqref{E:Evaluation} that for each $p>0$, there is an evaluation map
\begin{equation}\label{E:ChenEvaluation}
\begin{split}
\ev_p\colon  \Delta^p\times \Omega  X & \longrightarrow X^p\\
(t_1, ..., t_p, \gamma) & \longmapsto (\gamma(t_1), ..., \gamma(t_p)). 
\end{split}
\end{equation}
Given smooth forms $\omega_1$, ..., $\omega_p$ on $X$, one can use $\ev_p$ to pull back their product to $\Delta^p\times \Omega X$, there by obtaining the form
\begin{equation}\label{E:omega_1...omega_p}
\omega_1\omega_2\cdots\omega_p=\ev^\ast_p(\pi^\ast_1\omega_1\wedge \cdots\wedge\pi^\ast_p\omega_p),
\end{equation}
where $\pi_k:X^p\to X$ is the projection onto the $k$th factor of $X^p$. This form can be pushed forward to $\Omega X$, i.e.~integrated along the fiber of the projection 
$\pr\colon \Delta^p\times \Omega X\longrightarrow \Omega X$.
The fiber is just $\Delta^p$, and we thus obtain the following differential form on $\Omega X$:
\begin{equation}\label{E:ChenForm}
\varint \omega_1 \omega_2 \cdots \omega_p:=\pr_*\bigl(\omega_1 \omega_2 \cdots \omega_p\bigr)=\int_{\Delta^p}\omega_1 \omega_2 \cdots \omega_p.
\end{equation}
The degree of this form is $q-p$, where $q=\sum^p_{k=1} |\omega_k|$ is the sum of the degrees of the forms $\omega_k$. In Chen's terminology \cite{Chen:Itr-Integrals, Chen:Path-Integrals}, $\varint \omega_1 \omega_2 \cdots \omega_p$ is an \emph{iterated integral}.

Now let $\Chen^{-p,q}(X)$
be the $\R$-vector space spanned by iterated integrals $\varint \omega_1 \omega_2 \cdots \omega_p$ of degree $q$ and set
\[
\Chen^{*,*}(X)=\bigoplus_{p,q\geq 0} \Chen^{-p,q}(X).
\]
This is the \emph{Chen complex of $X$} and is a subcomplex of $\ChdR^\ast(\Omega X)$, the CDGA of smooth differential forms on $\Omega X$.\footnote{Since $\Omega X$ is technically not a manifold, for the purpose of defining $\ChdR^\ast(\Omega X)$, Chen  introduced a notion of a {\em differentiable space} \cite{Chen:Itr-Integrals, Chen:Path-Integrals}.}
The differential on $\Chen^{*,*}(X)$ can be computed via the Stokes' Theorem as
\[
d\varint \omega_1\omega_2\cdots\omega_p=d\Bigl(\int_{\Delta^p}\omega_1\omega_2\cdots\omega_p \Bigr)= 
\int_{\Delta^p} d \bigl(\omega_1\omega_2\cdots\omega_p\bigr) 
+\int_{\partial\Delta^p}\omega_1\omega_2\cdots\omega_p.
\]
Expanding this expression, in the iterated integral notation we obtain
\begin{equation}\label{E:ChenDifferential}
\begin{split}
d\bigl(\varint \omega_1\cdots \omega_p\bigr) & =\sum^p_{i=1} (-1)^i \varint \omega_1\cdots d \omega_i\cdots \omega_p\\
& \qquad +\sum^{p-1}_{i=1} (-1)^{\sum^i_{k=1} (|\omega_k|+1)} \varint \omega_1\cdots (\omega_i\wedge \omega_{i+1})\cdots \omega_p.
\end{split}
\end{equation}

Notice that, if $\omega_k$ are closed, then the first summand on the right side of the above equation is zero.  This case occured in the setting of the braid diagram complex in \refS{CobarBraidCohomology}, where $X$ was the configuration space. 
In fact,  $\Chen^{*,*}(X)$ has the structure of a double cochain complex,  i.e. we have differentials
\[
d_1\colon \Chen^{-p,q}(X)\longrightarrow \Chen^{-p,q+1}(X),\qquad
d_2\colon \Chen^{-p,q}(X)\longrightarrow \Chen^{-p+1,q}(X),
\]
given by the two summands of the differential in \eqref{E:ChenDifferential}.
Thus we obtain the total complex associated to this double complex, denoted by $\Tot \Chen^{*,*}(X)$ (in the literature, shorter notation $\Chen^\ast(X)$ is often used); this is the cochain complex which is in degree $k$ given by 
\[
\Tot \Chen^{*,*}(X)^k=\bigoplus_{k=q-p} \Chen^{-p,q}(X)
\] 
with the total differential 
\begin{equation}\label{E:TotalDiffChen}
d\colon \Chen^{-p,q}(X) \longrightarrow \Chen^{-p+1,q}(X)\oplus \Chen^{-p,q+1}(X)
\end{equation}
defined as $d=d_1+ d_2$. Clearly, the above construction is reminiscent of the bar complex \eqref{E:TotCochainsCosimplicial} and its differential $\delta$ defined in \eqref{E:TotalDiff}.
In fact, directly from \eqref{E:BarTotalDiff} and \eqref{E:BarQuasiTot}, we have a CDGA map which we call the {\em Chen integration map}:
\begin{equation}\label{E:ChensintegrationMap}
\begin{split}
\varint_{\mathrm{Chen}}\colon\Ba(\ChdR^{\ast}(X)) & \longrightarrow \Chen^{*,*}(X)\subset \Ch^\ast(\Omega X),\\
[\omega_1 | \omega_2 | \cdots | \omega_p] & \longrightarrow \varint \omega_1 \omega_2 \cdots \omega_p.
\end{split}
\end{equation}
For a simply connected manifold $X$, the map $\varint_{\mathrm{Chen}}$ to $\Chen^{*,*}(X)$  is an isomorphism of CDGAs \cite{Chen:Itr-Integrals, Chen:Path-Integrals}.  In the next section, we show that it induces a quasi-isomorphism to $\Ch^\ast(\Omega X)$, which is enough for our purposes.

%%%%%%%%%%%%%%%%%%%%%%%%%%%%%%%%%%%%%%%%%%%%%%%%%%%%%%%%%%%%%%%%%%%%%%%%%%%%%%%%%%%%%%%%%%%%%%%%%%%%%%%%

\subsection{Chen's Theorem}\label{S:Chen'sTheorem}

%%%%%%%%%%%%%%%%%%%%%%%%%%%%%%%%%%%%%%%%%%%%%%%%%%%%%%%%%%%%%%%%%%%%%%%%%%%%%%%%%%%%%%%%%%%%%%%%%%%%%%%%

The main contribution of this subsection is to relate the Chen's integration map $\varint_{\mathrm{Chen}}$ to the map 
$\rho$ defined in \eqref{E:rho-def} and, as a consequence of results in Section \ref{S:CobarCohomology}, obtain Chen's Theorem:

\begin{thm}[\cite{Chen:Itr-Integrals, Chen:Path-Integrals}]\label{T:Chen}
	For a simply connected manifold $X$, the Chen's integration map $\varint_{\mathrm{Chen}}$ is a quasi-isomorphism. Thus, the cohomology of the complex  $\Tot \Chen^{*,*}(X)$
	is isomorphic to the singular cohomology of the loop space $\Omega X$ with real coefficients, i.e.
	\begin{equation}\label{eq:Chen-isomorphism}
	\Ho^k(\Tot \Chen^{*,*}(X))\cong \Ho^k(\Omega X;\R), \ \ \ k\geq 0.
	\end{equation}
\end{thm}
\begin{rem}\label{R:CodegeneraciesCohomology}
	One actually obtains an isomorphism of Hopf algebras; see Theorem 3.1 in \cite{Hain:Integrals}.
\end{rem}

\begin{proof}
	
	Denote as usual by $\Ch^*$ and $\Ch_*$ smooth singular chains and cochains, and by $\ChdR^\ast$ the de Rham cochains. From \eqref{E:TotAndCochains} and \eqref{E:rho-def}, 
	the map
	\[
	\rho:\Tot \Ch^{*,*}(X^\bullet) \longrightarrow \Ch^*\Tot (X^\bullet)
	\]
	is the dual of the natural map
	\[
	\Ch_q(\Tot (X^\bullet))\xrightarrow{\iota_p\times\cdot} \Ch_{p+q}(\Delta^p\times \Tot (X^\bullet))\xrightarrow{(\ev_p)_\ast} \Ch_{p+q}(X^p)
	\]
	on chains.  Here
 $\iota_p:\Delta^p\longrightarrow\Delta^p$ is the standard simplex of \eqref{E:rho-def}, and  $\iota_p\times\cdot$ is defined on any singular simplex $\sigma:\Delta^q\longrightarrow \Tot (X^\bullet)$ as
	\begin{equation}
	(\iota_p\times\cdot)(\sigma)=\iota_p\times \sigma,
	\end{equation}
	i.e. $\iota_p\times \sigma$ represents the singular chain
	\[
	\iota_p\times \sigma:\Delta^p\times \Delta^q\longrightarrow \Delta^p\times \Tot (X^\bullet), \qquad (\iota_p\times \sigma)(s,t)=(\iota_p(s),\sigma(t))
	\]  
	after a suitable subdivision of $\Delta^p \times \Delta^q$.  Map $(\ev_p)_\ast$ is the induced homomorphism via the evaluation map \eqref{E:ChenEvaluation}.
	Given $\tau\in \Ch^{p+q}(X^p)$, $\rho(\tau)$ evaluates on chains in $\Ch_q(\Tot (X^\bullet))$ as 
	\begin{equation}\label{E:rho-eval}
	\rho(\tau)(\sigma)=\tau\bigl((\ev_p)_\ast(\iota_p\times \sigma)\bigr).
	\end{equation}
	Next, we can replace singular cochains $\Ch^{p+q}(X^p)$ in \eqref{E:rho-def} with de Rham cochains $\ChdR^{p+q}(X^p)$. Recall the  K\"{u}nneth Theorem isomorphism for differential forms (see, for example, \cite{Bott-Tu:DiffForms})
	\[
	\begin{split}
	\bigoplus_{i_1+i_2+\cdots+i_p=p+q,\ i_k\geq 0} \ChdR^{i_1} (X)\otimes\cdots\otimes\ChdR^{i_p} (X) & \stackrel{\cong}{\longrightarrow} \ChdR^{p+q}(X^p),\\
	\omega_1\otimes\cdots\otimes \omega_p & \longrightarrow \pi^\ast_1 \omega_1\wedge \pi^\ast_2 \omega_2\wedge 
	\cdots \wedge \pi^\ast_p \omega_p,
	\end{split}
	\] 
	where  $\pi_i\colon X^p\to X$ is the projection onto the $i$th factor of $X^p$. Using the standard de Rham isomorphism induced by the integral duality (again see \cite{Bott-Tu:DiffForms}),
	\[
	\begin{split}
	\ChdR^{p+q}(X^p)\times \Ch_{p+q}(X^p) & \longrightarrow \R,\\
	(\alpha,\sigma) & \longrightarrow \int_{\sigma} \alpha=\int_{\Delta^{p+q}} \sigma^\ast\alpha,
	\end{split}
	\]
	and $\tau=\pi^\ast_1 \omega_1\wedge \pi^\ast_2 \omega_2\wedge 
	\cdots \wedge \pi^\ast_p \omega_p$, we can now compute \eqref{E:rho-eval} as follows:
	\[
	\begin{split}
	\rho(\tau)(\sigma)=\tau\bigl((\ev_p)_\ast(\iota_p\times \sigma)\bigr) & =\int_{\Delta^p\times\Delta^q} (\iota_p\times \sigma)^\ast\bigl(\ev^\ast_p(\pi^\ast_1\omega_1\wedge \cdots\wedge\pi^\ast_p\omega_p)\bigr)\\
	& =\int_{\Delta^q} \int_{\Delta^p} \sigma^\ast\bigl(\ev^\ast_p(\pi^\ast_1 \omega_1\wedge 
	\cdots \wedge \pi^\ast_p \omega_p)\bigr)= \int_{\Delta^q} \sigma^\ast\bigl(\varint \omega_1\omega_2\cdots\omega_p\bigr).
	\end{split}
	\]
	Here we used Fubini's Theorem, along with \eqref{E:omega_1...omega_p} and \eqref{E:ChenForm}. The inclusion  $j\colon \ChdR^{\ast}(X^\bullet)\longrightarrow \Ch^{\ast}(X^\bullet)$ is a quasi-isomorphism (by the standard de Rham isomorphism) and thus induces 
	an isomorphism $\Ba(\ChdR^{\ast}(X))\cong\Tot \ChdR^{*,*}(X^\bullet)\cong \Tot \Ch^{*,*}(X^\bullet)$.  The above calculation thus shows 
	\[
	\rho\circ j=\varint_{\mathrm{Chen}}.
	\]
	Now the desired claim follows from Theorem \ref{T:CommuteTotAndCochains}.
\end{proof}

%%%%%%%%%%%%%%%%%%%%%%%%%%%%%%%%%%%%%%%%%%%%%%%%%%%%%%%%%%%%%%%%%%%%%%%%%%%%%%%%%%%%%%%%%%%%%%%%%%%%%%%%

\section{Relation to Bott-Taubes configuration space integrals}
\label{S:B-TRelation}

%%%%%%%%%%%%%%%%%%%%%%%%%%%%%%%%%%%%%%%%%%%%%%%%%%%%%%%%%%%%%%%%%%%%%%%%%%%%%%%%%%%%%%%%%%%%%%%%%%%%%%%%

In this section, we establish some of our main results, relating configuration space integrals for long links
%(see work of Bott and Taubes \cite{Bott-Taubes}, Cattaneo, Cotta-Ramusino, and Longoni \cite{CCRL-AGT}, and Munson and the last two authors \cite{KMV:FTHoLinks})
to integrals for braids.  The latter type of integrals are Chen's iterated integrals composed with the formality configuration space integrals, i.e.~the bottom map in the square \eqref{E:MainCommutativeSquare}.  This section discusses the other maps in that diagram and establishes its commutativity.

In Sections \ref{S:ReviewOfLD} and \ref{S:ReviewOfLDIntegrals}, we review $\LD(m)$, the diagram complex for configuration space integrals for long links.  In Section \ref{S:OrderDefectForBraidDiagrams}, we introduce a new bigrading on the bar complex on diagrams for braids $\Ba(\D(m))$, which will be compatible with the bigrading on $\LD(m)$.  In Section \ref{S:MappingLDtoBD}, we define a map from $\LD(m)$  to $\Ba(\D(m))$.  Finally, in Section \ref{S:Integrals} we use this map of diagram spaces as well as the restriction from link cohomology to braid cohomology to compare these two types of integrals show that they are compatible.

\begin{figure}[ht]
	\centering
	\includegraphics[width=0.5\linewidth]{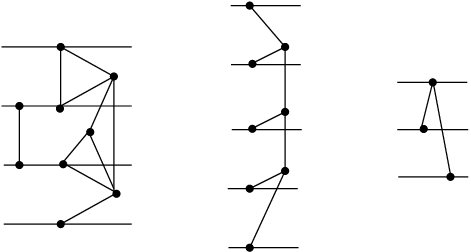}
	\qquad \qquad \raisebox{1pc}{\includegraphics[width=0.25\linewidth]{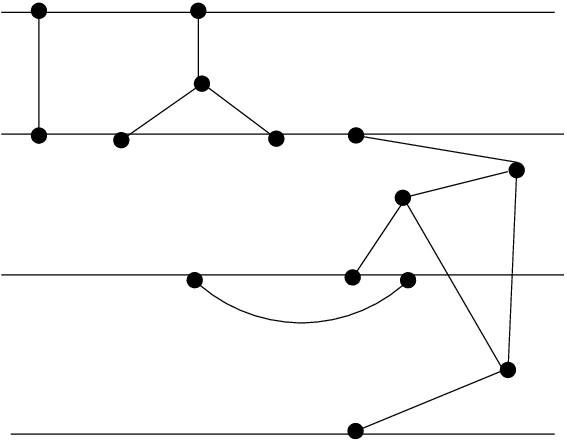}}
	\caption{Example diagrams in $\LD(m)$ for $m=4,5,3$, and $4$ respectively.}
	\label{fig:hodiagramexamples}
\end{figure}

\subsection{Review of the complex $\LD(m)$}
\label{S:ReviewOfLD}

We now review the cochain complex $\LD(m)$ of diagrams for long links of $m$ components.  Details can be found in \cite[Section 3]{KMV:FTHoLinks}.  Such diagrams differ only slightly from braid diagrams in the bar complex $B(\D(m))$ (Definition \ref{D:Diagrams}).  The ambient dimension for our links is ${n+1}$, since a loop in $\Conf(m, \R^n)$ may be viewed as such a long link.

\begin{defin}
\label{D:LinkDiagrams}
Fix $n\geq 2$ and $m\geq 1$.  A \emph{link diagram $\Gamma$ (in $\R^{n+1}$ on $m$ vertices)} consists of 

\begin{itemize}
\item
$m$ oriented intervals called \emph{segments};
\item
a set of vertices, partitioned into \emph{segment vertices} and \emph{free vertices}.  Thus
$$V(\Gamma) = V_{\mathrm{seg}}(\Gamma) \sqcup V_{\mathrm{free}}(\Gamma);$$
\item
a set of edges $E(\Gamma)$ joining vertices of $\Gamma$, where an edge between two distinct segment vertices is called a \emph{chord};
\item
a set of \emph{arcs}, where an \emph{arc} is a part of a segment between two vertices (and an arc is not considered an edge).
\end{itemize}

We require that each free vertex has valence at least 3 and each segment vertex has valence at least 1, where only edges (not arcs) count towards valence.
%We impose the following conditions:
%\begin{itemize}
%\item
%Each free vertex of $\Gamma$ has valence at least 3.
%%\item
%%There are no self-loops in $\Gamma$, i.e., no edge has the same vertex as both of its endpoints.
%\end{itemize}
\end{defin}

By a \emph{component} of $\Gamma$ we mean a connected component in the diagram resulting from forgetting the $m$ segments.  We will also use the notion of \emph{grafts} of $\Gamma$, the precise definition of which is given in \cite[Definitions 4.9 and 4.10]{KMV:FTHoLinks}.  Roughly, one first replaces each segment vertex by the set of its neighbors; the resulting components are the grafts.  One may think of this process as blowing up segment vertices so that they become univalent (possibly disconnecting components of $\Gamma$).

\begin{defin}
\label{D:LDOrientation}
An \emph{orientation} of a link diagram $\Gamma$ consists of the following data:
\begin{itemize}
\item
for $n+1$ odd, an ordering of $V(\Gamma)$ (i.e.~a labeling by a totally ordered set), and an orientation of each edge;
\item 
for $n+1$ even, an ordering of $V_{\mathrm{seg}}(\Gamma)$, and an ordering of $E(\Gamma)$.
\end{itemize}
We call a link diagram together with an orientation an \emph{oriented link diagram}.
\end{defin}

As with $\D(m)$, we let $\LD(m)$ be the $\R$-vector space of oriented link diagrams modulo these relations:
\begin{itemize}
\item
$\Gamma \sim 0$ if $\Gamma$ has any multiple edges;
\item
for odd $n$, $\Gamma \sim - \Gamma'$ if $\Gamma$ and  $\Gamma'$ differ by a transposition of two vertex labels;
\item
for odd $n$, $\Gamma \sim -\Gamma'$ if $\Gamma$ and $\Gamma'$ differ by an orientation-reversal of an edge;
\item
for even $n$, $\Gamma \sim - \Gamma'$ if $\Gamma$ and  $\Gamma'$ differ by a transposition of two edge labels.
\end{itemize}

It is clear that for a given orientation in either parity, there is a representative labeling in which  
the segment vertices are labeled according to the order of the segments and the order of vertices on each segment.  For odd $n$, it is clear that the labeling can be further chosen so that the segment vertex labels precede all the free vertex labels.

The \emph{order} of a diagram $\Gamma$ is defined as 
\[
r = |E(\Gamma)| - |V_{\mathrm{free}}(\Gamma)|
\]
while the \emph{defect} of $\Gamma$ is defined as 
\[
s = 2 |E(\Gamma)| - 3 |V_{\mathrm{free}}(\Gamma)| - |V_{\mathrm{seg}}(\Gamma)|.
\]
Thus a diagram has defect zero precisely if it is trivalent, if we count edges and arcs alike towards valence.
The \emph{degree} of a diagram $\Gamma$ of order $r$ and defect $s$ is 
\[
k = ((n+1)-3)r + s  = (n-2)r +s
\]
and may be viewed as the total degree in $\LD(m)$.

The differential $d$ on $\LD(m)$ is given by 
\[
d\Gamma  = \sum \varepsilon(e) \Gamma/e
\]
where the sum is taken over arcs and non-chord edges $e$, $\varepsilon(e)$ is a sign depending on $e$, and $\Gamma/e$ is the result of contracting $e$ to a point.  
It preserves order and raises defect by one, thus raising total degree by one.  
This differential $d$ makes $\LD(m)$ into a cochain complex.

There is a product called the \emph{shuffle product} on $\LD(m)$.  
Note that via the orientation of each segment, any link diagram $\Gamma$ induces an ordering of the segment vertices on each segment.
Given diagrams $\Gamma_1, \Gamma_2$, a shuffle $\sigma$ is an ordering of the vertices of $\Gamma_1 \sqcup \Gamma_2$ on the $i$th segment for each $i=1,...,m$, which respects the orders induced by those on $\Gamma_1$ and $\Gamma_2$.  Define $\Gamma_1 \cdot_\sigma \Gamma_2$ as the link diagram given by placing the segment vertices on the segments in the order specified by $\sigma$.  
The orientation is given by leaving all the labels on vertices and edges in $\Gamma_1$ unchanged and appropriately raising the labels in $\Gamma_2$.  The shuffle product is then defined as 
\[
\Gamma_1 \bullet \Gamma_2 := \sum_\sigma \Gamma_1 \cdot_\sigma \Gamma_2.
\]
This makes $\LD(m)$ into a CDGA.  
(Note that the labels on the segment vertices will in general not agree with the order of the vertices on the segment, as incorrectly stated just before Definition 3.21 in the work of the last two authors \cite{KMV:FTHoLinks}.)

There is also a \emph{coproduct}, given by deconcatenation.  Roughly the coproduct of $\Gamma$ is the sum of all the ways to cut each of the $m$ strands in two and produce two diagrams.  One can make this more precise by considering partitions of the segment vertices into two parts $p_1,p_2$ such that the order of the segment vertices on each segment is preserved and such that for each component $c \subset \Gamma$, all the segment vertices of $c$ are in either $p_1$ or $p_2$.

In \refS{MappingLDtoBD}, we will need the fact that the cohomology of the defect zero part of this complex is isomorphic to finite type invariants of long links, as well as a certain combinatorial description of this space.  In particular, classes in $\Ho^*(\LD(m))$ of defect zero are precisely linear combinations of trivalent diagrams whose coefficients satisfy the STU relation.  That is, they yield zero when paired with the sum below, where again we use Kronecker pairing on the basis of link diagrams, which as with braid diagrams is canonical up to signs.\footnote{Technically, the pairing also has a factor of $|\mathrm{Aut}(\Gamma)|$, as in \cite[Section 3.4]{KMV:FTHoLinks}.  However, in our application of this relation in \refS{MappingLDtoBD}, we will only consider forest long link diagrams, which have no nontrivial automorphisms.}  In each picture, the horizontal line is a segment, while the arrows are edges.  To get the correct signs, one needs orientations of the diagrams (i.e.~vertex-labels and edge-orientations or edge-labels), but the signs will not be needed in our application.  (A planar embedding of a diagram determines an orientation of odd type, and the signs are for these orientations.)  For more details, see \cite[Proposition 3.33]{KMV:FTHoLinks}. 
\begin{figure}[ht]
	\centering
\begin{equation}
\label{E:STU}
\raisebox{-1.8pc}{\includegraphics[height=3.6pc]{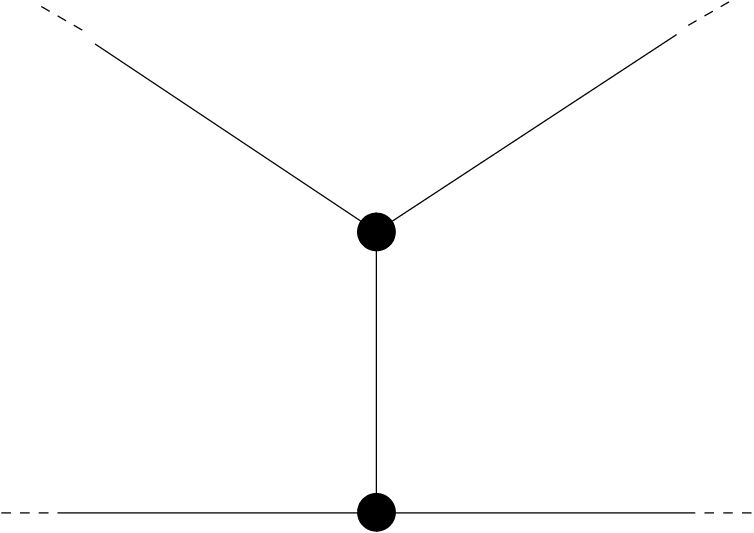}} -
\raisebox{-1.8pc}{\includegraphics[height=3.6pc]{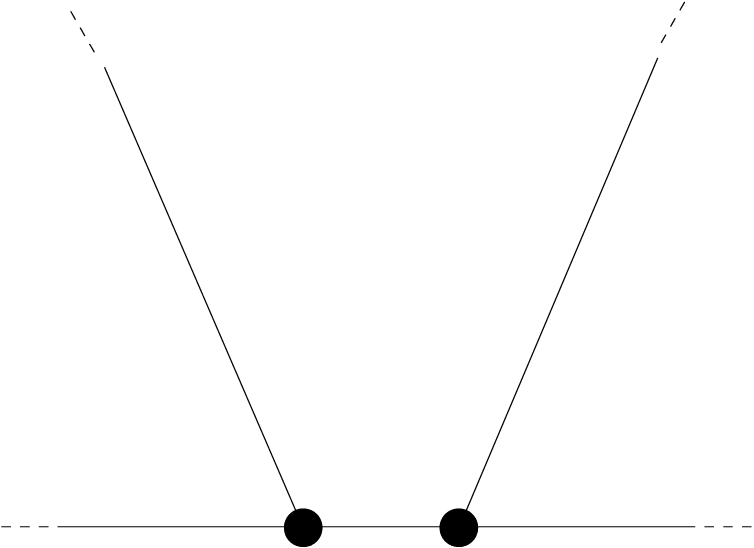}} +
\raisebox{-1.8pc}{\includegraphics[height=3.6pc]{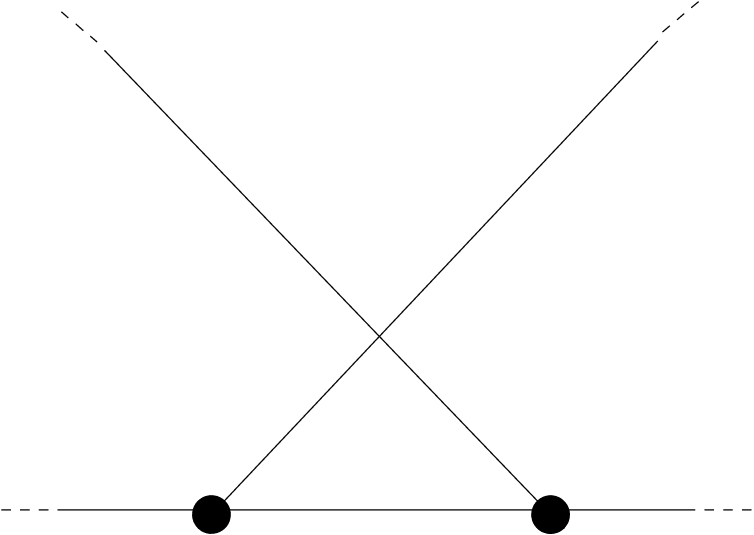}}=0
\end{equation}
\caption{The STU relation.}
\end{figure}

\subsection{Integration on the link diagram complex}
\label{S:ReviewOfLDIntegrals}  Recall from \refS{BraidsAsLinks} the definition of the space $\mathcal{L}_m^{n+1}$ of $m$-component long links in $\R^{n+1}$. 
We now briefly review the integration map from $\LD(m)$ which produces cohomology classes in $\mathcal{L}_m^{n+1}$; \cite[Section 4]{KMV:FTHoLinks} has the details.  

The crucial ingredient is the square \eqref{E:BT-Square} below, where the upper-left space is defined as the pullback.  
In line with Remark \ref{R:Compactification}, all the configuration spaces are now in fact compactified; in particular, the spaces on the right side are Axelrod--Singer compactifications of configuration spaces of points in $\R^{n+1}$, viewed as a subset of $S^{n+1}$.  
A point in the configuration space in the lower-left corner roughly consists of $i_k$ points on the $k$th copy of $\R$, but technically, one must embed these as a long link in $\R^{n+1}$, and then view them as points on a singular submanifold of $S^{n+1}$ by compactifying at $\infty$.  
For this compactification to be well defined, we must also require that the $2m$ fixed affine-linear embeddings have distinct directions in $S^n$; essentially, the $2m$ link strands must intersect transversely at infinity.  We may however take these $2m$ lines to lie in a subspace $\R^2 \subset \R^{n+1}$.
The right column forgets the last $j$ points, while the bottom row is evaluation of the long link at the configuration points in its domain. 
\begin{equation}
\label{E:BT-Square}
\xymatrix{
\Conf_{\mathcal{L}}(i_1,\dots ,i_m;j)
\ar[r]  \ar[d]  &
\Conf (i_1 + \cdots + i_m + j, \R^{n+1})\ar[d]^{pr}  \\
\mathcal{L}_m^{n+1}\times  \Conf \left((i_1, \ldots, i_m), \coprod_{1}^m \R\right)  \ar[r]^-{ev} &
\Conf (i_1+ \cdots +i_m, \R^{n+1})
}
\end{equation}

The pullback space above is the total space of a bundle
\[
\widetilde{\pi}\colon \xymatrix{\Conf_{\mathcal{L}}(i_1,...,i_m;j) \ar[r] & \mathcal{L}_m^{n+1} }
\]
using the left vertical map, followed by projection to $\mathcal{L}_m^{n+1}$.  Denote its fiber $F[i_1,...,i_m; j]$.  With the same notation as in the formality integration, we let $\alpha_{ij}$ denote a similar pullback from $S^n$.  Namely, $\alpha_{ij}$ is the pullback of a unit volume form on $S^n$ via the composition 
\[
\xymatrix{
\Conf_{\mathcal{L}}(i_1,...,i_m;j) \ar[r] & \Conf (i_1 + \cdots + i_m + j, \R^{n+1}) \ar[r]^-{\phi_{ij}} & S^n
}
\]
where $\phi_{ij}(x_1,...,x_n) := (x_j - x_i) / |x_j - x_i|$ for interior points (and is extended continuously to the compactification).  
Let
$$
\alpha_\Gamma = \bigwedge_{\text{edges $(i,j)$ of $\Gamma$}}\alpha_{ij}.
$$
Now let $\Gamma$ be a diagram with $i_1, ..., i_m$ segment vertices on the segments $1,...,m$ and $j$ free vertices.  Define the Bott--Taubes integration map by 
\begin{equation}
\label{E:BT-Integral}
\widetilde{I}(\Gamma) =\widetilde{\pi}_*(\alpha_\Gamma)=
\int_{F[i_1, ..., i_m; j]} \bigwedge_{\text{edges }(i,j)} \alpha_{ij} 
\end{equation}
and extending by linearity to all of $\LD(m)$.  Note the similarity to the formality integral  \eqref{E:FormalityIntegral}.  The main difference is that the base space of the bundle along which integration takes place is different.  Hence integration  $I(\Gamma)=\pi_{*}(\alpha_\Gamma)$ produces forms on the configuration space $\Conf(m, \R^n)$ while integration $\widetilde{I}(\Gamma) =\widetilde{\pi}_*(\alpha_\Gamma)$ produces forms on $\mathcal{L}_m^{n+1}$.  In both cases, the fact that the integration map is a chain map boils down to Stokes' Theorem and checking that all of the integrals along codimension one faces either vanish or are accounted for by the differential in the diagram complex.  

Differential forms on the infinite-dimensional manifold $\mathcal{L}_m^{n+1}$ can be made precise via a structure similar to Chen's differentiable spaces: roughly, one defines what smooth maps from arbitrary finite-dimensional manifolds into $\mathcal{L}_m^{n+1}$ are.  For details, see \cite[Section 2]{KMV:FTHoLinks} or the references mentioned therein.

%%%%%%%%%%%%%%%%%%%%%%%%%%%%%%%%%%%%%%%%%%%%%%%%%%%%%%%%%%%%%%%%%%%%%%%%%%%%%%%%%%%%%%%%%%%%%%%%%%%%%%%%

%%%%%%%%%%%%%%%%%%%%%%%%%%%%%%%%%%%%%%%%%%%%%%%%%%%%%%%%%%%%%%%%%%%%%%%%%%%%%%%%%%%%%%%%%%%%%%%%%%%%%%%%

\subsection{Another bigrading on braid diagrams}
\label{S:OrderDefectForBraidDiagrams}

%%%%%%%%%%%%%%%%%%%%%%%%%%%%%%%%%%%%%%%%%%%%%%%%%%%%%%%%%%%%%%%%%%%%%%%%%%%%%%%%%%%%%%%%%%%%%%%%%%%%%%%%

For the purpose of mapping $\LD(m)$ to $\Ba(\D(m))$, we adapt the notions of defect and order on $\LD(m)$ to $\Ba(\D(m))$, thus defining a new bidegree $(r,s)$ on the latter complex.  
This new bidegree is not absolutely indispensable for our results in the next subsection, but it helps streamline considerations about the map of diagram complexes.
The most important result in this subsection is \refP{BraidCohSpannedByDefectZero}, which says that all of the cohomology of $\Ba(\D(m))$ lies in defect zero.  In the next subsection, after defining the map $\phi\colon \LD(m) \to \Ba(\D(m))$, we will see in \refP{PhiPreservesDegreeEtc} that it preserves both defect and order.

We first need to extend the notions of \emph{components} and \emph{grafts} to diagrams in $\D(m)$ and braid diagrams in $\Ba(\D(m))$.   For the following definition, view generators of $\Ba^{-p,q} (\D(m))$ as $p$ diagrams lying on $m$ parallel strands.
 
\begin{defin} \ 
\begin{itemize}
\item
For a diagram $\Gamma \in \D(m)$, a \emph{component} is a connected component of $\Gamma$ which contains at least one edge.  
\item
For such a diagram, we define the \emph{grafts} of $\Gamma$ similarly as for link diagrams: Replace each segment vertex by the set of its neighbors and create a new diagram with the same number of free vertices and edges as $\Gamma$, but possibly more segment vertices; the grafts are the components of the resulting diagram.  
\item
For a braid diagram $[\Gamma_1 | \cdots | \Gamma_p] \in \D(m)^{\otimes p}$, the components and grafts are the disjoint unions of those of the factors $\Gamma_i$.
\item
Call a diagram $\Gamma$ \emph{connected} if it has only one component.
\item
Call a diagram $\Gamma$ \emph{internally connected} if it has only one graft.
\end{itemize}
\end{defin}

Define the \emph{order} of a diagram $\Gamma \in \D(m)$ as
\[
r(\Gamma) = |E(\Gamma)| - |V_{\mathrm{free}}(\Gamma)|.
\]
Define the \emph{defect} of a diagram $\Gamma \in \D(m)$ as 
\[
s(\Gamma) = |E(\Gamma)| - 2 |V_{\mathrm{free}}(\Gamma)|  - 1.
\]
%For any diagram, both quantities are nonnegative integers.
%OOPS not if diagram has loops
%The defect $s(\Gamma)$ of a diagram $\Gamma$ with grafts $G_1, \dots, G_\ell$ satisfies 
%\begin{equation}
%\label{DefectAndGrafts}
%s(\Gamma) =(\ell -1) + \sum_{i=1}^\ell s(G_i)
%\end{equation}
%where the graft $G_i$ is viewed as a diagram in an obvious way and $s(G_i)$ is its defect.  The order of $\Gamma$ is simply the sum of the orders of its grafts $G_i$.

We extend these definitions to braid diagrams $[\Gamma_1 | \cdots | \Gamma_p] \in \D(m)^{\otimes p}$ by defining the order to be
\[ 
r([\Gamma_1 | \cdots | \Gamma_p]) := \sum_{i=1}^p r(\Gamma_i) = \sum_{i=1}^p (|E(\Gamma_i)| - |V_{\mathrm{free}}(\Gamma_i)|)
\]
and the defect to be 
\[
s([\Gamma_1 | \cdots | \Gamma_p]) :=
\sum_{i=1}^p s(\Gamma_i) =
\sum_{i=1}^p (|E(\Gamma_i)| - 2 |V_{\mathrm{free}}(\Gamma_i)| - 1).
\]

Each of the two differentials in $\Ba(\D(m))$ preserves the order and raises the defect by one, just as in the case of the differential on $\LD$.  So the same is true of the total differential in the bar complex.  The reader may verify the next proposition directly from the definitions.  It will later guarantee that, via the map of diagram complexes, the two integration maps produce forms of the same degree.
\begin{prop}
Let $\gamma = [\Gamma_1 | \cdots | \Gamma_p] \in \Ba^{-p,q}(\D(m))$ be a braid diagram (so that $q=|\Gamma_1|+...+|\Gamma_p|$).
Let $r$ be the order of $\gamma$ and $s$ be the defect of $\gamma$. 
Then the total degree $k:=q-p$ of $\gamma$ satisfies $k=(n-2)r+s$.
\qed
\end{prop}

The following Proposition will merely streamline the argument for surjectivity in cohomology of the map $\LD(m) \to \Ba(\D(m))$ in the next subsection.

\begin{prop}
\label{P:BraidCohSpannedByDefectZero}
The cohomology of the complex $\Ba(\D(m))$ is spanned by braid diagrams $[\Gamma_1 | \cdots  | \Gamma_p]$ %where each $\Gamma_i$ is 
of defect zero.  
%Moreover, each $\Gamma_i$ is internally connected.
%OOPS could have multiple grafts, if the diagrams have loops, since such a diagram would have negative defect...
\end{prop}
\begin{proof}
We use the zig-zag \eqref{E:BarZigzag} of maps $I$ and $\overline{I}$ from $\Ba(\D(m))$ to cochains and cohomology respectively.
The map $\overline{I}$ is nonzero only on diagrams in which all grafts are chords.  
%Furthermore, \refP{NoCohomAboveVanishLine} says that such a diagram with multiple grafts (i.e.~multiple chords in one bar factor) vanishes when we take the cohomology of $\Ba(\Ho^*(\Conf(m, \R^n)))$.  
%%elements in the bar complex having diagrams with multiple grafts vanish in cohomology.
Thus every $\gamma = \sum_i c_i \gamma_i \in \Ba(\D(m))$ which is nontrivial in cohomology has at least one term (with nonzero coefficient) $\gamma_0 = [\Gamma_1 | \cdots | \Gamma_p]$ where each $\Gamma_i$ is a chord.  Such a diagram $\gamma_0$ has defect zero.  Now view $\Ba(\D(m))$ as the direct sum over $r$ of its defect $r$ subspaces, and recall that the differential raises defect by one.  
This implies that a representative for $\gamma$ can be taken to be a linear combination of terms of defect zero.
%have only terms of the form $c[\Gamma_1 | \cdots  | \Gamma_p]$ where all factors $\Gamma_i$ have defect zero.  
%%They are internally connected by the nonnegativity of defect and equation \eqref{DefectAndGrafts}. (UNIQUE expression because no negative defect -- correction there is negative defect for B(D(m))...).
\end{proof}

\begin{rems}  
\label{R:DefectZeroRems}
The following are facts about braid diagrams in defect zero, with proofs sketched or left as exercises for the reader.
\begin{enumerate}
\item
For a defect zero braid diagram $\gamma$ in which every factor has one graft, the order of $\gamma$ is
\[
r=p + \sum |V_{\mathrm{free}}(\Gamma_i)|.
\]
\item
In the case of classical braids $(n=2)$, defect zero corresponds to total degree zero, or equivalently the main diagonal $\Ba^{-p,p}(\D(m))$. 
\item
As noted in the proof of \refP{BraidCohSpannedByDefectZero}, a braid diagram in which every factor is a single chord has defect zero.  Such diagrams span the vanishing line $q=(n-1)p$ in $\Ba^{-p,q}(\D(m))$.  However, not every defect zero diagram lies on this line.  For example, the tripod (one free vertex with three edges to three segment vertices on different segments) is of defect zero, and while it lives in $\Ba^{-1,1}(\D(m))$ when $n=2$, it lives in $\Ba^{-1, \, 2n-3}(\D(m))$ in general.  
\item
If $\Gamma$ is an internally connected diagram such that its graft is a trivalent tree, then $\Gamma$ has defect zero.  This can be checked using trivalence and the fact that a tree has Euler characteristic one.
\item
The converse of the previous statement is not true, as can be seen from diagrams with closed loops of edges and vertices of valence $>3$.   However, the surjectivity in cohomology of $\LD(m) \to \Ba(\D(m))$ will ensure that  cocycles in $\Ba(\D(m))$ can be written as combinations of $[\Gamma_1 | ... | \Gamma_p]$ where each $\Gamma_i$ is internally connected with a trivalent tree as its graft.
\end{enumerate}
\end{rems}

%%%%%%%%%%%%%%%%%%%%%%%%%%%%%%%%%%%%%%%%%%%%%%%%%%%%%%%%%%%%%%%%%%%%%%%%%%%%%%%%%%%%%%%%%%%%%%%%%%%%%%%%

\subsection{Mapping the complex $\LD(m)$ to the bar complex of $\D(m)$}
\label{S:MappingLDtoBD}
We will now define a map $\phi$ from the link diagram complex $\LD(m)$ to the bar complex of (braid) diagrams $\Ba(\D(m))$.  
This map will not quite be a chain map, but we will show in \refP{PhiChainMapOnZ0} that its restriction to a certain subspace of cocycles of degree zero is a chain map.  
We will show in \refT{DiagramsSurjOnCohom} that $\phi$ induces a surjection onto the cohomology of $\Ba(\D(m))$ and hence the cohomology of the space of braids.  
Roughly, $\phi$ will send a combination of link diagrams to one where every component lies in a ``time slice'' $\R^{n}$, by killing some diagrams and modifying others.  
Though the precise definition of $\phi$ below is lengthy, the idea is simple:
\begin{itemize}
\item 
Kill non-forests, i.e.~diagrams with loops of edges.
\item
Collapse each arc between a pair of legs (i.e.~leaves) of a tree.  Technically, this operation also requires multiplying by a certain fractional coefficient.
\item
View the result as the (sum of all the) braid diagram(s) represented by the resulting picture, where if the picture does not look like a braid diagram at all, we send it to 0.
\end{itemize}
The image of a cocycle of link diagrams will have at most as many terms as the original cocycles.  In \refP{HtpyDiagramsAreGood}, we will see that equality holds for those cocycles corresponding to Milnor homotopy invariants.

\begin{defin}
\label{D:LinksToBraidsMap}
Define a linear map $\phi: \LD(m) \to B(\D(m))$ first on connected diagrams $\Gamma$, i.e.~those with just one component:
\begin{itemize}[leftmargin=1.0in]
\item[(Quotient 1)]
Send to zero each $\Gamma$ which has a closed path of edges (none of which are arcs), including possibly a self-loop.
\item[(Collapse)]
For any other connected diagram $\Gamma$, first let $\overline{\Gamma}$ be the result of collapsing each arc between segment vertices $s_1,s_2$ on the same component.  
The result can clearly be viewed as an unoriented diagram in $\D(m)$.
Then define 
$$\phi(\Gamma):=\frac{1}{k_1 ! \cdots k_m !} \overline{\Gamma},$$
where $k_i$ is the number of segment vertices on the $i$th strand.
\item[(Orientations)]
For a connected $\Gamma$ with $\phi(\Gamma)\neq 0$, define the labeling (i.e.~orientation) on $\phi(\Gamma)$ as follows. 
As mentioned just after \refD{LDOrientation}, we may suppose that $\Gamma$ has a representative labeling where the segment vertices are labeled first, in order, by their segment and the segment orientation.
In either parity, this canonically determines a labeling of the result of the collapses of arcs above.  Since $\phi(\Gamma) \neq 0$, $\Gamma$ is a tree, and this guarantees that an orientation of odd type (free vertex ordering and edge orientations) canonically determines an orientation of even type (edge ordering).  For details, see the Appendix.
\end{itemize}
This concludes the definition of $\phi$ for connected $\Gamma$.  Now we define $\phi$ on all of $\LD$:
\begin{itemize}[leftmargin=1.0in]
\item[(Quotient 2)]
Send to zero each $\Gamma$ in which the order of the segment vertices on each segment does \emph{not} induce a partial order on the components, in the following sense.  If $c \neq c'$ are components of $\Gamma$ and $c \to c'$ means that on some segment $i$, there are vertices $s<s'$ in $c,c'$ respectively, then the existence of a directed cycle of components $c_1 \to \cdots \to c_k \to c_1$ for some $k \geq 2$ is precisely the condition for $\Gamma$ to be sent to zero.  We will call such a diagram $\Gamma$ \emph{order-violating}.  For instance, the diagrams in Examples \ref{PhiExamples} (e) and (f) are order-violating.
\item[(Orientations)]
For $\Gamma$ with $\phi(\Gamma) \neq 0$, choose a total order on the components of $\Gamma$ which respects the partial order on components induced by the segment vertices.  
Then relabel $\Gamma$ so as to respect this total order.
If the new labeling gives the opposite orientation from the original one on $\Gamma$, then multiply by $-1$.  It is clear that this labeling induces a labeling of each component as an element of $\D(m)$.
\item[(Shuffle)]
Finally, consider all possible total orders of the components of $\Gamma$ which respect the partial order induced by the segment vertices.  
%(For example, if the components are labeled $a,b,c,d$ and the partial order is given by $a<b$ and $c<d$, the possible total orders are $a<b<c<d$, $a<c<b<d$, $a<c<d<b$, $c<a<b<d$, $c<a<d<b$, and $c<d<a<b$.)  
Define $\phi$ by sending $\Gamma$ to all the shuffles in $\Ba(\D(m))$ of its components corresponding to these total orders, where each shuffle appears with the labeling as induced above and with the sign given in \eqref{E:BarShuffle}.  See Example \ref{PhiExamples} (a) below for an illustration (without signs).
\end{itemize}
For brevity, we will refer to the relations (Quotient 1) and (Quotient 2) as (Q1) and (Q2) respectively.  We will call diagrams that are sent to zero by these quotients as \emph{bad diagrams} and all other diagrams \emph{good diagrams}.
\end{defin}

\begin{example}
\label{PhiExamples}
Below are some examples of the map $\phi$, where labelings (and thus signs) are unspecified.  Almost all of these diagrams are in the defect zero subspace, which will be of particular interest. 

(a) \begin{align*}
\qquad \qquad \qquad \raisebox{-2.3pc}{\includegraphics[scale=0.2]{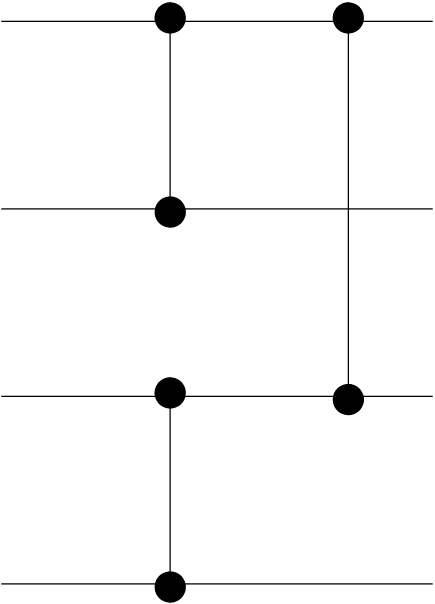}}
\qquad &\mapsto \qquad
\pm \quad \raisebox{-2.3pc}{\includegraphics[scale=0.2]{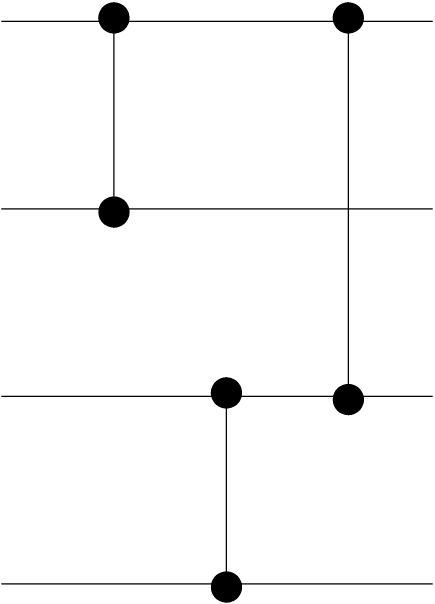}}
\quad \pm \quad \raisebox{-2.3pc}{\includegraphics[scale=0.2]{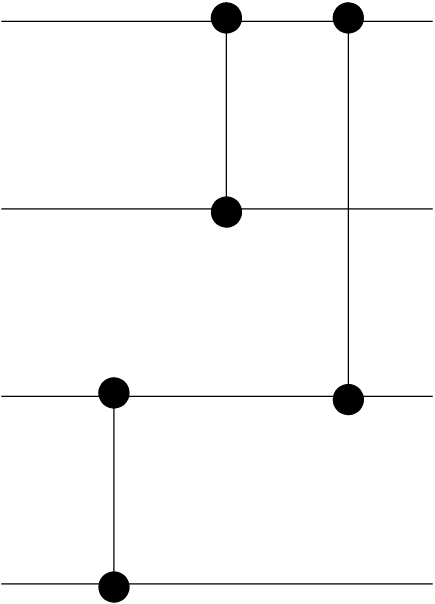}}
\end{align*}
(b) \begin{align*}
\raisebox{-1.4pc}{\includegraphics[scale=0.2]{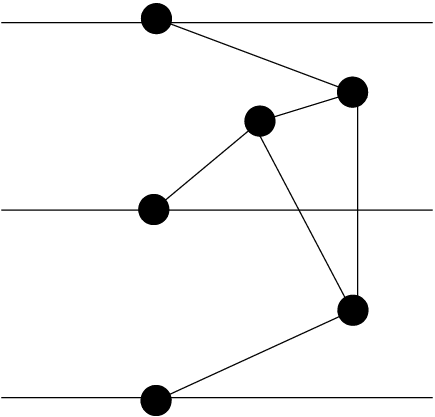}} \qquad &\mapsto \qquad 0 \qquad \qquad \qquad 
\end{align*}
(c)
\begin{align*}
\raisebox{-1.4pc}{\includegraphics[scale=0.2]{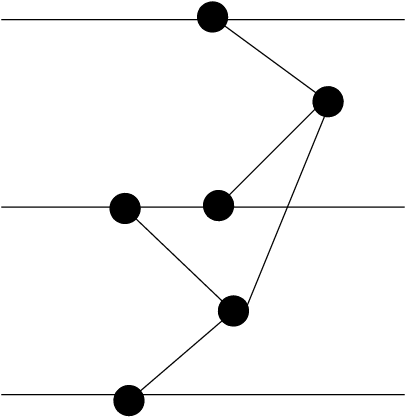}} 
\qquad &\mapsto \qquad
\pm \, \frac{1}{2} \, \raisebox{-1.4pc}{\includegraphics[scale=0.2]{triangle.eps}}  \,\,
\end{align*}
(d)
\begin{align*}
\, \raisebox{-1.4pc}{\includegraphics[scale=0.2]{triangle.eps}} \qquad &\mapsto \qquad 0  \qquad \qquad \qquad 
\end{align*}
(e)
\begin{align*}
\raisebox{-1.4pc}{\includegraphics[scale=0.2]{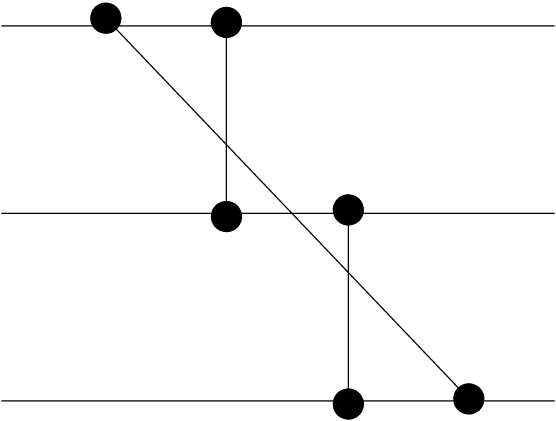}} \qquad  &\mapsto \qquad 0 \qquad \qquad \qquad \quad
\end{align*}
(f)
\begin{align*}
\raisebox{-2.9pc}{\includegraphics[scale=0.4]{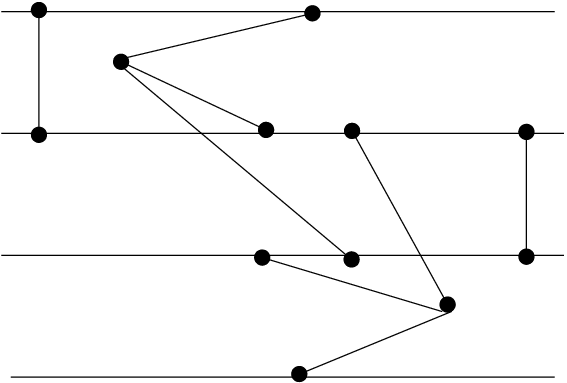}} \qquad  &\mapsto \qquad 0 \qquad \qquad \qquad \qquad \qquad \qquad 
\end{align*}
(g)
\begin{align*}
\raisebox{-2.9pc}{\includegraphics[scale=0.4]{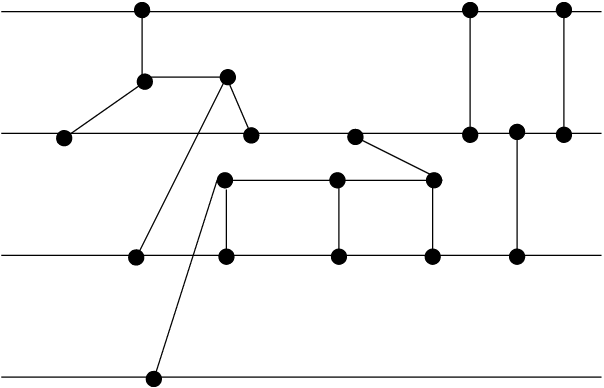}} 
\qquad &\mapsto \qquad
\pm \, \frac{1}{2! \, 3!} \, \raisebox{-2.9pc}{\includegraphics[scale=0.4]{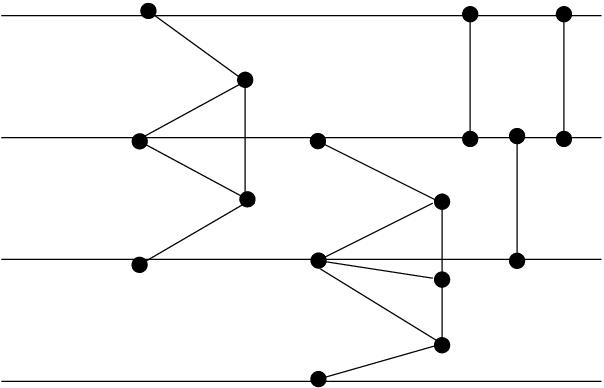}}
\end{align*}
\end{example}

It is not difficult to verify the following statement.

\begin{prop}
The map $\phi$ is a map of Hopf algebras.
\qed
\end{prop}

The next proposition is convenient for organizational purposes.

\begin{prop}
\label{P:PhiPreservesDegreeEtc}
The map $\phi$ preserves any subspace given by a fixing a defect, order, or (total) degree of the diagrams.
%defect, order, and (total) degree of a diagram are all preserved by $\phi$.
\end{prop}

In other words, $\phi$ preserves the defect, order, and total degree of any diagram except for those in its kernel (since we don't define  those three quantities for the element $0$).

\begin{proof}
For each of these three quantities, the value on $\Gamma \in \LD(m)$ is the sum of the values on the components of $\Gamma$.
Similarly, the value on a braid diagram $[\Gamma_1 | \cdots | \Gamma_p]$ is the sum of the values on each $[\Gamma_i]$, where we interpret the total degree of such a term to be $|\Gamma_i| - 1$.  Since every link diagram is sent to shuffles of its connected components (or $0$), it suffices to check the statement in the case that $\Gamma \in \LD(m)$ is connected.

The fact that $\phi$ preserves order follows from the agreement of the definitions of order in the two settings and the fact that $\phi$ preserves the numbers of edges and free vertices.  
The defect of $\Gamma  \in \LD(m)$ is
\begin{equation}
\label{E:LinkDefect}
2|E(\Gamma)| - 3|V_{\mathrm{free}}(\Gamma)| - |V_{\mathrm{seg}}(\Gamma)|
\end{equation}
while the defect of an element $[\Gamma] \in \Ba(\D(m))$ is
\begin{equation}
\label{E:BraidDefect}
|E(\Gamma)| - 2|V_{\mathrm{free}}(\Gamma)| - 1.
\end{equation}
The only connected link diagrams $\Gamma$ which map to nonzero elements under $\phi$ are trees, and we use again the fact that $\phi$ preserves the numbers of edges and free vertices.  For a tree $\Gamma$, one can then check that indeed \eqref{E:LinkDefect} and \eqref{E:BraidDefect} agree, since a tree has Euler characteristic 1 and since $|V(\Gamma)|=|V_{\mathrm{free}}(\Gamma)| + |V_{\mathrm{seg}}(\Gamma)|$.
\end{proof}

Now $\phi$ is not a chain map in general, since contracting arcs in a (bad) diagram sent to zero by (Q2) may produce a (good) diagram whose image under $\phi$ is nonzero.  
This can be seen using any diagram where two components cross each other as in Example \ref{PhiExamples}(f).  
Nonetheless, in \refP{PhiChainMapOnZ0}, we will see that its restriction to cocycles in the following subspace will be a chain map.

\begin{defin}
\label{D:LDfDef}
Let $\LD_{\mathrm{f}}(m)$ denote the subspace of $\LD(m)$ spanned by diagrams $\Gamma$ satisfying the following conditions:
\begin{enumerate}
\item $\Gamma$ is a forest, in the sense that all its connected components are trees; and
\item if a component $c$ of $\Gamma$ touches only one strand, then some segment vertex of another component $c'$ must lie between the first and last segment vertices of $c$ (using the order on the segment).
\end{enumerate}
\end{defin}

See Figure \ref{F:LDfFig} for some examples illustrating these conditions.
Condition (2) above will be needed only later, in \refS{Integrals}, for a technicality related to Bott--Taubes configuration space integrals.

\begin{figure}[h!]
(a) \includegraphics[scale=0.45]{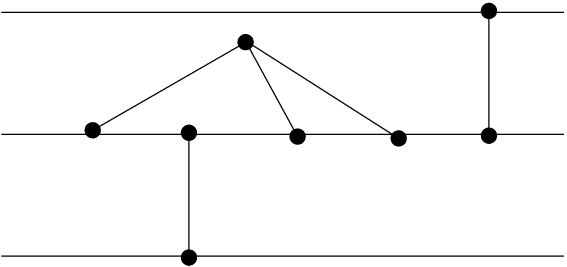}
\qquad \qquad
(b) \includegraphics[scale=0.45]{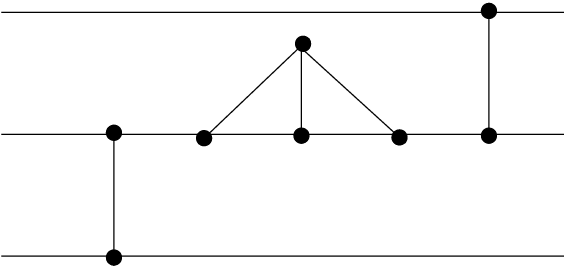}
\caption{Examples of (a) a diagram which satisfies both conditions in Definition \ref{D:LDfDef} and (b) a diagram which satisfies condition (1) but not condition (2).}
\label{F:LDfFig}
\end{figure}

We are particularly interested in $\ZZ^{0,*}(\LD_{\mathrm{f}}(m))$, the cocycles of defect zero (and any order) in $\LDf(m)$.  
This space $\ZZ^{0,*}(\LD_{\mathrm{f}}(m))$ is spanned by trivalent forests.  
Furthermore, its image in cohomology $\Ho^{0,*}(\LD_{\mathrm{f}}(m))$ is isomorphic to $\ZZ^{0,*}(\LD_{\mathrm{f}}(m))$ itself, since there are no cocycles of negative defect in the link diagram complex $\LD$.

Before proceeding with generalities, we mention two examples of cocycles of defect zero.  
One example is $\mu_{123}$ in \eqref{E:mu_123}, which represents the Milnor triple linking number.  From a graphical perspective, the map $\phi$ is essentially the identity on $\mu_{123}$, a fact generalized to all Milnor homotopy invariants in \refP{HtpyDiagramsAreGood} below.  The next example is however more nuanced.

\begin{example}
\label{3StrandsOrder3}
 As one can compute using the 4T relation, the space of finite type 3 invariants for 3-strand braids is two-dimensional.  
For one basis element, one can take the sum of braid chord diagrams 
\begin{equation}
\label{E:ChordsForType3Invt}
\raisebox{-1.4pc}{\includegraphics[scale=0.2]{LLM.eps}} - 
\raisebox{-1.4pc}{\includegraphics[scale=0.2]{LLR.eps}} + 
\raisebox{-1.4pc}{\includegraphics[scale=0.2]{LMR.eps}} - 
\raisebox{-1.4pc}{\includegraphics[scale=0.2]{LRR.eps}} + 
\raisebox{-1.4pc}{\includegraphics[scale=0.2]{MRR.eps}}  
\end{equation}
where the signs shown are for even $n$.
This sum corresponds to only the terms in $\Ba(\D(3))$ which survive under the map $\overline{I}$ to cohomology.  One can deduce that the full cocycle $\beta$ in $\Ba(\D(3))$ also includes the terms (up to sign)
%\robfn{I should get these signs right, just for the classical case (even $n$).}
\[
\pm 
\raisebox{-1.4pc}{\includegraphics[scale=0.2]{LT.eps}} \pm 
\raisebox{-1.4pc}{\includegraphics[scale=0.2]{TR.eps}} \pm 
\raisebox{-1.4pc}{\includegraphics[scale=0.2]{triangle.eps}}.
\]
There is a corresponding cocycle $\alpha$ in $\LD(3)$ (as will be guaranteed by \refP{PhiChainMapOnZ0}).  
The terms in $\alpha$ mapping to the third diagram above are the two trees with shuffles of leaves on strand 2, each with coefficient $\pm1$.  In addition to the obvious lifts to $\LD$ of the two diagrams with a tripod and a chord, there are two further diagrams in $\alpha$ involving the tripod and a chord (respectively, between strands 1 and 2 and strands 2 and 3); each of these diagrams has coefficient $\pm 1$.  The remainder of the terms are chord diagrams, including those in \eqref{E:ChordsForType3Invt} as well as others, determined by the STU relations, the coefficients of the diagrams with free vertices, and the coefficients of the  chord braid diagrams.  
This example thus illustrates how the map $\phi$ simplifies configuration space integral formulas for braid invariants.
\qed
\end{example}

We now show that the restriction of $\phi$ to $\ZZ^{0,*}(\LD_{\mathrm{f}}(m))$ is a chain map.  (For this restriction, the relation (Q1) becomes vacuous.)  By abuse of notation, we also let $\phi$ denote this restriction.

\begin{prop}
\label{P:PhiChainMapOnZ0}
The restriction of $\phi$ to $\ZZ^{0,*}(\LD_{\mathrm{f}}(m))$ satisfies $d \phi = \phi d$.
Thus its image is contained in the space $\ZZ^{0,*}(\Ba(\D(m)))$ of cocycles of defect zero.
\end{prop}

\begin{proof}
Let $\gamma \in \ZZ^{0,*}(\LD_{\mathrm{f}}(m))$.  
We have $d\gamma=0$, and we want to show that $d(\phi(\gamma))=0$.  
The terms appearing in $d\gamma$ come from edge and arc contractions.  The terms appearing in $d(\phi(\gamma))$ come from edge contractions and from identifying two time-slices (the bar differential).  

We will use the fact that the STU relation \eqref{E:STU} is a condition satisfied by the coefficients of any cocycle $\gamma$ of defect zero.   This is \cite[Proposition 3.29]{KMV:FTHoLinks}; see also \cite[Proposition 1]{Conant-Vogtmann:Morita-Vanish} for an alternative proof of essentially the same fact.

Given a diagram $\Gamma$, call a diagram $\Gamma'$ a \emph{leaf shuffle} of $\Gamma$ if $\Gamma'$ is obtained from $\Gamma$ by a rearrangement of the leaves on each segment (and a possible relabeling, i.e.~up to orientation/sign).
The STU condition and the fact that all the diagrams in $\gamma$ are forests establishes the following claim.
\begin{claim}
\label{C:AllSegmentVShuffles}
If $\gamma \in \ZZ^{0,*}(\LD_{\mathrm{f}}(m))$ and the diagram $\Gamma'$ is a leaf shuffle of $\Gamma$, then the coefficients of $\Gamma$ and $\Gamma'$ in $\gamma$ agree up to sign.  Thus if $\Gamma$ appears in $\gamma$, then so do all its leaf shuffles, and with the same coefficient (up to sign).
\end{claim}

For example, if $\gamma$ has one of the two diagrams below with coefficient $c$, then the other appears with a coefficient $\pm c$:
\[
\includegraphics[scale=0.3]{tree-4-leaves.eps}
\qquad \qquad \qquad
\includegraphics[scale=0.3]{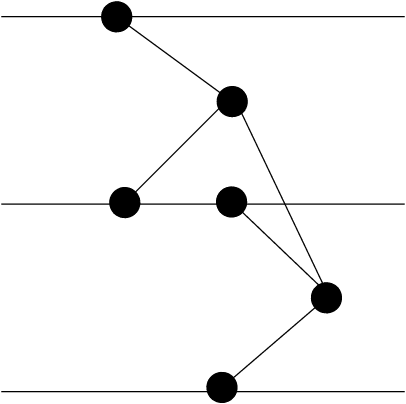}
\]

To show that $d(\phi(\gamma))=0$, we will show that
\begin{itemize}
\item[(1)] the contributions from contracting edges between free vertices sum to zero; 
\item[(2)] the remaining contributions sum to zero.
\end{itemize}

To use the fact that $d \gamma =0$, we set some notation.  First, suppose $\overline{\Gamma} \in \Ba^{-p,*}(\D(m))$ is a braid diagram appearing in $\phi(\gamma)$.  For a permutation $\sigma \in \Sigma_p$, let $\sigma \overline{\Gamma}$ be the result of permuting the factors of $\overline{\Gamma}$ by $\sigma$. 
Then all the shuffles $\sigma\overline{\Gamma}$ which give rise to the same partial order of the components of $\overline{\Gamma}$ (using the orientation of the segments) must appear in $\phi(\gamma)$, as in Example \ref{PhiExamples} (a) above.  
Thus $\sum_\sigma \sigma\overline{\Gamma}$ appears in $\phi(\gamma)$ (with some coefficient), where the sum is taken over all such partial-order-preserving shuffles.

For step (1), consider contractions of edges between free vertices in $\phi(\gamma)$.
We claim that such terms in $\phi(\gamma)$ are in correspondence with edge contraction terms from good diagrams in $\gamma$.  Indeed, given an edge $e$ between free vertices, every term $\overline{\Gamma}/e$ in $d \phi(\gamma)$ is part of a sum $\sum_\sigma \sigma(\overline{\Gamma}/e)$.
Then Claim \ref{C:AllSegmentVShuffles} implies that each term $\sum_\sigma \sigma(\overline{\Gamma}/e)$ in $d \phi(\gamma)$ is matched by an edge contraction in each of $k_1! \cdots k_m!$ diagrams in $\gamma$ that map to $\sum_\sigma \sigma\overline{\Gamma}$ under $\phi$.  The division by $k_1! \cdots k_m!$ in the definition of $\phi$ establishes the correspondence.  (This is the only part where we use the forest condition to show that this restriction of $\phi$ is a chain map.)  
Now note that an edge contraction in a bad diagram yields another bad diagram.
This implies that all the terms in $d \phi(\gamma)$ involving the contraction of an edge between free vertices sum to zero.

For step (2), consider first a term in $d \phi(\gamma)$ from one of the summands in the bar differential.  
(We will ultimately treat the contraction of an edge between a free vertex and segment vertex in case (b) below.)
By definition, a diagram in $\phi(\gamma)$ has only one component in each bar factor, so there are two possibilities for a term in $d \phi(\gamma)$: either 
\begin{itemize}
\item[(a)] it has two components on disjoint sets of strands in some time-slice or 
\item[(b)] it has the contraction to one component of two components $\overline{\Gamma}_i$ and $\overline{\Gamma}_{i+1}$ in $\phi(\gamma)$, where there at least one arc between a vertex on 
$\overline{\Gamma}_i$ and a vertex on $\overline{\Gamma}_{i+1}$
\end{itemize}

In case (a), the shuffle in the definition of $\phi$ ensures that such contributions cancel in pairs.  

For case (b), let $\overline{\Gamma} = [\overline{\Gamma}_1 | \dots | \overline{\Gamma}_p]$ be a braid diagram in $\phi(\gamma)$ with components $\overline{\Gamma}_t$ and $\overline{\Gamma}_{u}$ that are contracted together in some term of $d \overline{\Gamma}$.  Thus $|t-u|=1$, and without loss of generality, $u=t+1$.  
Find $\Gamma \in \LD(m)$ with $\phi(\Gamma)=\sum_\sigma \sigma\overline{\Gamma}$ as above.  Then $\Gamma$ is necessarily a good diagram.  
Let $\Gamma_t$ and $\Gamma_{u}$ be the components in $\Gamma$ corresponding to $\overline{\Gamma}_t$ and $\overline{\Gamma}_{u}$ respectively, and let $\Gamma'$ be the (good) link diagram with the order of $\Gamma_t$ and $\Gamma_{u}$ swapped.  (That is, to obtain $\Gamma'$ from $\Gamma$, one swaps segment vertices of $\Gamma_t$ and $\Gamma_u$ until on each segment, all the segment vertices of $\Gamma_u$ precede those of $\Gamma_t$.)  
Suppose $\Gamma_t$ and $\Gamma_{u}$ have respectively $k_i$ and $k'_i$ leaves on the $i$th strand.  
We will consider all the diagrams from iterating STU relations which involve swaps of leaves of $\Gamma_t$ and $\Gamma_{u}$ (the diagrams T and U, or possibly U and T) and a third diagram where $\Gamma_t$ and $\Gamma_{u}$ are joined into one component by adding a free vertex (the S diagram).  
A priori, these STU relations involve some bad diagrams.  Our goal is to combine these relations to get relations involving only good diagrams.  
By Claim \ref{C:AllSegmentVShuffles}, the coefficient of any link diagram 
%$\Gamma_1$ 
agrees with that of any other link diagram 
%$\Gamma_2$ 
which differs by only a shuffle of leaves on the same component along a given strand (and has the appropriate orientation).  Thus we may consider diagrams up to such leaf shuffles on the $i$-th strand.  
(For example, up to leaf shuffles, there are $\sum_{i=1}^m k_i k'_i$ diagrams of type S, by keeping track of only the order of the segment vertices of $\Gamma_t$ and $\Gamma_u$.)  

Without loss of generality (by relabeling), we may assume that the strands with vertices from both components are labeled $1,2,3,...,\ell \leq m$.
We start on strand $1$.  By repeatedly using the STU relation on strand $1$ and combining these relations into a telescoping sum, one can ``move each of the $k_{1}$ leaves of $\Gamma_t$ on strand 1 past each of the $k'_{1}$ leaves of $\Gamma_u$ on strand 1.''  More precisely, we obtain the vanishing of the sum (with the appropriate signs) of the coefficient of $\Gamma$, the coefficients of $k_{1} k'_{1}$ diagrams of type S (call them $S(1,1,1), ..., S(1,k_1, k'_1)$), and a diagram $\Gamma''$ with all the $k_{1}$ leaves moved past all the $k'_{1}$ leaves.  Abusing notation and representing the coefficients by the same symbols as the diagrams themselves, we may write
\[
\Gamma + \sum_{(j,j')=(1,1)}^{k_1,k'_1} S(1, j, j') + \Gamma'' =0
\]
where signs have also been omitted.
The last term $\Gamma''$ is a bad diagram, but by applying the same procedure, we may replace it by a sum of $k_{2} k'_{2} \ $ diagrams of type S and a bad diagram with $k_1 + k_2$ leaves moved past $k'_1 + k'_2$ leaves.  Continuing, we obtain the vanishing of the sum (with appropriate signs) of the coefficient of $\Gamma$, the coefficients of $\sum_i k_i k'_i$ diagrams of type S, and the coefficient of $\Gamma'$:
\begin{equation}
\label{E:GoodDiagramsSumToZero}
\Gamma + \sum_{i=1}^\ell \sum_{(j,j')=(1,1)}^{k_i,k'_i} S(i, j, j') + \Gamma' =0
\end{equation}
Moreover, all of these diagrams are good.
Up to the equivalence relation of leaf shuffles, these are precisely the diagrams whose images under $\phi$ contribute a term 
%$\overline{\Gamma} / (\overline{\Gamma}_t \sim \overline{\Gamma}_u)$.  
$c [\Gamma_1| \dots | \Gamma_t \Gamma_u | \dots | \Gamma_p ]$ (where $c$ is some coefficient).  
Again, the division by factorials in the definition of $\phi$ guarantees that the coefficient of each equivalence class in $\phi(\gamma) \in \Ba(\D)$  coincides with the coefficient of any representative in $\gamma \in \LD(m)$.  
Thus all terms of the form 
%$\overline{\Gamma} / (\overline{\Gamma}_t \sim \overline{\Gamma}_u)$ 
$c [\Gamma_1| \dots | \Gamma_t \Gamma_u | \dots | \Gamma_p ]$ combine to zero in $d\phi(\gamma)$.  
This accounts for all terms in $d\phi(\gamma)$ where two time slices are identified.  It also accounts for all terms in $d\phi(\gamma)$ where an edge from a free vertex to a segment vertex is contracted in some diagram $\Gamma_S$: any such contribution must be cancelled by some 
%$\overline{\Gamma} / (\overline{\Gamma}_t \sim \overline{\Gamma}_u)$ 
$c [\Gamma_1| \dots | \Gamma_t \Gamma_u | \dots | \Gamma_p ]$, in which case $\Gamma_S = \phi(S(i,j,j'))$ for some $S(i,j,j')$ as in \eqref{E:GoodDiagramsSumToZero}.
\end{proof}

We can now finally address the left vertical map in our main square \eqref{E:MainCommutativeSquare}.  
Write $\Ho^{*,*}(\Ba(\D(m)))$ for the cohomology of $\Ba(\D(m))$, bigraded by defect and order, just like $\Ho^{*,*}(\LD_{\mathrm{f}}(m))$.  (Recall that defect and order differ from the original bigrading on $\Ba(\D(m))$.)

\begin{thm}
\label{T:DiagramsSurjOnCohom}
The map $\phi$ induces a surjection $\xymatrix{\Ho^{0,*}(\LD_{\mathrm{f}}(m)) \ar@{->>}[r] & \Ho^{*,*}(\Ba(\D(m))) }$ in cohomology.
\end{thm}

\begin{proof}
\refP{PhiChainMapOnZ0} guarantees that the restriction of $\phi$ to defect zero cocycles induces a map on cohomology.  
By \refP{PhiPreservesDegreeEtc}, there is a further restriction map 
$$\Ho^{0,n}(\LD_{\mathrm{f}}(m)) \longrightarrow \Ho^{0,n}(\Ba(\D(m)))$$ between subspaces of diagrams of defect zero and order $n$.  
We know from \refP{BraidCohSpannedByDefectZero} that all the cohomology of the bar complex is concentrated in defect zero, so it will suffice to establish surjectivity for each $n$.
By another mild abuse of notation, we will simply refer to this restriction to bidegree $(0,n)$ of the induced map in cohomology as $\phi$.

We will use facts about finite type invariants for links and braids in $\R^3$.  
In particular, the surjectivity will be a consequence of diagram \eqref{E:ConcordanceDiagram} below, namely its commutativity, the isomorphisms on the right-hand vertical edge, and the surjectivity of the left-hand edge followed by the bottom row.
\begin{equation}
\label{E:ConcordanceDiagram}
\xymatrix{
\Ho^{0,n}(\LDf(m)) \ar^-\phi[rr] \ar@{^(->}[d]  & & \Ho^{0,n}(\Ba(\D(m))) \ar^-\cong[d] \\
%\Ho^{0,n}(\LDf(m)) \ar@{^(->}[r] \ar_-\int[d] & 
\Ho^{0,n}(\LD(m)) \ar@{->>}^-\cong[r] \ar_-\int^\cong[d] & (\mathcal{CD}_n(m) /4T)^* \ar@{->>}[r] & (\mathcal{BCD}_n(m) /YB)^* \\
%\mathcal{CV}_n(m) / \mathcal{CV}_{n-1}(m) \ar@{^(->}[r] & 
\mathcal{LV}_n(m) / \mathcal{LV}_{n-1}(m) \ar@{=}[r] & \mathcal{LV}_n(m) / \mathcal{LV}_{n-1}(m) \ar^-\cong[u] \ar@{->>}[r] & \mathcal{BV}_n(m) / \mathcal{BV}_{n-1}(m) \ar^-\cong[u]
}
\end{equation}

We first define the spaces involved.
%, where we have for simplicity entirely omitted the number of strands $m$ from the notation.
\begin{itemize}
\item The spaces $\mathcal{CD}_n(m)$ and $\mathcal{BCD}_n(m)$ are respectively the spaces of chord diagrams and  braid chord diagrams with $n$ chords on $m$ strands.  The relations $YB$ are the Yang--Baxter Lie relations \eqref{E:Yang-Baxter1} and \eqref{E:Yang-Baxter2}.
\item 
The spaces 
%$\mathcal{CV}_n(m), 
$\mathcal{LV}_n(m)$ and $\mathcal{BV}_n(m)$ 
%are all spaces of finite invariants of type $n$.  They 
are respectively 
%the space of concordance invariants of $m$-component long links, 
the space of type $n$ isotopy invariants of $m$-component long links and type $n$ isotopy invariants of $m$-component (pure) braids.  
\end{itemize}
We now explain the maps. 
\begin{itemize}
\item 
The horizontal map to $(\mathcal{CD}_n/4T)^*$ is given by killing all trivalent diagrams that are not chord diagrams.  
\item The downward vertical map to $(\mathcal{BCD}_n/YB)^*$ is defined similarly, and it is an isomorphism by \refT{BraidCohIsCDMod4T}.
\item The map $(\mathcal{CD}_n/4T)^* \twoheadrightarrow (\mathcal{BCD}_n/YB)^*$ kills all chord diagrams that are not braid diagrams.  Commutativity of the top square is then clear.
\item
The bottom row is induced by dualizing the inclusion of (isotopy classes of) pure braids into (isotopy classes of) long links.
%sequence of maps
%\[
%\xymatrix{
%\{ \text{pure braids}\}/\text{isotopy } \ar@{^(->}[r] &
%\{ \text{long links}\}/\text{isotopy } \ar@{->>}[r] &
%\{ \text{long links}\}/\text{concordance}
%}
%\]
\item
The upward vertical maps from the bottom row are the canonical maps (from finite type theory) from invariants to weight systems. Here braid invariants map to weight systems in $(\BCD(m))^*$ because of Artin's theorem \cite{Artin:Braids} that an isotopy of braids in the space of long links can be taken to be an isotopy in the subspace of braids.    Commutativity of the lower right square is now clear.  
\item
The lower-left downward map $\int: \Ho^{0,n}(\LD(m)) \to \mathcal{LV}_n / \mathcal{LV}_{n-1}$ is the composition 
 of the map $\Ho^{0,n}(\LDf(m)) \to \mathcal{LV}_n$ induced by the Kontsevich integral, followed by the projection  $\mathcal{LV}_n \to \mathcal{LV}_n / \mathcal{LV}_{n-1}$.  It is inverse to the canonical map, i.e.~the bottom 
%middle 
left square commutes.  
%The bottom left square obviously commutes.
\end{itemize}
To complete the proof, we apply results of Habegger and Masbaum.  They showed that the tree part (or perhaps more accurately, the ``forest part'') of the Kontsevich integral produces all finite-type concordance invariants \cite[Theorem 15.1]{HabMas-KontsevichMilnor}.  
The map they studied is closely related to the dual of the left-hand vertical composition.  
%, where the map $\int$ is given by the Kontsevich integral.
%showed that the forest part of the Kontsevich integral 
%%(at least with an appropriate choice of Drinfeld associator) 
%is a universal finite type concordance invariant \cite[Theorem 15.1]{HabMas-KontsevichMilnor}.  
%Their result immediately implies that, if one takes the integration map to be the Kontsevich integral, the restriction of $\int$ to the subspace of all forests in $\Ho^{0,n}(\LD(m))$ maps isomorphically onto the concordance invariants $\mathcal{CV}_n / \mathcal{CV}_{n-1}$.  
The only difference is that because of condition (2) in Definition \ref{D:LDfDef}, our subspace $\LDf(m)$ excludes some forests where a component tree has all its leaves on the same segment.
%it is not difficult to see via the Hopf algebra structure on $\LD(m)$ 
Nonetheless, one can show (using the shuffle product and cocyclicity) that $\Ho^{0,*}(\LDf(m))$ maps onto the quotient of the space of finite-type concordance invariants by the ideal generated by nontrivial \emph{knot} invariants.  This suffices for our purposes, since such knot invariants vanish in $\mathcal{BV}_n(m)$.
%that the difference between the subspace of all forest diagrams and $\LDf(m)$ corresponds to finite type concordance invariants of \emph{knots}.  
Habegger and Masbaum's work then implies that the composition $\Ho^{0,n}(\LDf(m)) \to \mathcal{BV}_n(m)/\mathcal{BV}_{n-1}(m)$ is surjective.  Indeed, in the proof of Theorem 16.1 in \cite{HabMas-KontsevichMilnor}, they show that the dual map from the corresponding filtration quotient of braids to the space of order-$n$ tree (i.e.~forest) diagrams is injective.  The key point there is that the dual of $\mathcal{BV}_n(m)/\mathcal{BV}_{n-1}(m)$ can be identified with a quotient in a filtration by Milnor concordance invariants, and again, these are precisely the invariants coming from forest diagrams.  
%Thus we obtain a surjection onto the space of concordance invariants of long links modulo concordance invariants of long knots, by composing the left-hand vertical map with the appropriate projection.
%In particular, it surjects onto the subspace of concordance invariants of pure braids.
%
%Finally, surjectivity of left-hand vertical map followed by the bottom row comes from the above conclusion plus the fact that for braids, the relation of concordance coincides with the relation of isotopy.  (See Le Dimet's paper \cite[Section 3]{LeDimet-1988} for a proof of that fact or the paper of Kirk, Livingston, and Wang \cite[Section 5]{Kirk-Livingston-Wang} for a summary of that proof.)
\end{proof}

\begin{rems} \ 
\begin{enumerate}
\item
The composition $\Ho^{0,n}(\LDf(m)) \to \mathcal{BV}_n(m)/\mathcal{BV}_{n-1}(m)$ is not injective.  In fact, Habegger and Masbaum showed that $\Ho^{0,n}(\LDf(m)) \to \LV_n(m)/\LV_{n-1}(m)$ maps isomorphically onto the corresponding filtration quotient of finite type \emph{concordance} invariants of string links.  So for example, for $n=3$, $m=2$, its image contains the Milnor invariant $\mu_{1122}$, which detects the Whitehead link, but which must map to 0 in $\mathcal{BV}_3(2)/\mathcal{BV}_{2}(2) = 0$.
For general $m$ and $n$, the dimension of this kernel can be found by comparing Willerton's work on Vassiliev invariants of pure braids \cite[Chapter 4]{Willerton:PhD} to Orr's result on Milnor concordance invariants \cite[Theorem 15]{Orr-Invar}.
\item
The surjectivity above is reminiscent of the fact that concordance coincides with isotopy for braids; see \cite[Section 3]{LeDimet-1988} or \cite[Section 5]{Kirk-Livingston-Wang}.  However, we need the surjectivity at the level of filtration quotients, rather than at the level of all braid invariants.
\end{enumerate}
\end{rems}

We next characterize the map $\phi$ on certain cocycles in $\LDf(m)\subset \LD(m)$.  Recall that a link invariant is a \emph{homotopy invariant} if its value is unaffected by changing a crossing involving any single strand.   Milnor's link invariants are well defined integers for long links, and a certain subset of them are homotopy invariants \cite{Milnor-Mu, HabLin-Classif}.  They are also finite type invariants, so there are cocycles in $\LD(m)$ which map to them under the Bott--Taubes integral.  The cocycles corresponding to Milnor's homotopy invariants lie in the subcomplex $\HD(m) \subset \LDf(m)$ of forests where the leaves of each tree lie on distinct segments \cite{KMV:FTHoLinks, KV:MilnorHoLinks}.\footnote{Furthermore, finite type invariants separate long links up to link homotopy \cite{BN:HoLink} and Bott--Taubes integrals yield such invariants \cite{KMV:FTHoLinks}, so the image of the Bott--Taubes integration map $\Ho^{0,*}(\HD(m)) \to \Ho^{0}(\Link(\coprod_m \R, \R^3))$ is large enough to distinguish elements of $\Ho_{0}(\Link(\coprod_m \R, \R^3))$.}

\begin{prop}
\label{P:HtpyDiagramsAreGood}
Let $\gamma \in \HD(m)$ be a cocycle associated to a Milnor link-homotopy invariant.
Then all of the diagrams appearing in $\gamma$ are good diagrams, and the diagrams in $\phi(\gamma)$ are in bijection with those in $\gamma$.
\end{prop}

\begin{proof}
Call a diagram in $\LD(m)$ \emph{arc-connected} if every pair of vertices can be joined by a path of edges and arcs.  
It was shown in \cite{KV:MilnorHoLinks} that the cocycle corresponding to a Milnor homotopy invariant via the Bott--Taubes configuration space integral map is a linear combination of trivalent trees in $\HD(m)$, plus some linear combination of diagrams in $\HD(m)$ with fewer free vertices coming from applying all possible STU relations.  
A trivalent tree is clearly a good, arc-connected diagram.  In each STU relation, with the tree as S, the diagrams T and U are not only arc-connected but also good, using the fact that the leaves of the tree lie on distinct segments.  Moreover, each component of T and U is a tree with leaves on distinct segments.  
Proceeding by downward induction on the number of free vertices in a diagram appearing in $\gamma$, the same argument shows that all the diagrams in $\gamma$ are good and arc-connected.  The goodness implies that none are sent to zero by $\phi$.  The arc-connectedness implies that the partial order on the components induced by the orders of the segment vertices is a total order, so there is only one shuffle of the components respecting this order.
\end{proof}

\begin{rem}
This result of course generalizes to linear combinations of Milnor invariants, but it does not quite hold for cocycles associated to arbitrary link-homotopy invariants.  Nonetheless, it fails in a relatively mild way.  In fact, any link-homotopy invariant is a polynomial in Milnor invariants.  The Bott--Taubes configuration space integral map from $\LD(m)$ takes the shuffle product to the product of invariants.  Thus if $\gamma \in \LD(m)$ corresponds to a link-homotopy invariant, each diagram in $\gamma$ is a shuffle product of good diagrams.  The map $\phi$ simply kills all the shuffles which are order-violating diagrams.
\end{rem}

\subsection{Relating formality, Chen's integrals, and configuration space integrals}\label{S:Integrals}

%%%%%%%%%%%%%%%%%%%%%%%%%%%%%%%%%%%%%%%%%%%%%%%%%%%%%%%%%%%%%%%%%%%%%%%%%%%%%%%%%%%%%%%%%%%%%%%%%%%%%%%%
With the map $\phi$ of diagram complexes in hand, we can now finish the proof of our main result, stated in the Introduction and recalled below for convenience.   
It remains only to check the compatibility of the two types of integrals from a diagram complex to the cohomology of an embedding space.  

\begin{thm}\label{T:MainDiagramThm}
Let $n\geq 3$.  Via the map $\phi$ of diagram complexes, Bott--Taubes integration on forest diagrams in $\LD(m)$ agrees with the composition of the formality integration map $I$ followed by Chen's iterated integrals on $\Ba(\D(m))$.  That is, we have the following commutative diagram:
\begin{equation}
\label{E:MainCommutativeSquareRep}
\xymatrix{
\Ho^{0,k}(\LD_\mathrm{f}(m)) \ar^-{\varint_{\mathrm{Bott-Taubes}}}[rr] \ar@{->>}_-\phi[d] & & \Ho^{k(n-2)}(\Emb(\coprod_m \R, \R^{n+1}) \ar^-{\iota^*}[d] \\
\Ho^{0,k}(\Ba(\D(m))) \ar^-{\varint_{\mathrm{Chen}} \circ \varint_{\mathrm{formality}}}_-{\cong}[rr] & & \Ho^{k(n-2)}(\Omega \Conf(m, \R^n))
}
\end{equation}
\end{thm}

By \refT{BraidCohIsCDMod4T}, the cohomology of the space of braids is nontrivial only in degrees that are multiples of $(n-2)$, so for various $k$ the classes above capture all the cohomology.

\begin{proof}
The surjectivity of the map $\phi$ was established in \refT{DiagramsSurjOnCohom}, while the bottom row is an isomorphism by \refT{DiagramsBraidCohomology} and \refT{Chen}.  So it remains to prove that the square commutes.

We first set some notation.  
Let $\Gamma \in \LD^{0,n}$ be a link diagram of defect 0 and order $n$.  
Let $\Conf_{\mathcal{L}}$ be the total space of the bundle associated to the link diagram $\Gamma$, as in  \eqref{E:BT-Square}, but with the numbers of vertices suppressed from the notation.  Let $F_{\mathcal{L}}$ denote its fiber.
Let $\Conf_B$ denote the total space of the bundle over which the composite of integration maps $\int_{\mathrm{Chen}} \circ \int_{\mathrm{formality}}$ is performed, and let $F_B$ denote its fiber.
A point in $F_B$ consists of some number of configuration points in $\R^n$ and a point in a simplex $\Delta^p$.  
A point in $\Conf_B$ consists of this data plus a braid.
Let $\alpha_{\Gamma}$ be the form on $\Conf_{\mathcal{L}}$ associated to the diagram $\Gamma$.
Let $\alpha_{\phi(\Gamma)}$ be the form on $\Conf_B$ associated to the braid diagram $\phi(\Gamma)$.  
Establishing commutativity of the square then amounts to showing that 
$\iota^*\int_{F_{\mathcal{L}}} \alpha_\Gamma = \int_{F_B} \alpha_{\phi(\Gamma)}$.
(The map  $\iota^*$ was defined in \refS{BraidsAsLinks}.)

The main idea is to consider forms on $S^n$ supported near a time slice $\R^n \subset \R^{n+1}$ perpendicular to the flow of the braid, and to use such forms to define $\alpha_\Gamma$ in the Bott--Taubes integral.  We first address the independence of $\int_{F_{\mathcal{L}}} \alpha_\Gamma$ from this choice of form.
In \cite[Theorem 4.36]{KMV:FTHoLinks}, this independence was proven for $n+1\geq 5$.  To cover the case of ambient dimension 4, we must recall that proof.  It is clear that different volume forms on $S^n$ yield cohomologous forms for $\alpha_\Gamma$, so suppose $\alpha_1$ and $\alpha_2$ are two cohomologous forms to be fiberwise integrated.  Let $\beta$ be a form with $d \beta = \alpha_1 - \alpha_2$.  By Stokes' Theorem, $\int_{F_{\mathcal{L}}} \alpha_1$ is cohomologous to $\int_{F_{\mathcal{L}}} \alpha_2$ precisely if $\int_{\partial F_{\mathcal{L}}} \beta=0$.  This latter condition is ensured by partitioning the contributions from various faces constituting $\partial F_{\mathcal{L}}$ and using cancellation, symmetry, and degeneracy arguments.  
The cancellation, symmetry, and some of the degeneracy arguments are independent of the ambient dimension.  
For the degeneracy arguments, since $\beta$ is a primitive of $\alpha_1 -\alpha_2$, one must show that the image of the appropriate face in $\Conf_{\mathcal{L}}$ in $\prod S^n$ has codimension at least 2.  
The only faces for which the argument depends on $n$ are those involving a collision of vertices in a subgraph $\Gamma' \subset \Gamma$ where all free vertices have valence at least 3 and all segment vertices have valence at least 1.  
But any diagram $\Gamma \in \LDf^{0,*}(m)$ consists of trivalent trees, so such a subgraph $\Gamma'$ would have to be a union of components of $\Gamma$.  A collision of all the vertices of $\Gamma'$ is then impossible, by condition (2) in \refD{LDfDef} of the subspace $\LDf^{0,*}(m)\subset \LD^{0,*}(m)$.

Having noted that the Bott--Taubes integral is independent of the choice of volume form on $S^n$, we take a sequence of antipodally symmetric forms $(\omega_\ell)_{\ell=1}^\infty$ on $S^n$ with support in increasingly small $\delta$-neighborhoods of $S^{n-1}$, the equator in the hyperplane perpendicular to the flow of the braid.  
Let $\alpha_{\Gamma, \ell}$ be the form on $\Conf_{\mathcal{L}}$ associated to the diagram $\Gamma$ and the form $\omega_\ell$ on $S^n$.  
By the independence of the integral from the choice of form and thus $\ell$, we may write $\int_{F_{\mathcal{L}}} \alpha_\Gamma$ to denote  $\int_{F_{\mathcal{L}}} \alpha_{\Gamma, \ell}$ for any $\ell$ without ambiguity.
We let $\omega$ be the
%Dirac
form on the equator $S^{n-1}$ such that $\int_{S^n} \omega_\ell \to \int_{S^{n-1}} \omega$, and we use this form to define $\alpha_{\phi(\Gamma)}$.  

We will now check commutativity by dividing the defect zero forest diagrams in $\LD(m)$ into good diagrams and bad diagrams.

\textbf{Good diagrams}:  For a good diagram $\Gamma$, the two integrals $\iota^*\int_{F_{\mathcal{L}}} \alpha_\Gamma$ and $\int_{F_B} \alpha_{\phi(\Gamma)}$ differ as follows:
\begin{itemize}
\item[(a)] The fiber $F_{\mathcal{L}}$ has configurations where vertices on the same component do not lie in a time-slice, i.e.~a hyperplane $\{t\} \x \R^{n}$, whereas in the fiber $F_B$, vertices on the same component are always cohyperplanar.
\item[(b)] 
The form $\alpha_\Gamma$ is pulled back from a product of $S^n$'s, whereas the form $\alpha_{\phi(\Gamma)}$ is pulled back from a product of $S^{n-1}$'s.
\end{itemize}
Using the sequence of forms $\omega_\ell$ on $S^n$, we see that the contributions to $\iota^*\int_{F_{\mathcal{L}}} \alpha_\Gamma$ from non-cohyperplanar configurations are arbitrarily small, and the remaining contributions are essentially measured by a form on $S^{n-1}$.  So while the sequence $\omega_\ell$ does not strictly speaking converge to a form, we claim that the resulting sequence of integrals does converge:
\[
\xymatrix{
\iota^* \displaystyle \int_{F_{\mathcal{L}}} \alpha_{\Gamma, \ell} \ar^-{\ell \to \infty}[rr] &&
\displaystyle \int_{F_B} \alpha_{\phi(\Gamma)}
}
\]  
The independence of the left-hand integral from the choice of form $\omega_\ell$ then implies that these integrals are constant as a function of $\ell$, establishing the commutativity of the diagram in this case.  
Modulo contributions from configurations with points outside of $I \x \R^n$, the convergence is not difficult to see. The point is that $F_B$ embeds into $F_{\mathcal{L}}$ as the subspace where configurations on a single component are cohyperplanar; this embedding comes from a map $\Delta^p \to \Conf((i_1,\dots, i_m), \coprod_1^m \R)$ defined via the appropriate diagonal maps.\footnote{Note that we need the shuffles in the definition of $\phi$ to attain all the possible configurations in $\Conf_{\mathcal{L}}$ with each component lying in a hyperplane.}

We finally address contributions from configurations with some points outside of $I \x \R^n$.  First, we can extend a loop $I \to \Conf(m, \R^n)$ to a map $\R \to \Conf(m,\R^n)$ that is constant outside $I$.  Then the Chen integral agrees with the integral over $p$-tuples $t_1 \leq ... \leq t_p$ where the $t_i$ are not constrained to lie in $[0,1]$; indeed, for $p$-tuples oustide of $\Delta^p$, the evaluation map factors through a positive-codimension subspace.  
Finally, while we technically need the link strands to approach infinity in different directions to make sense of the configuration space integrals, we take a path parametrized by say $s\in [a,b]$ from such an embedding of $\coprod_m [1, \infty)$ to the one with parallel strands, i.e.~given by the $m$ maps $t \mapsto (t, x_i)$.  We enlarge the space of long links to include all the behaviors at infinity parametrized by $s \in[a,b)$.  On this space, the configuration space integrals can be defined.
The integral at time $b$ is well defined only for cohyperplanar configurations, in which case we get the Chen integral, but this is the limit as $s \to b$ of the configuration space integrals.  This completes the check of commutativity for good diagrams.

(The reader may consider the above argument in the example of the cohomology class corresponding to a pairwise linking number, i.e.~a diagram $\Gamma_{ij}$ with a single chord.  In its configuration space integral, one integrates an $n$-form over $(t_1,t_2)\in \R^2$.  By taking only contributions near the equator, the integrals approach the integral of an $(n-1)$-form over the diagonal subspace $(t,t) \in \R^2$.)

\textbf{Bad diagrams}:  
Since $\LDf(m)$ consists of forests, a bad diagram in $\LDf(m)$ must be order-violating.  
For these bad diagrams, it suffices to show the vanishing of $\iota^*\int_{F_{\mathcal{L}}} \alpha_\Gamma$.  
Let $\Gamma$ be an order-violating diagram with components $(c_1 \to \cdots \to c_j \to c_1)$ that cannot be ordered.  Recall that every arrow corresponds to an arc between a pair of adjacent segment vertices.
Let $\mathfrak{S}$ be the subspace of $\Conf_B$ where each such pair of segment vertices has collided, meaning that in $(t_1 \leq ... \leq t_p) \in \Delta^p$, some of the $t_i$  are equal.  
Note again that $\Conf_B$ embeds into $\Conf_{\mathcal{L}}$, using appropriate diagonals to map $\Delta^p \to \Conf((i_1, \dots, i_m), \coprod_1^m \R)$.
Considering the same sequence of forms $\omega_\ell$ on $S^{n}$ with support approaching the equator, we obtain zero contribution to $\int_{F_{\mathcal{L}}} \alpha_\Gamma$ outside of any neighborhood of the image of $\mathfrak{S}$ in $\Conf_{\mathcal{L}}$.  But $\mathfrak{S}$ has positive codimension in each fiber of $\Conf_B$ and hence in each fiber of $\Conf_{\mathcal{L}}$, so we conclude that $\iota^*\int_{F_{\mathcal{L}}} \alpha_\Gamma=0$.  
%\textbf{Case (3)}:  Let $\Gamma$ be a bad diagram containing a closed loop of edges.  
%For simplicity, first suppose $\Gamma$ is connected.
%Note that  the form $\iota^*\int \Gamma$ has degree
%\[
%d_1 := n|E(\Gamma)| - (n+1)|V_{\mathrm{free}}(\Gamma)| - |V_{\mathrm{seg}}(\Gamma)|
%\]
%Consider the limit of the integrals using the forms $\omega_\ell$.  On the one hand, each term is the integral $\iota^*\int\Gamma$ along the fiber of the bundle $\Conf_{\mathcal{L}}$.  On the other hand, as noted in Case (1), the limit is the fiberwise integral of a form of degree $(n-1)|E(\Gamma)|$ on the bundle $\Conf_B$, where the fiber dimension is $|V_{\mathrm{free}}(\Gamma)|+1$. The degree of this limit form is thus
%\[
%d_2:= (n-1)|E(\Gamma)| - |V_{\mathrm{free}}(\Gamma)| - 1.
%\]
%The difference $d_2 - d_1$ is precisely $\chi(\Gamma)-1$, that is, one less than the Euler characteristic of $\Gamma$.  So if $\Gamma$ is not a tree, $\iota^* \int \Gamma$ is arbitrarily close to a form of strictly smaller degree, hence must be zero.
%
%Now if $\Gamma$ has multiple components $\Gamma_i$, each degree $d_1$ and $d_2$ can be calculated as a sum of the degrees over the components.  Thus $d_2 - d_1 = \sum (\chi(\Gamma_i)-1)$, and this quantity is less than zero if (and only if) any $\Gamma_i$ is not a tree.  Hence the same argument as for connected $\Gamma$ applies.
\end{proof}

\begin{rem}
We can show that $\iota^*\int_{F_{\mathcal{L}}} \alpha_\Gamma=0$ for any diagram $\Gamma$ that is not a forest.  Indeed, we can reason as above that this form is approximated arbitrarily well by a form that is an integral over $\Conf_B$.   If the  Euler characteristic of a component of $\Gamma$ is $\leq 0$, then the latter form has degree strictly lower than the degree of the former, showing that it must be zero.
However, it is not immediately clear whether the configuration space integrals for diagrams in $\LD$ of defect $>0$ vanish for braids.
\end{rem}

\begin{cor}
\label{C:LinksSurjectBraids}
The inclusion of the space of braids into the space of long links in $\R^{n+1}$ induces a surjection in cohomology 
with real coefficients $\Ho^*(\mathcal{L}_m^{n+1}) \twoheadrightarrow \Ho^*(\Omega \Conf(m, \R^n))$ for all $n\geq 3$.
\end{cor}
\begin{proof}
This follows from the square \eqref{E:MainCommutativeSquareRep}, in particular from the isomorphism in the bottom row, the surjectivity of the left vertical map, and the commutativity of the diagram.
\end{proof}

The above statement mirrors the situation for classical links and braids.  See Remark \ref{R:ClassicalSurj} below.

\bigskip
\bigskip
\bigskip

%%%%%%%%%%%%%%%%%%%%%%%%%%%%%%%%%%%%%%%%%%%%%%%%%%%%%%%%%%%%%%%%%%%%%%%%%%%%%%%%%%%%%%%%%%%%%%%%%%%%%%%%

\section{The case of braids in $\R^3$}\label{S:R^3BraidsCohomology}

%%%%%%%%%%%%%%%%%%%%%%%%%%%%%%%%%%%%%%%%%%%%%%%%%%%%%%%%%%%%%%%%%%%%%%%%%%%%%%%%%%%%%%%%%%%%%%%%%%%%%%%%

In the classical case $n=2$, our main theorem essentially holds with one modification: the cohomology of the bar complex no longer maps isomorphically onto the cohomology of $m$-strand pure braids.  
Indeed, we can see that for even $m=2$, the map cannot be an isomorphism.  In that case, $\Ho^*(\Ba(\D(2)))$ is an infinite direct sum (with the sum taken over different numbers of chords), while $\Ho^0(\Omega \Conf(2,\R^2)) \cong \Ho^0(\Omega S^1)$ is dual to an infinite direct sum (with the sum taken over different winding numbers), hence an infinite direct product.  
As explained below, this is in some sense the only obstacle to extending the full theorem for $n \geq 3$ to $n=2$.
Recall the definition of finite invariants of pure braids of type $k$, as in \refS{IntroChenFormality}.

\begin{thm}
\label{T:Classical}
Let $n=2$.  For an appropriate choice of form on $S^2$, there is a commutative square   
\[
\xymatrix{
\Ho^{0,k}(\LD_\mathrm{f}(m)) \ar^-{\varint_{\mathrm{Bott-Taubes}}}[rr] \ar@{->>}_-\phi[d] & & \Ho^{0}(\Emb(\coprod_m \R, \R^{3}) \ar^-{\iota^*}[d] \\
\Ho^{0,k}(\Ba(\D(m))) \ar^-{\varint_{\mathrm{Chen}} \circ \varint_{\mathrm{formality}}}[rr] & & \Ho^{0}(\Omega \Conf(m, \R^2))
}
\]
where the bottom horizontal arrow maps isomorphically onto the filtration quotient $V_k(P_m)/V_{k-1}(P_m)$ of finite type invariants of pure braids of type $k$.
\end{thm}

\begin{rems}
\label{R:ClassicalSurj} \ 
\begin{enumerate}
\item
Although this theorem is slightly weaker than the version for $n\geq 3$, it is not much weaker.
First, since $\Conf(m,\R^2)$ is a $K(\pi,1)$ space \cite{FN:ConfFibration}, the only nontrivial cohomology of $\Omega \Conf(m,\R^2)$ is in degree zero.  Second, while the maps $\Ho^{0,k}(\Ba(\D(m))) \to \Ho^{0}(\Omega \Conf(m, \R^2))$ for various $k$ are not isomorphisms, they are strong enough to separate pure braids.  For example, \cite{Kohno:VassilievBraids} and \cite{MostovoyWillerton} contain proofs of the fact that finite type invariants separate pure braids, using powers of the augmentation ideal of the pure braid group.  
\item
The inclusion $j: \Emb(\coprod_m \R, \R^{3}) \to \Link(\coprod_m \R, \R^{3})$ from embedded links to link maps induces an injection $j^*$ on $\Ho^0$, from homotopy invariants to isotopy invariants of long links.  Since isotopy classes of pure braids surject onto homotopy classes of string links \cite{Goldsmith:1973}, the composition $\iota^* \circ j^*$ is injective, and we may view homotopy invariants as a subset of $\Ho^{0}(\Omega \Conf(m, \R^2))$.  The previous remark then implies that the image of (either) composition in the above square separates long links up to homotopy, as also guaranteed by \refP{HtpyDiagramsAreGood} and the fact that Bott--Taubes integrals restricted to $\HD(m) \subset \LDf(m)$ separate such homotopy classes \cite{KMV:FTHoLinks, BN:HoLink}.
\item
While either composition in the square above is not a surjection in cohomology, the analogue of \refC{LinksSurjectBraids} holds, that is, the map $\iota^*$ is still a surjection in this classical case.  In fact, Artin's theorem that braids isotopic as long links are isotopic as braids \cite{Artin:Braids} implies that the map $\iota_*$ is injective on $\Ho_0$ and thus $\iota^*$ is surjective on $\Ho^0$.  (It is trivially a surjection in higher degrees.)
\end{enumerate}
\end{rems}

\begin{proof}[Proof of \refT{Classical}]
First note that the map $\phi$ and its surjectivity in cohomology are independent of $n$, and the map $\iota$ can be defined for any $n\geq 2$.  Next, we address the top horizontal map.  In general, for $n+1=3$, configuration space integrals require the addition of an anomaly term for arbitrary link invariants from $\LD(m)$.  However, this anomaly term is only required to cancel contributions from boundary faces 
where the colliding vertices correspond precisely to a connected component of a diagram $\Gamma$.  Condition (2) in \refD{LDfDef} rules out such a possibility for $\Gamma \in \LDf(m)$.  So the anomaly correction is not needed for these diagrams, and the top horizontal map is well defined.  

The formality integral $I: \D(m) \to \Ch_{dR}^*(\Conf(m,\R^2))$ is still a quasi-isomorphism for $n=2$ (though for showing formality in this case, one can embed the cochains directly into cohomology \cite{Arnold:Formality}).  As in \refT{BarFormal}, this passes to a quasi-isomorphism of bar complexes.  
(Computing the cohomology of the bar complex for $n=2$ then proceeds as for $n\geq 3$, yielding a similar result, but concentrated in total degree 0.)

For the Chen integral map, first recall from Remarks \ref{R:DefectZeroRems} that defect $0$ and order $k$ correspond to the main diagonal $\Ba^{-p,p}(\D(m))$ with $p \leq k$.
For a space that is not simply connected, such as $\Conf(m,\R^2)$, the Chen integral does not compute the whole loop space cohomology, but it does tell us about $\Ho^0$.  Let $\mathcal{F}^k\Ho^0(\Ba(\Ch_{dR}^*\Conf(m,\R^2)))$ denote the cohomology in total degree 0 represented by  $\Ba^{-p,p}(\Ch_{dR}^*\Conf(m,\R^2))$ with $p \leq k$.  
Let $J$ be the augmentation ideal in the group ring $\R[\pi_1(\Conf(m,\R^2))] = \R[P_m]$.
By \cite{Chen:Pi1}, we know that the Chen integral gives an isomorphism
\[
\mathcal{F}^k\Ho^0(\Ba(\Ch_{dR}^*\Conf(m,\R^2))) \stackrel{\cong}{\longrightarrow} \Hom(J / J^{p+1}, \R)
\]
The right side above is $V_k(P_m)/ \R$, the space of type $k$ finite type invariants modulo constant functions.

Finally, we address the commutativity of the diagram.  The proof proceeds as for the case $n\geq 3$, using a limit of forms with support in smaller and smaller neighborhoods of the equatorial $S^{1} \subset S^2$ perpendicular to the flow of the braid.  
The key question is again whether the Bott--Taubes integral depends on the choice of form on $S^2$.  As in the proof of \refT{MainDiagramThm}, the question reduces to whether $\int_{\partial F} \beta=0$, where $\beta$ is a primitive for the difference $\alpha_2 - \alpha_1$ of two forms to be integrated.  
%(See \cite[Theorem 4.36]{KMV:FTHoLinks}, \cite[Section A.4]{CCRL}, or \cite[Section 4.2]{Thurs}.)  
Again, the cancellation and symmetry arguments are independent of $n$.  For the degeneracy arguments, one can show that the images of all the relevant faces have codimension 2, \emph{except} for the case where a component of $\Gamma$ collides with $\infty$.  
(For example, if we take the tripod on 3 strands, the $\alpha_i$ are 6-forms, and the face where all the vertices approach $\infty$ has dimension 5, so there is no reason for the integral of the 5-form $\beta$ to vanish.)  
However, if we assume that the behavior of our long links towards infinity is affine-linear in the $xy$-plane, we can guarantee the vanishing along such faces by choosing a form that vanishes in a neighborhood of the equator $S^1$ in the $xy$-plane (see \cite[Proposition 4.31]{KMV:FTHoLinks} for details). 
\end{proof}

%%%%%%%%%%%%%%%%%%%%%%%%%%%%%%%%%%%%%%%%%%%%%%%%%%%%%%%%%%%%%%%%%%%%%%%%%%%%%%%%%%%%%%%%%%%%%%%%%%%%%%%%

%%%%%%%%%%%%%%%%%%%%%%%%%%%%%%%%%%%%%%%%%%%%%%%%%%%%%%%%%%%%%%%%%%%%%%%%%%%%%%%%%%%%%%%%%%%%%%%%%%%%%%%%

%%%%%%%%%%%%%%%%%%%%%%%%%%%%%%%%%%%%%%%%%%%%%%%%%%%%%%%%%%%%%%%%%%%%%%%%%%%%%%%%%%%%%%%%%%%%%%%%%%%%%%%%

\appendix
\section{Orientations of diagrams}

In mapping between link diagrams and braid diagrams, we need to convert between an orientation of ``odd type'' and an orientation of ``even type,'' since a loop of configurations of points in $\R^n$ is a long link in $\R^{n+1}$.  For braid diagrams $\Gamma$, an orientation of odd type consists of an ordering of the vertices and an orientation of every edge, whereas an orientation of even type consists of an ordering of the edges.  The same is true for link diagrams, except that an orientation of even type consists of an ordering of the edges and an ordering of the segment vertices.  
We used the following key fact in \refD{LinksToBraidsMap}.  

\begin{prop}
Suppose that $\Gamma$ is a connected link diagram whose component is a tree, i.e.~there are no closed loops of edges.
Then an orientation of odd type on $\Gamma$ corresponds canonically to an orientation of even type on $\Gamma$ and vice-versa.
\end{prop}
\begin{proof}
This fact follows from the statements in Table 1 in Section 3.1 of the unpublished manuscript of Kuperberg and Thurston \cite{Kuperberg-Thurston}, where various types of orientations are viewed as parity functors.  
The proofs of those statements are essentially given in the Appendix of D.~Thurston's undergraduate thesis \cite{Thurs} by considering orientations on the vector spaces in an exact sequence.
The key points in this special case of a tree $\Gamma$ are that $\Ho_1(\Gamma)=0$, and one can choose an orientation for the one-dimensional space $\Ho_0(\Gamma)$ by viewing it as formal $\R$-multiples of the single component of $\Gamma$.  
We will however give a slightly more explicit correspondence.

Suppose $\Gamma$ is a tree.  Choose any vertex to be the root, and consider an odd type labeling of $\Gamma$ that is \emph{flow-preserving}.  By this we mean a labeling where the root is labeled 1, and all edges are oriented away from the root, and every edge points from lower to higher vertex labels.   This determines an even type orientation by taking the $i$th edge to be the one which is oriented towards vertex $i+1$.  

In the other direction, suppose again that a vertex of the tree $\Gamma$ is designated as the root, and consider an even type labeling of $\Gamma$ that is \emph{flow-preserving}.  By this we mean a labeling where in any path the edge-labels increase moving away from the root.  Then for the edge with the first label, label the endpoints 1,2, and orient this edge to point from 1 to 2.  (Either choice gives the same odd type orientation.)  This determines an odd type orientation by (inductively) labeling the (so far) unlabeled endpoint of edge $i$ with the label $i+1$.  

Of course not every labeling is flow-preserving, but by definition every labeling $\sigma$ is equivalent to either $\tau$ or $-\tau$ for some flow-preserving labeling $\tau$.  The fact that the composition of the two correspondences on flow-preserving labelings is the identity shows that the map is well defined and an isomorphism of parity functors.
\end{proof}

\bibliographystyle{amsplain}

%\bibliography{/Users/ivolic/Desktop/Math/Bibliography}
\bibliography{Bibliography}

\def\cprime{$'$} \def\cprime{$'$}
\providecommand{\bysame}{\leavevmode\hbox to3em{\hrulefill}\thinspace}
\providecommand{\MR}{\relax\ifhmode\unskip\space\fi MR }
\providecommand{\MRhref}[2]{%
  \href{http://www.ams.org/mathscinet-getitem?mr=#1}{#2}
}
\providecommand{\href}[2]{#2}
\begin{thebibliography}{10}

\bibitem{Adams:CobarPNAS}
J.~F. Adams, \emph{On the cobar construction}, Proc. Nat. Acad. Sci. U.S.A.
  \textbf{42} (1956), 409--412.

\bibitem{Adams:Cobar}
\bysame, \emph{On the cobar construction}, Colloque de topologie alg\'ebrique,
  {L}ouvain, 1956, Georges Thone, Li\`ege; Masson \& Cie, Paris, 1957,
  pp.~81--87.

\bibitem{Adams-Hilton}
J.~F. Adams and P.~J. Hilton, \emph{On the chain algebra of a loop space},
  Comment. Math. Helv. \textbf{30} (1956), 305--330.

\bibitem{Anderson:HSS}
D.~W. Anderson, \emph{A generalization of the {E}ilenberg-{M}oore spectral
  sequence}, Bull. Amer. Math. Soc. \textbf{78} (1972), 784--786.

\bibitem{Arnold:Formality}
V.~I. Arnol{\cprime}d, \emph{The cohomology ring of the group of dyed braids},
  Mat. Zametki \textbf{5} (1969), 227--231.

\bibitem{AT:Graph}
Greg Arone and Victor Turchin, \emph{Graph-complexes computing the rational
  homotopy of high dimensional analogues of spaces of long knots}, Annales de
  l'Institut Fourier \textbf{65} (2015), no.~1, 1--62.

\bibitem{Artin:Braids}
E.~Artin, \emph{Theory of braids}, Ann. of Math. (2) \textbf{48} (1947),
  101--126.

\bibitem{AS}
Scott Axelrod and I.~M. Singer, \emph{Chern-{S}imons perturbation theory.
  {II}}, J. Differential Geom. \textbf{39} (1994), no.~1, 173--213.

\bibitem{BN:Vass}
Dror Bar-Natan, \emph{On the {V}assiliev knot invariants}, Topology \textbf{34}
  (1995), no.~2, 423--472.

\bibitem{BN:HoLink}
\bysame, \emph{Vassiliev homotopy string link invariants}, J. Knot Theory
  Ramifications \textbf{4} (1995), no.~1, 13--32.

\bibitem{BG:FunctSpaces}
Martin Bendersky and Sam Gitler, \emph{The cohomology of certain function
  spaces}, Trans. Amer. Math. Soc. \textbf{326} (1991), no.~1, 423--440.

\bibitem{B:Koszul}
Alexander Berglund, \emph{Koszul spaces}, Trans. Amer. Math. Soc. \textbf{366}
  (2014), no.~9, 4551--4569.

\bibitem{Bezrukavnikov:Koszul-Config}
R.~Bezrukavnikov, \emph{Koszul {DG}-algebras arising from configuration
  spaces}, Geom. Funct. Anal. \textbf{4} (1994), no.~2, 119--135. \MR{1262702}

\bibitem{BottSegal}
Raoul Bott and Graeme Segal, \emph{The cohomology of the vector fields on a
  manifold}, Topology \textbf{16} (1977), no.~4, 285--298.

\bibitem{BT}
Raoul Bott and Clifford Taubes, \emph{On the self-linking of knots}, J. Math.
  Phys. \textbf{35} (1994), no.~10, 5247--5287.

\bibitem{Bott-Tu:DiffForms}
Raoul Bott and Loring~W. Tu, \emph{Differential forms in algebraic topology},
  Graduate Texts in Mathematics, vol.~82, Springer-Verlag, New York-Berlin,
  1982. \MR{658304}

\bibitem{CW:ModelConfig}
Ricardo Campos and Thomas Willwacher, \emph{A model for configuration spaces of
  points}, arXiv:1604.02043.

\bibitem{CCRL}
Alberto~S. Cattaneo, Paolo Cotta-Ramusino, and Riccardo Longoni,
  \emph{Configuration spaces and {V}assiliev classes in any dimension}, Algebr.
  Geom. Topol. \textbf{2} (2002), 949--1000 (electronic).

\bibitem{Chen:Itr-Integrals}
Kuo~Tsai Chen, \emph{Iterated integrals of differential forms and loop space
  homology}, Ann. of Math. (2) \textbf{97} (1973), 217--246.

\bibitem{Chen:Pi1}
\bysame, \emph{Iterated integrals, fundamental groups and covering spaces},
  Trans. Amer. Math. Soc. \textbf{206} (1975), 83--98.

\bibitem{Chen:Reduced_bar}
\bysame, \emph{Reduced bar constructions on de {R}ham complexes}, pp.~19--32,
  Academic Press, New York, 1976.

\bibitem{Chen:Path-Integrals}
\bysame, \emph{Iterated path integrals}, Bull. Amer. Math. Soc. \textbf{83}
  (1977), no.~5, 831--879.

\bibitem{CG}
Frederick~R. Cohen and Sam Gitler, \emph{On loop spaces of configuration
  spaces}, Trans. Amer. Math. Soc. \textbf{354} (2002), no.~5, 1705--1748
  (electronic).

\bibitem{Cohen:Homology}
Frederick~R. Cohen, Thomas~J. Lada, and J.~Peter May, \emph{The homology of
  iterated loop spaces}, Springer-Verlag, Berlin, 1976, Lecture Notes in
  Mathematics, Vol. 533.

\bibitem{Conant-Vogtmann:Morita-Vanish}
James Conant and Karen Vogtmann, \emph{Morita classes in the homology of {${\rm
  Aut}(F_n)$} vanish after one stabilization}, Groups Geom. Dyn. \textbf{2}
  (2008), no.~1, 121--138.

\bibitem{DGKMSV:Triplelinking}
Dennis DeTurck, Herman Gluck, Rafal Komendarczyk, Paul Melvin, Clayton
  Shonkwiler, and David~Shea Vela-Vick, \emph{Triple linking numbers, ambiguous
  {H}opf invariants and integral formulas for three-component links}, Mat.
  Contemp. \textbf{34} (2008), 251--283.

\bibitem{EM:Cotor}
Samuel Eilenberg and John~C. Moore, \emph{Homology and fibrations. {I}.
  {C}oalgebras, cotensor product and its derived functors}, Comment. Math.
  Helv. \textbf{40} (1966), 199--236.

\bibitem{FN:ConfFibration}
Edward Fadell and Lee Neuwirth, \emph{Configuration spaces}, Math. Scand.
  \textbf{10} (1962), 111--118.

\bibitem{FHT:RHT}
Yves F\'elix, Stephen Halperin, and Jean-Claude Thomas, \emph{Rational homotopy
  theory}, Graduate Texts in Mathematics, vol. 205, Springer-Verlag, New York,
  2001.

\bibitem{FTW:EnOperads}
Benoit Fresse, Victor Turchin, and Thomas Willwacher, \emph{The rational
  homotopy of mapping spaces of {$E_n$} operads}, arXiv:1703.06123.

\bibitem{FM}
William Fulton and Robert MacPherson, \emph{A compactification of configuration
  spaces}, Ann. of Math. (2) \textbf{139} (1994), no.~1, 183--225.

\bibitem{GTZ:ChenMapping}
Gr\'egory Ginot, Thomas Tradler, and Mahmoud Zeinalian, \emph{A {C}hen model
  for mapping spaces and the surface product}, Ann. Sci. \'Ec. Norm. Sup\'er.
  (4) \textbf{43} (2010), no.~5, 811--881.

\bibitem{Goldsmith:1973}
Deborah~L. Goldsmith, \emph{Homotopy of braids---in answer to a question of
  {E}. {A}rtin}, Topology {C}onference ({V}irginia {P}olytech. {I}nst. and
  {S}tate {U}niv., {B}lacksburg, {V}a., 1973), Springer, Berlin, 1974,
  pp.~91--96. Lecture Notes in Math., Vol. 375. \MR{0356021}

\bibitem{HabLin-Classif}
Nathan Habegger and Xiao-Song Lin, \emph{The classification of links up to
  link-homotopy}, J. Amer. Math. Soc. \textbf{3} (1990), no.~2, 389--419.

\bibitem{HabMas-KontsevichMilnor}
Nathan Habegger and Gregor Masbaum, \emph{The {K}ontsevich integral and
  {M}ilnor's invariants}, Topology \textbf{39} (2000), no.~6, 1253--1289.

\bibitem{Hain:Integrals}
Richard Hain, \emph{Iterated integrals and algebraic cycles: examples and
  prospects}, Contemporary trends in algebraic geometry and algebraic topology
  ({T}ianjin, 2000), Nankai Tracts Math., vol.~5, World Sci. Publ., River Edge,
  NJ, 2002, pp.~55--118.

\bibitem{HLTV}
Robert Hardt, Pascal Lambrechts, Victor Turchin, and Ismar Voli\'{c},
  \emph{Real homotopy theory of semi-algebraic sets}, Algebr. Geom. Topol.
  \textbf{11} (2011), 2477--2545.

\bibitem{Idrissi:ModelConfig}
Najib Idrissi, \emph{The lambrechts--stanley model of configuration spaces},
  arXiv:1608.08054.

\bibitem{Kirk-Livingston-Wang}
Paul Kirk, Charles Livingston, and Zhenghan Wang, \emph{The {G}assner
  representation for string links}, Commun. Contemp. Math. \textbf{3} (2001),
  no.~1, 87--136.

\bibitem{Kohno:Vassiliev}
Toshitake Kohno, \emph{Vassiliev invariants and de {R}ham complex on the space
  of knots}, Symplectic geometry and quantization ({S}anda and {Y}okohama,
  1993), Contemp. Math., vol. 179, Amer. Math. Soc., Providence, RI, 1994,
  pp.~123--138.

\bibitem{Kohno:VassilievBraids}
\bysame, \emph{Vassiliev invariants of braids and iterated integrals},
  Arrangements---{T}okyo 1998, Adv. Stud. Pure Math., vol.~27, Kinokuniya,
  Tokyo, 2000, pp.~157--168.

\bibitem{Kohno:LoopsFiniteType}
\bysame, \emph{Loop spaces of configuration spaces and finite type invariants},
  Invariants of knots and 3-manifolds ({K}yoto, 2001), Geom. Topol. Monogr.,
  vol.~4, Geom. Topol. Publ., Coventry, 2002, pp.~143--160 (electronic).

\bibitem{Kohno:BarComplex}
\bysame, \emph{Bar complex, configuration spaces and finite type invariants for
  braids}, Topology Appl. \textbf{157} (2010), no.~1, 2--9.

\bibitem{Kom:Helicity}
Rafal Komendarczyk, \emph{The third order helicity of magnetic fields via link
  maps}, Comm. Math. Phys. \textbf{292} (2009), no.~2, 431--456.

\bibitem{KM:Ropelength}
Rafal Komendarczyk and Andreas Michaelides, \emph{Ropelength, crossing number
  and finite type invariants of links (submitted)}, arXiv:1604.03870.

\bibitem{KV:VolpresFields}
Rafal Komendarczyk and Ismar Voli\'c, \emph{On volume-preserving vector fields
  and finite-type invariants of knots}, Ergodic Theory Dynam. Systems
  \textbf{36} (2016), no.~3, 832--859. \MR{3480346}

\bibitem{K:Vass}
Maxim Kontsevich, \emph{Vassiliev's knot invariants}, I. M. Gel\cprime fand
  Seminar, Adv. Soviet Math., vol.~16, Amer. Math. Soc., Providence, RI, 1993,
  pp.~137--150.

\bibitem{K:OMDQ}
\bysame, \emph{Operads and motives in deformation quantization}, Lett. Math.
  Phys. \textbf{48} (1999), no.~1, 35--72, Mosh\'e Flato (1937--1998).

\bibitem{KoSo}
Maxim Kontsevich and Yan Soibelman, \emph{Deformations of algebras over operads
  and the {D}eligne conjecture}, Conf\'erence Mosh\'e Flato 1999, Vol. I
  (Dijon), Math. Phys. Stud., vol.~21, Kluwer Acad. Publ., Dordrecht, 2000,
  pp.~255--307.

\bibitem{Ko:Linkhomotopy}
Ulrich Koschorke, \emph{Link homotopy}, Proc. Nat. Acad. Sci. U.S.A.
  \textbf{88} (1991), no.~1, 268--270. \MR{1084346 (92e:57036)}

\bibitem{Ko:GenMilnor}
\bysame, \emph{A generalization of {M}ilnor's {$\mu$}-invariants to
  higher-dimensional link maps}, Topology \textbf{36} (1997), no.~2, 301--324.
  \MR{MR1415590 (2000a:57063)}

\bibitem{KMV:FTHoLinks}
Robin Koytcheff, Brian~A. Munson, and Ismar Voli{\'c}, \emph{Configuration
  space integrals and the cohomology of the space of homotopy string links}, J.
  Knot Theory Ramifications \textbf{22} (2013), no.~11, 1--73.

\bibitem{KV:MilnorHoLinks}
Robin Koytcheff and Ismar Voli\'c, \emph{Milnor invariants of string links,
  trivalent trees, and configuration space integrals}, Submitted.
  arXiv:1511.02768.

\bibitem{Kuperberg-Thurston}
Greg Kuperberg and Dylan Thurston, \emph{Pertubative 3-manifold invariants by
  cut-and-paste topology}, arXiv [math/9912167].

\bibitem{LV}
Pascal Lambrechts and Ismar Voli{\'c}, \emph{Formality of the little
  {$N$}-disks operad}, Mem. Amer. Math. Soc. \textbf{230} (2014), no.~1079,
  viii+116.

\bibitem{LeDimet-1988}
Jean-Yves Le~Dimet, \emph{Cobordisme d'enlacements de disques}, M\'em. Soc.
  Math. France (N.S.) (1988), no.~32, ii+92.

\bibitem{LV:AlgOperads}
Jean-Louis Loday and Bruno Vallette, \emph{Algebraic operads}, Grundlehren der
  Mathematischen Wissenschaften [Fundamental Principles of Mathematical
  Sciences], vol. 346, Springer, Heidelberg, 2012.

\bibitem{McCleary:UserGuide}
John McCleary, \emph{A user's guide to spectral sequences}, second ed.,
  Cambridge Studies in Advanced Mathematics, vol.~58, Cambridge University
  Press, Cambridge, 2001.

\bibitem{Milnor-Mu}
John Milnor, \emph{Link groups}, Ann. of Math. (2) \textbf{59} (1954),
  177--195.

\bibitem{MostovoyWillerton}
Jacob Mostovoy and Simon Willerton, \emph{Free groups and finite-type
  invariants of pure braids}, Math. Proc. Cambridge Philos. Soc. \textbf{132}
  (2002), no.~1, 117--130.

\bibitem{MV:Cubes}
Brian~A. Munson and Ismar Voli{\'c}, \emph{Cubical homotopy theory}, New
  Mathematical Monographs, vol.~25, Cambridge University Press, Cambridge,
  2015.

\bibitem{Orr-Invar}
Kent~E. Orr, \emph{Homotopy invariants of links}, Invent. Math. \textbf{95}
  (1989), no.~2, 379--394.

\bibitem{PT:Mapping}
Fr\'ed\'eric Patras and Jean-Claude Thomas, \emph{Cochain algebras of mapping
  spaces and finite group actions}, Topology Appl. \textbf{128} (2003),
  no.~2-3, 189--207.

\bibitem{Priddy:Koszul}
Stewart~B. Priddy, \emph{Koszul resolutions}, Trans. Amer. Math. Soc.
  \textbf{152} (1970), 39--60.

\bibitem{Rector:E-MSS}
David~L. Rector, \emph{Steenrod operations in the {E}ilenberg-{M}oore spectral
  sequence}, Comment. Math. Helv. \textbf{45} (1970), 540--552.

\bibitem{S:Compact}
Dev~P. Sinha, \emph{Manifold-theoretic compactifications of configuration
  spaces}, Selecta Math. (N.S.) \textbf{10} (2004), no.~3, 391--428.

\bibitem{Spanier:Top}
Edwin~H. Spanier, \emph{Algebraic topology}, Springer-Verlag, New York, 1981.

\bibitem{Thurs}
Dylan Thurston, \emph{Integral expressions for the {V}assiliev knot
  invariants}, arXiv:math/9901110.

\bibitem{Vass:Cohom}
V.~A. Vassiliev, \emph{Cohomology of knot spaces}, Theory of singularities and
  its applications, Adv. Soviet Math., vol.~1, Amer. Math. Soc., Providence,
  RI, 1990, pp.~23--69.

\bibitem{Willerton:PhD}
Simon Willerton, \emph{On the vassiliev invariants for knots and for pure
  braids}, Ph.D. thesis, University of Edinburgh, 1997, Available at
  \texttt{http://simonwillerton.staff.shef.ac.uk/ftp/thesis.pdf}.

\end{thebibliography}
%\begin{thebibliography}{10}
%
%\end{thebibliography}

\end{document}